\documentclass[11pt,letterpaper]{amsart}
\usepackage{amssymb}
\usepackage{amsmath}
\usepackage{amsthm}
\usepackage{verbatim}
\usepackage{txfonts}
\usepackage[all,cmtip]{xy} 
\usepackage{yfonts}
\usepackage{color}

\usepackage[T1]{fontenc}
\usepackage[utf8]{inputenc}
\usepackage{mathtools}   
\usepackage{amsfonts}
\usepackage[alphabetic,initials,nobysame]{amsrefs}

\numberwithin{equation}{section}

\theoremstyle{plain}

\newcommand{\A}{\ensuremath{{\mathbb{A}}}}

\newcommand{\Z}{\ensuremath{{\mathbb{Z}}}}

\newcommand{\Q}{\ensuremath{{\mathbb{Q}}}}

\newcommand{\F}{\ensuremath{{F}}}

\newcommand{\E}{\ensuremath{{\mathbb{E}}}}

\newcommand{\OF}{O_F}

\newcommand{\Vol}{\text{Vol}}
\newcommand{\GL}{\ensuremath{{\text{GL}}}}
\newcommand{\Char}{\ensuremath{{\text{char}}}}

\newtheorem{theo}{Theorem}[section]
\newtheorem{lem}[theo]{Lemma}
\newtheorem{prop}[theo]{Proposition}
\newtheorem{cor}[theo]{Corollary}

\newtheorem{conj}[theo]{Conjecture}
\theoremstyle{remark}
\newtheorem{rem}[theo]{Remark}
\newtheorem{example}[theo]{Example}
\theoremstyle{definition}
\newtheorem{defn}[theo]{Definition}

\newtheorem*{cor*}{Corollary}
\newtheorem*{choiceoftest}{Choice of test vectors}

\newcommand{\zxz}[4]{\begin{pmatrix} #1 & #2 \\ #3 & #4 \end{pmatrix}}
\renewcommand{\Re}{\operatorname{Re}}

\newcommand{\Hom}{\operatorname{Hom}}

\newcommand{\Tr}{\text{Tr}}

\newcommand{\lb}{\left(}
\newcommand{\rb}{\right)}
\newcommand{\T}{\text{T}}
\newcommand{\RS}{\text{RS}}

\newcommand{\yh}[1]
{{\color{cyan} \sf YH:[#1]}}
\newcommand{\Ph}[1]
{{\color{green} \sf Ph:[#1]}}
\newcommand{\paul}[1]
{{\color{violet} \sf Paul:[#1]}}

\newcommand{\yhn}[1]
{#1}
\newcommand{\todo}[1]
{{\color{red} \sf Todo:[#1]}}

\newtheorem{assumption}[theo]{Assumption}

\newtheorem*{thm*}{Theorem}

\newtheorem*{lem*}{Lemma}

\theoremstyle{definition}

\newtheorem*{example*}{Example}

\newtheorem*{conj*}{Conjecture}
\newtheorem{conv}{Convention}

\theoremstyle{definition}

\newtheorem*{rem*}{Remark}

\newcommand{\resprod}{\mathop{{\prod}^{\mathbf{'}}}\limits}

\newcommand{\Eis}{\mathrm{Eis}}

\newcommand{\ram}{\mathrm{ram}}
\newcommand{\unram}{\mathrm{unram}}
\newcommand{\supp}{\mathrm{supp}}
\newcommand{\equid}{\mathrm{e}}
\newcommand{\Equid}{\mathrm{E}}
\newcommand{\Pro}{\mathrm{Pr}}

\newcommand{\Ad}{\mathrm{Ad}}

\newcommand{\Std}{\mathrm{Std}}

\newcommand{\GLd}{{\mathrm{GL}_2}}

\newcommand{\rmG}{{\mathrm{G}}}
\newcommand{\rmN}{{\mathrm{N}}}

\newcommand{\rmB}{{\mathrm{B}}}

\newcommand{\GA}{{\rmG(\Aa)}}

\newcommand{\GF}{{\rmG(F)}}

\newcommand{\Gv}{{\rmG_v}}

\newcommand{\PGL}{{\mathrm{PGL}}}

\newcommand{\ov}{{\rm nr}}
\newcommand{\tr}{{\rm tr}}
\newcommand{\mcA}{\mathcal{A}}
\newcommand{\mcB}{\mathcal{B}}

\newcommand{\mcV}{\mathcal{V}}
\newcommand{\mcE}{{\mathcal{E}}}
\newcommand{\mcS}{{\mathcal S}}
\newcommand{\mcC}{{\mathcal C}}
\newcommand{\mcW}{{\mathcal W}}

\newcommand{\Bt}{{\rmB^{\times}}}
\newcommand{\BtA}{{\rmB^{\times}(\Aa)}}

\newcommand{\eps}{\varepsilon}
\newcommand{\vphi}{\varphi}
\newcommand{\vPhi}{\varPhi}

 \newcommand{\ra}{\rightarrow}
 
\def\peter#1{\langle #1\rangle}
\def\ov#1{\overline{#1}}

 \DeclareFontFamily{OT1}{rsfs}{}
\DeclareFontShape{OT1}{rsfs}{n}{it}{<-> rsfs10}{}
\DeclareMathAlphabet{\mathscr}{OT1}{rsfs}{n}{it}

 \newcommand{\beq}{\begin{displaymath}}
\newcommand{\eeq}{\end{displaymath}}
\newcommand{\be}{\begin{equation}}
\newcommand{\ee}{\end{equation}}

\newcommand{\order}{\mathscr{O}}

\newcommand{\Oo}{O}   

\newcommand{\Cc}{{\mathbb{C}}}

\newcommand{\Nn}{{\bf{N}}}

\newcommand{\Zz}{{\mathbb{Z}}}

\newcommand{\Rr}{{\mathbb{R}}}
\newcommand{\Qq}{{\mathbb{Q}}}

\newcommand{\Qp}{\Qq_{p}}
\newcommand{\Aa}{{\mathbb{A}}}
\newcommand{\At}{{\mathbb{A}}^{\times}}
\newcommand{\Ct}{\Cc^{\times}}

\newcommand{\bash}{\backslash}
\newcommand{\Feas}{\mathrm{Feas}}

\newcommand{\fin}{f}

\newcommand{\ootimes}{\times}
\newcommand{\lleq}{\leq}
\newcommand{\ggeq}{\geq}

\title[Subconvexity for Rankin--Selberg and triple product]{The subconvexity problem for Rankin--Selberg and triple product L-functions}

\begin{document}
\author{Yueke Hu}

\address{YMSC, Tsinghua University}
\email{yhumath@tsinghua.edu.cn}

\author{Philippe Michel}
\address{EPFL}
\email{philippe.michel@epfl.ch}
\author{Paul Nelson}
\address{AArhus University}
\email{nelson.paul.david@gmail.com}
\begin{abstract}
In this paper we study the subconvexity problem for the Rankin-Selberg L-function and triple product L-function, allowing joint ramifications and conductor dropping range. We first extend the method of Michel-Venkatesh to reduce the bounds for L-functions to local conjectures on test vectors, then verify these local conjectures under certain conditions, giving new subconvexity bounds as long as the representations are not completely related. 
\end{abstract}
\maketitle
\tableofcontents

 
%


\section{Introduction}
Let $F$ be a number field with ring of adeles $\Aa$, and let $\Pi$ be an automorphic representation of a reductive group $G$. Let $L(\Pi,s)$ be the L-function associated to $\Pi$. The {subconvexity problem} for the central value $L(\Pi,1/2)$ is to obtain   
a (non-trivial) bound of the shape
$$L(\Pi,1/2)\ll_F C(\Pi)^{1/4-\delta}$$
where $\delta>0$ is some positive absolute constant (independent of $\Pi$) and $C(\Pi)$ is the analytic conductor of $\Pi$ which is a product over all places of the local analytic conductors $C_v(\Pi)$.

In the case $G=\GL_1$, $\GL_2$, the subconvexity problem has now been solved completely through the efforts of many people over the years, starting with the work of Weyl \cite{Weyl} and ending with \cite{MVIHES} with important intermediary results along the way such as \cite{Bu,HB,Good,DFI1,DFI2,DFI3, BHM, Ven}. 

Beyond $\GL_2$ this problem is far from well-understood, even for the Rankin--Selberg case $\GL_2\times \GL_2$. Here we briefly summarize the progress for Rankin--Selberg L-functions.

The problem is solved if one assumes that the conductor of one of the representations, $C(\pi_1)$ say, is fixed, or more generally, is bounded by a small enough but fixed positive power of the conductor of the other (see \cite{KMV,Sarnak,HaM,Ven,MVIHES,hu_triple_2017}.
 Some sporadic cases are known when both conductors vary simultaneously and  have comparable sizes, as in \cite{HoM,HolTem,ZYe,LMY}. 
These works usually  assume disjoint ramifications.

In contrast the problem remains open when $\pi_{2,v}\simeq \tilde{\pi}_{1,v}$ at places with ramifications, as in the QUE case. 
We shall refer to this case as the QUE-like case later on.

It is natural to ask if one can bridge the gap between the known cases and the QUE-like cases, allowing joint ramifications and some conductor dropping away from QUE-like cases for $\pi_i$. 

Alternatively one can understand this problem using the Langlands functoriality and view, for example, the Rankin--Selberg L-function as a special case of standard L-functions on $\GL_4$. There  has been some  important progresses in the subconvexity problem for L-functions on higher rank groups as in  \cite{Li11, Mu15}, and most recently \cite{Ne21}. But the available results are specific to the ``uniform growth'' cases where the conductor does not drop.  In practice, the techniques employed in these papers depend polynomially upon the extent to which the conductor drops, so they can be extended to the range where the conductor drops a little.
However, it is generally much more difficult to let the conductor drop substantially.

%

In this paper we obtain new cases of subconvexity bounds for the Rankin--Selberg and triple product L-functions of $\GL_2$, allowing general number fields, joint ramifications and conductor dropping range away from the QUE-like cases.
 In the non-archimedean aspect a special case of our main results states the following. 
\begin{theo}
Let $\pi_1$ and $\pi_2$ be cuspidal automorphic representations of $\GL_{2,F}$ both with trivial central character whose archimedean components are bounded.  Then
\begin{equation}\label{Eq:SubconvRScase}
  L(\pi_1 \times \pi_2, 1/2) \ll C(\pi_1 \times \pi_2)^{1/4+\eps}
  \left( \frac{C(\pi_1 \times \pi_2)}{ C (\pi_2 \times \tilde{\pi}_2)} \right)^{-\delta}.
\end{equation}
\end{theo}

By re-indexing if necessary, the above upper bound  provides a non-trivial power saving whenever 
there exists fixed $\gamma>0$ such that
\begin{equation}
\label{eq:norelated}
	\min_{i=1,2}
\{
C(\pi_i\times\tilde \pi_i) 
\}
\leq C(\pi_1\times\pi_2)^{1-\gamma} .
\end{equation}
Thus we obtain a subconvexity bound for the Rankin--Selberg L-function on $\PGL_2(F)$ even when the conductor drops all the way down except the QUE-like cases.

Note that when $C_v(\pi_1\times \pi_2)$ is large enough, we have by \cite[Corollary 3.1]{BH17}
$$C_v(\pi_1\times \pi_2)\geq \max_{i=1,2}\{ C_v(\pi_i\times \tilde\pi_i)\},$$
with equality in the QUE-like case. 

\subsection{Conductor dropping range}
Here we specify the meaning of conductor dropping range locally in terms of the Langlands correspondence. We fix a local field for the discussion.

In general 
if $\sigma$ is an n-dimensional local Weil--Deligne representation which decomposes into irreducible subrepresentations as $\oplus \sigma_i$, then we say $\sigma$ is in the conductor dropping range if any one of the normalized conductors $C(\sigma_i)^{\frac{1}{\dim \sigma_i}}$ has significantly different size  compared with others. In the Rankin--Selberg/triple product case, let $\sigma_{i,v}$ be associated to $\pi_{i,v}$ via the local Langlands correspondence, and we say $\{\pi_{i,v}\}$ are in the conductor dropping range if $\sigma=\bigotimes\limits_{i}\sigma_{i,v} $ is so.
As we shall also consider horizontal aspect uniformly, any fixed power of $p$ will also be considered as a significant difference.
\begin{example}\label{Ex:conductordrop}
Consider the  case where $\nu=p$ is a finite prime, and for $i=1,2$, $$\pi_{i,p}\simeq \pi(\chi_i,\chi_i^{-1})$$  are  principal series representations.
Suppose for simplicity that $p>2$ and $\chi_i$ are not quadratic.
The corresponding Weil--Deligne representation $\otimes_{i=1,2} \sigma_{i,p}$ decomposes into a sum of four characters (of the Galois group of $\Qp$), 
corresponding to the characters $\chi_1^{\pm 1}\chi_{2}^{\pm 1}$ via Class Field Theory.
Given a character $\chi$ of $\Qq_p^\times$, let $c(\chi)$ be its conductor exponent (that is the smallest integer such that $\chi$ is trivial on $1+p^{c(\chi)}\Z_p$, and $c(\chi)=0$ when $\chi$ is unramified).
 
The conductor dropping range occurs when $$c(\chi_1)=c(\chi_2)>c(\chi^{\pm 1}_1\chi^{\pm 1}_2) $$
for some $(\pm 1,\pm 1)$; for example it occurs when $\chi_2=\chi_1^{\pm 1}$. This is also the range where
\begin{equation}\label{Eq:defnConductordrop}
C(\pi_1\times \pi_2)< \max\limits_{i}\{C(\pi_i)^2\}.
\end{equation}
In comparison, when $c(\chi_1)\neq c(\chi_2)$ or $c(\chi_1)=c(\chi_2)=c(\chi^{\pm 1}_1\chi^{\pm 1}_2) $, we have
$$C(\pi_1\times \pi_2)= \max\limits_{i}\{C(\pi_i)^2\}.$$

There are similar examples when both $\pi_{1,p},\ \pi_{2,p}$ are supercuspidal representations constructed from related datum.

Consider now the case where the ramification of $\pi$ comes from essentially a single place, that is, $$C(\pi_i)\asymp C(\pi_{i,p})$$ for a  large enough prime $p$, and that $\pi_{1,p},\ \pi_{2,p}$ are principal series as above, then condition \eqref{eq:norelated} is automatic if $C(\pi_{1,p})\neq C(\pi_{2,p})$. On the other hand if $c(\chi_1)=c(\chi_2)\geq 2$, \eqref{eq:norelated} is satisfied if and only if
\begin{equation}\label{Eq1-2-1}
\min\{ c(\chi_1\chi_2),\ c(\chi_1\chi_2^{-1})\}\geq \gamma' c(\chi_i^2)
\end{equation}
for some constant $\gamma'>0$. (Actually $ \gamma'=\frac{\gamma}{1-\gamma}$ with $\gamma<1$ in the current case.)
This is because we have $\max\{c(\chi_1\chi_2),\ c(\chi_1\chi_2^{-1})\}= c(\chi_i^2)=\frac{1}{2}c(\pi_{i,p}\times \pi_{i,p})$, and $c(\pi_{1,p}\times \pi_{2,p})=2(c(\chi_1\chi_2)+ c(\chi_1\chi_2^{-1}))$. Here $c(\Pi_p)$ is the exponent of $C(\Pi_p)$.

On the other hand $\min\{ c(\chi_1\chi_2),\ c(\chi_1\chi_2^{-1})\}=0$ corresponds to the QUE-like case.

We leave it to readers to check that 
\eqref{eq:norelated} and  \eqref{Eq1-2-1} remain equivalent when $c(\chi_1)=c(\chi_2)=1$, in which case it could happen that $c(\chi_i^2)=0$ for quadratic characters.
\end{example}
In practice, one can also use \eqref{Eq:defnConductordrop} as a working definition for the conductor dropping range when we work with representations of trivial central characters in this paper. 

\subsection{Main result and strategy of proof}
To state our results for both the Rankin--Selberg and the triple product L-functions, we denote
$$Q=Q_\infty Q_\fin=\prod_\nu Q_\nu=\prod_\nu C_\nu(\pi_1\ootimes\pi_2\ootimes \pi_3),$$
and
$$P=P_\infty P_\fin=\prod_v P_v=
\prod_v \frac{Q_v^{1/2}
}{\max\limits_{i=2,3}\{ C_v(\pi_i\times \tilde{\pi}_i) \}}.$$
Here in the Rankin--Selberg case, we take $\pi_3=\pi(1,1)$.
Note that $Q_v=P_v=1$ for almost every $\nu$, and $P_v\geq 1$ at least when $Q_v$ is large enough.
 The main result we shall prove in this paper is the following theorem:
\begin{theo}\label{Theo:introthm}
Suppose that $\pi_1,\ \pi_2$ are cuspidal automorphic representations with trivial central characters. Suppose further that $(C_\fin(\pi_2), C_\fin(\pi_3))=1$ (this is automatic in the Rankin--Selberg case since $C_\fin(\pi_3)=1$). Then there exists an absolute constant $\delta>0$ such that for any $\eps>0$, one has
\begin{align*}
L(\pi_1\ootimes\pi_2\ootimes\pi_3,1/2)&\ll C(\pi_1\ootimes\pi_2\ootimes\pi_3)^{1/4+\eps} \frac{1}{P^\delta},\\
L(\pi_1\ootimes \pi_2, 1/2)&\ll C(\pi_1\ootimes \pi_2)^{1/4+\eps}\frac{1}{P^\delta}.
\end{align*}
Here the implicit constants depend on $\eps, F$ and (continuously) on the archimedean conductors $C_\infty(\pi_i)$ for $ i=1,2,3$.

\end{theo}
In the proof below we actually start with more general situations and formulate a  conjecture which potentially allows one to remove/relax the technical assumptions in this result. Currently we stick to the setting of Theorem \ref{Theo:introthm}.


The main global strategy follows and extends that of \cite{MVIHES}. 
Note first that the cases of Rankin--Selberg L-function and triple product L-function can be put into a uniform framework in the sense that they have analogous integral representations: 
for $G=\GL_2$ or $B^\times$ for $B$ some suitable quaternion algebra, depending on the (global and local) root numbers $\epsilon(\pi_1\ootimes \pi_2\ootimes\pi_3, 1/2)$, the Rankin--Selberg theory/triple product formula provides an identity of the shape
\begin{equation}\label{Eq:periodvague}
\left|\int\limits_{[G]}\prod\limits_{i=1}^{3}\varphi_i(g)dg\right|^2=\left|\peter{\varphi_1,\overline{\varphi_2\varphi_3}}\right|^2
\sim L(\pi_1\ootimes \pi_2\ootimes\pi_3,1/2) \cdot \prod_v I_v(\varphi_1,\varphi_2,\varphi_3).
\end{equation}
Here $[G]=Z_G(\A)G(F)\backslash G(\A)$; $\sim$ means equality up to some unimportant factors;
$\varphi_i\in \pi_i^{G}$ are $L^2$-normalized automorphic forms in the image of $\pi_i$ under the Jacquet-Langlands correspondence; $\peter{\cdot,\cdot}$ is the unitary integral pairing for automorphic forms on $G$ with the same central characters; the local integral $I_v(\varphi_1,\varphi_2,\varphi_3)$ is an integral of products of local matrix coefficients in case of triple product L-function, or  the local Rankin--Selberg integrals in case of Rankin--Selberg L-function. Furthermore in the Rankin--Selberg case, $G=\GL_2$, the Jacquet-Langlands correspondence is the identity and $\varphi_3$ is an Eisenstein series constructed out of $\pi(1,1)$.

On the one hand we need to control $I_v(\varphi_1,\varphi_2,\varphi_3)$ from below; this is done by choosing appropriate factorable vectors $\varphi_i$ which also captures the information of $Q=C(\pi_1\ootimes \pi_2\ootimes \pi_3)$ and detects the conductor dropping range.

On the other hand, we also have to control $$|\peter{\varphi_1,\overline{\varphi_2\varphi_3}}|^2$$ from above; we do this by applying the Cauchy--Schwarz inequality, and then the Plancherel formula (the regularized version in the Rankin--Selberg case) to get
\begin{align}\label{Eq:IntroGlobalControl}
|\peter{\varphi_1,\overline{\varphi_2\varphi_3}}|^2\leq \peter{\varphi_2\overline{\varphi_2},\varphi_3\overline{\varphi_3}}=\int\limits_{\pi'}\sum\limits_{\psi\in B(\pi')} \peter{\varphi_2\overline{\varphi_2},\psi}\overline{\peter{\varphi_3\overline{\varphi_3},\psi}}.
\end{align}

Note that the first inequality is the source of asymmetry in the definition of $P$.

When all $\pi_i$ are varying, the length of spectral sum/integral could however become very long. Denote $$M_v:=\min\{C_v(\pi_2), C_v(\pi_3)\}\hbox{,\ \  }M:=\prod_v M_v.$$
 We further assume that  there exist test vectors $\varphi_i$ such that for every place $v$, if $j_v=2$ or $3$ (which may depend on $v$) is
the index $i$ of $\pi_{i,v}$ which obtains the minimum $M_v$, that is, $$C_v(\pi_{j_v})=M_v,$$ then the complexity of $\vphi_{j_v,v}$ is controlled by $M_v$.  See Assumption \ref{Assumption:Main} for the precise meaning.
Then one can restrict the spectral sum/integral in \eqref{Eq:IntroGlobalControl}  to those $\pi'$ and $\psi$ that are also controlled in terms of $M_v$, in which range we control the contributions term-wisely.

To study the individual terms $\peter{\varphi_2\overline{\varphi_2},\psi}\overline{\peter{\varphi_3\overline{\varphi_3},\psi}}$ for non-residual spectrum, we apply  the period relation \eqref{Eq:periodvague} again and the known convexity bound for the resulting $L-$functions $L(\pi_i\times \tilde\pi_i\times \pi',1/2)$. One also needs a reasonable upper bound for $I_v(\varphi_i,\overline{\varphi_i}, \psi)$, by which we expect a power saving in terms of $P$ for non-residual spectrum which, combined with the amplification method, should give a power saving for the initial global period $\peter{\varphi_1,\overline{\varphi_2\varphi_3}}$.

To summarize, we expect
the global consideration above to reduce the subconvexity problem to the following
 test vector problem in our setting:
\begin{conj}\label{Conj:intro}

There exist factorable (normalized) test vectors $\varphi_i\in \pi_i$ satisfying the following three local requirements.
\begin{enumerate}
\item[(0)]The local component $\varphi_{j_v,v}$ should be controlled by $M_v$, to control the length of the spectral sum/integral;
\item[(1)]$I_v(\varphi_1,\varphi_2,\varphi_3)$ is bounded from below by $\frac{1}{Q_v^{1/4}};$
\item[(2)]$I_v(\varphi_i,\overline{\varphi_i}, \psi)$ is controlled from above to provide power saving in terms of $P_v$, while $\psi$ is controlled in terms of $M$.
\end{enumerate}
\end{conj}
See Conjecture \ref{Conj:localresults} below for more precise formulation. Note that scaling the test vectors would change the requirements for (1) and (2) simultaneously. So one can assume without loss of generality that $\varphi_i$ are $L^2$-normalized.

The main challenge of the test vector problem is the lower bound in item (1).
The existence of such test vectors is partially supported by previous know cases away from the conductor dropping range. 
Within the conductor dropping range however, the required lower bound for $I_v(\varphi_1,\varphi_2,\varphi_3)$ becomes substantially larger compared with the non-conductor dropping range.
Whether this is possible is not clear at all from, for example, \cite{MVIHES} or other literature.

There are two main sources of test vectors, coming from translates of classical newforms, or localized vectors (or minimal vectors) as used in \cite{Ne18,HNS,HN18,NV,Ne21}. The newforms are commonly used in the cases with disjoint ramifications, but they have more complicated matrix coefficients and Whittaker functions when involved in the conductor dropping range. On the other hand the localized vectors
have simpler description of matrix coefficients/Whittaker functions, but they do not work well for representations with small levels, requiring more case by case discussions.

To our surprise, in the setting of Theorem \ref{Theo:introthm} (in level aspect, with trivial central character, and $M=1$), we are able to find test vectors uniformly using only diagonal translates of newforms. 
The complicated nature of matrix coefficients/Whittaker functions of newforms from different $\pi_i$ turns out to be reflecting the conductor dropping range in a relatively simple fashion.
The extent of diagonal translates would then have to match  the extent of conductor dropping to achieve item (1) of Conjecture \ref{Conj:intro}.

We also have some sporadic evidences for the conjecture when $M$ is square-free, or in certain archimedean aspects, but we shall skip these cases for the sake of conciseness. In more general situations, the localized vectors may also turn out to be useful.

\subsection{Main local ingredient}
Here we briefly discuss the main local ingredient to verify item (1) of Conjecture \ref{Conj:intro}, in the  conductor dropping range  where $c=c(\pi_{1,v})=c(\pi_{2,v})$ is the exponent of the local conductors at $v$. 

Consider for simplicity the Rankin--Selberg case.
 The local representations $\pi_{i,v}, i=1,2$ of $\PGL_2(F_v)$, in most cases, can be associated to some characters $\theta_i$ over some \'{e}tale quadratic algebra $E_i/F_v$ by the local Langlands correspondence. Then the evaluation of the local period integral $I_v$ can be closely related to the following correlation between twisted Gauss sum/integrals:
$$P_k(\theta_1,\theta_2)=\sum\limits_{c(\chi)=c-k}\int\
\theta_1^{-1}(u)\chi_{E_1}(u) \psi_{E_1}(u)d^\times u \overline{
\int
\theta_2^{-1}(w) \chi_{E_2}(w)\psi_{E_2}(w)d^\times w}.
$$
Here $\psi_{E_i}=\psi\circ \Tr_{E_i/F_v}$ is an additive character on $E_i$,  $\chi$ is a multiplicative character of $F_v^\times$ with specified level, and $\chi_{E_i}=\chi\circ \text{Nm}_{E_i/F_v}$. 

An interesting phenomenon we find is that this quantity  $P_k$ is exactly capturing the conductor dropping range: when $k\geq c(\pi_{1,v}\times\pi_{2,v})/2$, there is no cancellation in the sum over $\chi$. On the other hand when $c/2\leq k<   c(\pi_{1,v}\times\pi_{2,v})/2$, there are heavy cancellations, and $P_k$ becomes much smaller and equals zero in many cases. The choice of test vector is then reduced to maximizing the contributions from those $P_k$ with $k\geq c(\pi_{1,v}\times\pi_{2,v})/2$.
This phenomenon may also have independent interest.
\subsection{The structure of the paper}
\begin{enumerate}
\item[-]After setting up notations in Section \ref{secnotations}, we recall in Section \ref{secgeneralbound} the triple product formula and the amplification method, and prove the main result Theorem \ref{metasubconvexbound}  that reduces the subconvexity problem for $L(\pi_1\ootimes\pi_2\ootimes\pi_3,1/2)$ to Conjecture \ref{Conj:localresults}, which becomes a local problem. 
\item[-] We recall in Section \ref{Sec:localPrep} some more local preparations including the local Whittaker functions. In Section \ref{Sec:OrthgonalWhittaker} we  prove Proposition \ref{Prop:Wipairing} for the partial pairing $P_k$ which detects the conductor dropping range. Finally in Section \ref{Sec:localbounds} we verify Conjecture \ref{Conj:localresults} for the setting in Theorem \ref{Theo:introthm}. 
\end{enumerate}

\subsection{Acknowledgment}
The first author is supported by the National Key Research and Development Program of China (No. 2021YFA1000700).

\section{Notations}\label{secnotations}
 
Let $F$ be a number field.
We denote by $\mcV$ the set of places of $F$, by  $\mcV_{f}$ and $\mcV_{\infty}$ the finite and archimedean places; 
we denote by $F_v$ the local field at $v$.
We denote by $\Aa=\resprod_v F_v$ the ring of ad\`eles of $F$ and for any $S\subset \mcV$, $$\Aa_S=\resprod_{v\in S}F_v,\ \Aa^{(S)}=\resprod_{v\not \in S}F_v.$$

We denote by $\psi:F\bash\Aa\ra\Ct$ the additive character $\psi(x)=\psi_\Qq(\tr_{F/\Qq}(x))$ where $\psi_\Qq$ 
is the additive character on $\Qq\bash\Aa_\Qq$ whose restriction to $\Rr$ is $\exp(2\pi ix_\Rr)$; the character $\psi$ decomposes into a product of  characters of $F_v$: 
$\psi=\prod_{v}\psi_v.$

Let $\rmG$ be a reductive group over $F$. For any place $v\in\mcV$ we denote by
$\Gv=\rmG(F_v)$ the groups of $F_v$ points, by $\GA=\resprod_{v\in\mcV}\Gv$ the group of adelic points, and more generally for any $S\subset \mcV$, we denote
$$\rmG(\Aa_S)=\resprod_{v\in S}G_v,\ \rmG(\Aa^{(S)})=\resprod_{v\not\in S}G_v.$$
We set
$$[\Gv]:=F_v^\times\bash\Gv,\ [\rmG]:=\GF \At\bash\GA.$$

 We choose for each place $v$ a maximal compact subgroup $K_v\subset\Gv$ and for $v$ finite a decreasing family of principal congruence subgroups $ K_v[m]\subset K_v,\ m\in\Nn$ as well as Haar measures on $\Gv$, $\GA$ and $[\rmG]$ as in \cite[\S 2.1]{MVIHES}. For $v$ non-archimedean, $K_v$ has measure $1$. 
 
In this paper $\rmG$ will most of the time be equal to $\Bt$ for $\rmB$ some quaternion algebra over $F$ (or product of such groups) and for any $v$ such that $\Gv=\GL_{2,v}$ we make the same choice of Haar measures as in \cite[\S 3.1.5]{MVIHES}. The choice of a maximal order $\order_\rmB\subset\rmB$ and of a $\Zz$-basis define local norms on $\rmB_v$ for every $v$. For $u=(u_v)\in\BtA$ we denote by 
\begin{equation}\label{Eq:defnOperatornorm}
\|u\|=\prod_v\|u_v\|_v
\end{equation}
 the operator norm of $\Ad u$, the action of $u$ on $\rmB(\Aa)$ by conjugation. The infinite product in the definition of $\| u\|$ is justified by the following straightforward result:
\begin{lem}
Suppose $G_v=\GL_{2,v}$ and $u_v=z k_1 a(x)k_2$ in the local Cartan decomposition. 
Then 
$$\| u_v\|_v\asymp \max\{|x|_v,|x|_v^{-1}\}$$
with absolute implied constant. Furthermore if $v$ is a p-adic place and $x\in \Oo_v^\times$, $$\|u_v\|_v=1.$$
On the other hand if $G_v$ comes from a local division algebra, we have $$\|u_v\|_v\asymp 1$$ by compactness.
\end{lem}

 \subsection{Automorphic representations}

Let $\rmG=\Bt$; given $\chi$ a unitary character of $F^\times\bash\At$ we denote by $\mcA(\rmG,\chi)$ the set of (unitary) automorphic representations of $\GA$ with central character $\chi$. Given such a representation we denote by $(\pi_v)_{v\in\mcV}$ the sequence of its local constituents: these are unitary $\Gv$-representation all but finitely many of which are unramified principal series representations (of $\Gv=\GL_{2}(F_v)$ for $v$ finite), and one has $\pi\simeq\bigotimes_v\pi_v$ the restricted tensor product being taken with respect to a sequence of spherical vectors $(\vphi^\circ_v)_{v}$ indexed by the finite places where $\pi$ is unramified (admits a $K_v$-invariant vector). We denote by $\unram(\pi)\subset\mcV_f$ the set of (finite) unramified places of $\pi$ and $\ram(\pi)\subset\mcV$ its complement\footnote{ by convention infinite places are contained in $\ram(\pi)$ even the ones for which $\pi_v$ is spherical}. At any place where $\Gv=\GL_{2,v}$, we denote by $\mcW(\pi_v)$ the Whittaker model of $\pi_v$ relative to the additive character $\psi_v$ and for $v\in\unram(\pi)$, we take 
the spherical vector $W^\circ\in\mcW(\pi_v)$ such that $W^\circ(1)=1$.

We define inner products on $\pi$ and $\pi_v$ as follows:  
For $\varphi_i\in \pi$, $i=1,2$,
$$\peter{\varphi_1,\varphi_2}=\int\limits_{[G]} \varphi_1(g)\overline{\varphi_2(g)}dg;$$
For $\varphi_{i,v}\in \pi_v$ and $W_{i,v}=W_{\varphi_i,v}$ the associated Whittaker functions in the case $B_v^\times=\GL_2(\F_v)$, 
$$\peter{\varphi_{1,v},\varphi_{2,v}}=\int\limits_{F_v^\times} W_{1,v}(a(x)) \overline{W_{2,v}(a(x))}d^\times x .$$
The $L^2$-norms of $\varphi$ and $\varphi_v$ are defined with respect to these pairings. Inner products can also be defined for compactly induced representations when $B_v$ is a division algebra, though we do not need the details in this paper.


We denote by $\tilde\pi\in\mcA(\rmG,\chi^{-1})$ the contragredient representation of $\pi$,
and by $\pi^{JL}\in\mcA(\GLd,\chi)$ the automorphic representation corresponding to $\pi$ under the Jacquet-Langlands correspondence (one has $\pi_v\simeq\pi^{JL}_v$ for $v\not\in\ram(\pi)$). 

\begin{conv}
Unless stated otherwise a local unitary representation of $\Gv$ for some $v$ will always be understood as a local constituent of some global automorphic representation. In particular, the Langlands parameters (when the representation is unramified) satisfy the best known approximation towards the Ramanujan-Petersson conjecture.	
\end{conv}

\begin{conv} To ease notations (and as long as the context is clear) we will  not distinguish between $\pi$ and the Jacquet-Langlands constituent $\pi^{JL}$ at least notationally: suppose we have a quantity, say $L(\pi^{JL})$, which a priori is constructed out of $\pi^{JL}$: we will denote it indifferently $L(\pi^{JL})$ or $L(\pi)$. This convention will be in use especially to denote L-functions and related quantities like conductors.
\end{conv}
\subsection{L-functions}\label{Lfunctions}
Let $\Pi\simeq\bigotimes_{v}\Pi_v$ be an automorphic representation of some reductive group $\rmG$ and $\rho$ a representation of its dual group $\rmG^L$ of dimension $d$; to these data, one associates an L-function given by an Euler product of local $L-$factors of degree $d$ converging in some halfspace $\{s,\ \Re s\gg 1\}$
$$L(\Pi,\rho,s)=\prod_{v\in\mcV_f}L_v(\Pi,\rho,s),$$
\yhn{where $ L_v(\Pi,\rho,s)=\prod_{i=1}^d(1-\frac{\alpha_{v,i}(\Pi,\rho)}{q_v^s})^{-1}$ at unramified places for $\alpha_{v,i}(\Pi,\rho)$ the eigenvalues of $\rho(c(\Pi_v))$ where $c(\Pi_v)$ is the Satake matrix associated to $\Pi_v$. At ramified places, $L_v$ is often defined as the common denominator of the relevant local period integrals, or using the local Langlands correspondence.}

The archimedean factors are defined similarly as
$$L_\infty(\Pi,\rho,s)=\prod_{v\in\mcV_\infty}L_v(\Pi,\rho,s),\ L_v(\Pi,\rho,s)=\prod_{i=1}^d\Gamma_{F_v}(s+\alpha_{v,i}(\Pi,\rho)).$$
In several cases one can prove that $L(\Pi,\rho,s)$ admits analytic continuation to $\Cc$ with at most a finite number of poles located on the line $\Re s=1$ and that it admits a functional equation of the shape
$$\Lambda(\Pi,\rho,s)=\eps(\Pi,\rho)C_f(\Pi,\rho)^{\frac{1}2-s}\ov{\Lambda(\Pi,\rho,1-\ov s)}$$
where
$$ \Lambda(\Pi,\rho,s)=L_\infty(\Pi,\rho,s)L(\Pi,\rho,s),\ |\eps(\Pi,\rho)|=1,\ C_f(\Pi,\rho)\in\Nn_{\geq 1}.$$ 
The factor $\eps(\Pi,\rho)$ is the root number and $C_f(\Pi,\rho)$ is the arithmetic conductor. Both quantities can also be factored into a product of local root numbers (upon choosing a non-trivial additive character of $F\bash\Aa$) and local conductors which are almost everywhere equal to $1$ and arise from local functional equations:
$$\eps(\Pi,\rho)=\prod_v \eps_v(\Pi,\rho),\ C_f(\Pi,\rho)=\prod_{v\in\mcV_F} C_v(\Pi,\rho).$$
One define the {\em analytic conductor} of $(\Pi,\rho)$ by completing the arithmetic conductor with archimedean local conductors:
$$C(\Pi,\rho)=C_f(\Pi,\rho)C_\infty(\Pi,\rho)=C_f(\Pi,\rho)\prod_{v\in\mcV_\infty} C_v(\Pi,\rho)$$
$$C_v(\Pi,\rho)=\prod_{i=1}^d(1+|\alpha_{v,i}(\Pi,\rho)|)^{[F_v:\Rr]},\ v\in\mcV_\infty.$$

\subsubsection{Convexity bound}
The following bound is also sometimes known:
for any $\eps>0$, one has
$$L(\Pi,\rho,1/2)\ll_{F,d,\eps}C(\Pi,\rho)^{1/4+\eps}.$$
The subconvexity problem consist in improving the exponent $1/4$ to one strictly smaller.

All the properties mentionned above are known for the following L-functions
\begin{itemize}
\item[-] Hecke-Godement-Jacquet L-functions: $n\geq 1$, $\rmG=\GL_n$, $\rho=\Std$, $d=n$. The L-function will be noted $L(\Pi,s)$.
\item[-] Rankin--Selberg L-functions: $\rmG=\GLd\times\GLd$, $\rho=\Std\otimes\Std$, $d=4$. If $\Pi=\pi_1\ootimes\pi_2$, L-function will be noted $L(\pi_1\ootimes\pi_2,s)$.
\item[-] Adjoint L-functions: $\rmG=\GLd$, $\rho=\Ad$, $d=3$. The L-function will be noted $L(\Pi,\Ad,s)$.
\item[-] Triple product L-function: $\rmG=\GLd\times\GLd\times\GLd$, $\rho=\Std\otimes\Std\otimes\Std$, $d=8$. If $\Pi=\pi_1\ootimes\pi_2\ootimes\pi_3$, the L-function will be noted $L(\pi_1\ootimes\pi_2\ootimes\pi_3,s)$.
	
\end{itemize}

In the following 
we will write
\begin{equation}\label{Eq:defQv}
C(\pi_1\times\pi_2\times\pi_3)=\prod_{v}C_v(\pi_1\times\pi_2\times\pi_3)
=Q=\prod_v Q_v.
\end{equation}

\subsubsection{Bound for adjoint L-functions}\label{smallthings}
Given $\pi\in\mcA(\GLd,\chi)$ it follows from \cite{GHL} that 
$$L(\pi,\mathrm{Ad},1)=C(\pi)^{o(1)}\hbox{ as }C(\pi)\ra\infty.$$


\subsection{Convention for inequalities}\label{Section:Inequalityconvention}
For two functions $f_1$ $f_2$ of variables $a,b,c$,etc., by
$$f_1(a,b,\cdots)\ll_{a,b} f_2(a,b,\cdots),$$
we mean $f_1(a,b,\cdots)\leq C_{a,b} f_2(a,b,\cdots)$ for some implied constant $C_{a,b}$ depending only on $a$ and $b$, 
independent of other variables. This convention is commonly used among analytic number theorists.

In the following we will omit additional variables from notations without confusion. To deal with infinite products over all places, we also introduce the following notation:

\begin{defn}
Let $f_1, f_2$ be two functions of $v\in \mcV$ and other variables $a,b,c$, etc. By
$$f_1\lleq_{v,a} f_2,$$
we mean $f_1\ll_{v,a} f_2$ with implied constant $C_{v,a}$, and $C_{v,a}=1$ if the cardinality of residue field $q_v\geq N_a$ where $N_a$ depends only on $a$. 

We define $f_1\ggeq_{v,a}f_2$ similarly. By $f_1\asymp_{v,a} f_2$, we mean $f_1\lleq_{v,a} f_2$ and $f_1\ggeq_{v,a} f_2$.
\end{defn}
The following lemma is straightforward:
\begin{lem}
If $f_1(v)=f_2(v)=1$ for almost all $v$, and $f_1\lleq_{v,a} f_2$, then
$$\prod\limits_{v}f_1\ll_a \prod\limits_{v}f_2.$$
\end{lem}
The nuance here is that the number of local factors $f_i(v)$ which are not $1$ may depend on other variables.

\begin{example}\label{Ex:control}
Let $f(v)=\zeta_v(1)$ for $v| Q_{\fin}$ and $f(v)=1$ otherwise. 
Then 
$$Q_v^{-\epsilon}\lleq_{v,\epsilon}f(v)\lleq_{v,\epsilon} Q_v^\epsilon,$$
$$Q^{-\epsilon}\ll_{\epsilon}\prod\limits_{v}f(v)\ll_{\epsilon}Q^\epsilon.$$
Same conclusions hold if $f(v)=C$ some fixed constant for $v|Q_{\fin}$ and $f(v)=1$ otherwise.
\end{example}

%

 \section{bounds for global periods}\label{secgeneralbound}

Let $\pi_i,\ i=1,2,3$ be three generic automorphic representations of $\GLd(\Aa)$ whose product of central characters is trivial; we will moreover assume that $\pi_1$ and $\pi_2$ are cuspidal.

The basic analytic properties stated in \S \ref{Lfunctions} (including the convexity bound) are known for the associated triple product L-function $L(\pi_1\ootimes\pi_2\ootimes\pi_3,s)$. If $\pi_3$ is a principal series representation this is the Rankin--Selberg theory; when $\pi_3$ is cuspidal these are non-trivial facts which are consequences of the work of several people including Garrett, Piatestsky-Schapiro--Rallis and Ramakrishnan \cite{Garrett,PSR,Ram}. $L(\pi_1\ootimes\pi_2\ootimes\pi_3,s)$ is holomorphic unless $\pi_3$ is a principal series and $\pi_1,\pi_2$ are contragredient to one another up to a twist. Note that the convexity bound follows from Ramakrishnan's proof that $L(\pi_1\ootimes\pi_2,s)$ is $\GL_4$-automorphic and from an argument of Molteni \cite{Mol}.

The condition on the product of central characters being trivial implies that $L(\pi_1\ootimes\pi_2\ootimes\pi_3,s)$ has real coefficients and that the global and local root numbers $$\eps(\pi_1\ootimes\pi_2\ootimes\pi_3)=\prod_{v\in\mcV}\eps_v(\pi_1\ootimes\pi_2\ootimes\pi_3)$$
are all $\pm 1$.

The {\em subconvexity problem} we are studying in this paper amounts to finding an absolute $\delta>0$ such that
$$L(\pi_1\ootimes\pi_2\ootimes\pi_3,1/2)\ll_F C(\pi_1\ootimes\pi_2\ootimes\pi_3)^{1/4-\delta}.$$

\subsection{The triple product formula}
We may assume that 
\begin{equation}\label{plusonesign}
 \eps(\pi_1\ootimes\pi_2\ootimes\pi_3)=+1,
 \end{equation}
 as otherwise $L(\pi_1\ootimes\pi_2\ootimes\pi_3,1/2)=0$ and we are done.
This implies that the set of places of $F$ at which the local root number equals $-1$ has even cardinality; this set determines a unique quaternion algebra $\rmB$ (the one ramified precisely at this even set of places) and three $\Bt$-automorphic representation $(\pi^\rmB_1,\pi^\rmB_2,\pi^\rmB_3)$ whose images under the Jacquet-Langlands correspondence are $(\pi_1,\pi_2,\pi_3)$. More precisely, at any place for which $\rmB$ is unramified, $\rmB^\times_v=\GL_{2,v}$ and $\pi_{i,v}^\rmB\simeq\pi_{i,v}$  while for any place where $\rmB$ is ramified, $\pi_{i,v}^\rmB$ is finite dimensional and $\pi_{i,v}$ is the image of $\pi_{i,v}^\rmB$ under the local Jacquet-Langlands correspondence; note that in the Rankin--selberg case, when $\pi_3$ is a principal series representation, $\Bt=\GLd$ and the correspondence is the identity. 

From now on, the quaternion algebra $\rmB$ will be considered {\em fixed}; we write
$$\pi=\pi^{\rmB}_1\in\mcA(\rmG,\chi_1),\ \Pi=\pi^{\rmB}_2\otimes\pi^{\rmB}_3\in\mcA(\rmG\times\rmG,\chi^{-1}_1).$$ 
The later is an automorphic representation of $\GA\times\GA$ on which $\GA$ acts via the diagonal embedding $\rmG\hookrightarrow\rmG\times\rmG$.

It follows from the work of Prasad \cite{Pra} that for any place $v$ the space 

\noindent $\Hom_\Gv(\Pi_v,\tilde\pi_{v})$ of $\Gv$-invariant intertwiners is one dimensional and a generator is given by the functional, $\Pro^{\pi_v}_{\Pi_v}$ say, well defined up to a scalar of modulus $1$ satisfying for any $\vphi_{v}\otimes\varPhi_{v}\in \pi_v\otimes\Pi_v$ 
$$\left|\Pro^{\pi_v}_{\Pi_v}(\vPhi_v)(\vphi_v)\right|^2=\frac{L_v(\pi,\Ad,1)L_v(\Pi,\Ad,1)}{\zeta_{F_v}(2)^2L_v(\pi\times\Pi,1/2)} \int_{[\Gv]}\peter{\vphi_{v},g.\vphi_{v}}\peter{\vPhi_{v},g.\vPhi_{v}} dg.$$
\begin{rem}Observe that the $\Gv$-integral is absolutely converging due to the known bounds for matrix coefficient \cite{KiSa,BlBr} and the assumption that $\pi_v$ and $\Pi_v$ are local components of global automorphic representations.
\end{rem}

Furthermore by work of Garrett, Piatetsky-Shapiro--Rallis, Harris-Kudla, Watson and Ichino \cite{Garrett,PSR,HK,Watson,Ichino},  the space of global $\GA$-equivariant intertwiners $\Hom_\GA(\Pi,\tilde\pi)$
is non-zero (and therefore one dimensional) if an only if 
$$L(\pi\times\Pi,1/2)\not=0$$ and then is generated by $\Pro^{\pi}_{\Pi}$ given by
\begin{align}\label{Prpidef}
\Pro^{\pi}_{\Pi}(\vPhi)(\vphi)&=\int_{[\rmG]} \vphi(g)\vPhi(g)dg \\
\Pro^{\pi}_{\Pi}(\vPhi)&=\sum\limits_{\vphi\in \mcB(\pi) }\Pro^{\pi}_{\Pi}(\vPhi)(\vphi)\notag
\end{align}
where $\vPhi$ is viewed as a function on $\GA\times\GA$, $g\mapsto \vPhi(g)$ denotes the restriction of $\vPhi$ to diagonally-embedded $\GA$ in $\GA\times\GA$ and $\mcB(\pi)$ denotes an orthonormal basis for $\pi$; note that the integral is converging since $\pi$ is cuspidal.
More precisely the following period formula relating global and local intertwiners holds (\cite
{Ichino}):
\begin{theo}\label{triplethm}
For $\vphi\otimes\vPhi\simeq \bigotimes_v\vphi_v\otimes\vPhi_v$ any non-zero factorable global vector in $\pi\otimes\Pi$, one has
\begin{equation}\label{tripleformula}
\frac{\left|\int_{[\rmG]} \vphi(g)\vPhi(g)dg\right|^2}{\peter{\vphi,\vphi}\peter{\vPhi,\vPhi}}=C_F\frac{L(\pi\times\Pi,1/2)}{L(\pi,\mathrm{Ad},1)L(\Pi,\mathrm{Ad},1)}\prod_v \frac{\left|\Pro^{\pi_v}_{\Pi_v}(\vPhi_v)(\vphi_v)\right|^2}{\peter{\vphi_v,\vphi_v}\peter{\vPhi_v,\vPhi_v}},
\end{equation}
where $C_F>0$ is an absolute constant. 


\end{theo}
\begin{rem} If $\rmG=\GLd$ and $\pi_3$ is a principal series representation the above formula amounts to the Rankin--Selberg theory, and $\Pro^{\pi_v}_{\Pi_v}(\vPhi_v)(\vphi_v)$ becomes the standard local Rankin--Selberg integral. 
\end{rem}
Assuming that $\vphi,\vPhi$ are both globally and locally $L^2$-normalized and using \S \ref{smallthings} we rewrite this formula in the following less precise form
\begin{equation}\label{tripleformula3}
{\left|\int_{[\rmG]} \vphi(g)\vPhi(g)dg\right|^2}=C(\pi\times\Pi)^{o(1)}{L(\pi\times\Pi,1/2)}\prod_v {\left|\Pro^{\pi_v}_{\Pi_v}(\vPhi_v)(\vphi_v)\right|^2}.
\end{equation}

\subsection{Test vectors} 

Following a method initiated in \cite{Ven}, we use this last formula to get our hands on the size of the central value
$L(\pi\times\Pi,1/2)$.

By  \eqref{Eq:defQv}
we rewrite \eqref{tripleformula3} in the following form
\begin{equation}\label{Lperiod}
\frac{L(\pi\times\Pi,1/2)}{C(\pi\times\Pi)^{1/4+\epsilon}}=\frac{\left|\int_{[\rmG]} \vphi(g)\vPhi(g)dg\right|^2}{\prod_v \left|\Pro^{\pi_v}_{\Pi_v}(\vPhi_v)(\vphi_v)\right|^2Q_v^{1/4}}.	
\end{equation}
From this  we see that in order to solve the subconvexity problem it will be sufficient to show the existence of factorable vectors $\vphi$, $\vPhi$ for which the global period is small, that is,
$$\left|\int_{[\rmG]} \vphi(g)\vPhi(g)dg \right|^2\ll Q^{-\delta}$$ for some $\delta>0$
 and for which the product of local periods 
\begin{equation}\label{Eq:expectedlocalint}
\prod_v \left|\Pro^{\pi_v}_{\Pi_v}(\vPhi_v)(\vphi_v)\right|^2Q_v^{1/4}\gg Q^{-\epsilon}
\end{equation}
is not too small. 
 \begin{rem} Observe that this product is converging: when $v$ is finite, $\pi_v$ and $\Pi_v$ are unramified principal series and  $\vphi^\circ_v,\ \vPhi^\circ_v$ are the spherical vectors, we have
 	$$\left|\Pro^{\pi_v}_{\Pi_v}(\vPhi^\circ_v)(\vphi^\circ_v)\right|^2=Q_v^{1/4}=1.$$
 \end{rem}

We will search for vectors of the following form:
\begin{defn}\label{Defn:testvec1}
 For $i=1,2,3$ a test vector for $\pi^\rmB_i\simeq\bigotimes_v\pi^\rmB_{i,v}$ is a factorable vector $\vphi_i=\otimes_v\vphi_{i,v}$ such that
\begin{enumerate}
\item For all $v$, $\|\vphi_{i,v}\|=1$.
\item For all $d\in\Nn$, $\mcS_d^{\pi_v}(\vphi_{i,v})\ll_d Q_v^{A(d)}$ for some $A$ depending on $d$ only.	
\item For $v\not\in \bigcup_{i=1,2,3}\ram(\pi^\rmB_i)$, $\vphi_{i,v}=\vphi^\circ_{i,v}$ is a spherical element.

\end{enumerate}
Here $\mcS_d^{\pi_v}$ is the Sobolev norm defined as in, for example, \cite[Section 2.3.5]{MVIHES}.
A test vector of $\Pi=\pi_2^\rmB\otimes\pi_3^\rmB$ is a tensor product of test vectors $\vPhi=\vphi_2\otimes\vphi_3$. 
	
\end{defn}

 The control of local periods from below is essentially a local problem (finding for each ramified place $v$ a quantitatively "good"  vector for the intertwiner $\Pro^{\pi_v}_{\Pi_v}$), but the explicit choice of test vectors is also related to the global upper bound. We make the following definition:
 \begin{defn}\label{Defn:feasindex}
 Given $\vphi_i,\ i=1,2,3$ test vectors of $\pi_i^\rmB$, the feasibility index of $\vPhi=\vphi_2\otimes\vphi_3$ with respect to $\vphi=\vphi_1$ at the place $v$ is defined as
 $$\Feas_v(\vPhi|\vphi)=\left|\Pro_{\Pi_v}^{\pi_v}(\vPhi_v)(\vphi_v)\right|^2Q_v^{1/4},$$
measuring if  \eqref{Eq:expectedlocalint} can be achieved for proper $L^2$-normalized local test vectors; 
 Define the total feasibility index to be the product
$$\Feas(\vPhi|\vphi)=\prod_v\Feas_v(\vPhi|\vphi).$$
Finally the {\em feasibility index} of $\vPhi$ with respect to the representation $\pi_1^\rmB$ is
$$\Feas(\vPhi|\pi):=\sup_{\vphi}\Feas(\vPhi|\vphi)$$
where the supremum is taken over test vectors in $\pi_1^\rmB$.
 \end{defn}

\subsection{The amplification method}
To give global upper bound we proceed as in \cite{MVIHES} by first convolving $\vphi$ by some amplifier $\sigma$ with the following properties: 

\begin{lem}\label{amplifier} Let $L\geq 1$ be some parameter (to be chosen for optimization)
, there exist a complex-valued measure $\sigma$ on $\GA$ 
 satisfying
\begin{enumerate}
\item $\sigma$ is compactly supported on $\rmG(\Aa_{S_L})$ where $S_L$ is a set of finite places of norm $\leq L$ at which $\vphi$, $\vPhi$ and $\psi$ are all unramified;
\item  \label{ampli0} For any $u\in\supp(\sigma)$ we have $ \|u\|\leq L$   where  $\|u\|$ is the operator norm as in \eqref{Eq:defnOperatornorm}. 
\item Let $|\sigma|$ denote the  total variation measure. Then
the total mass of $|\sigma|$ is bounded above by $L^{B}$, for some absolute constant $B$. 
Moreover, with $|\sigma|^{(2)}=|\sigma|\star|\check\sigma|$, one has for any 
$\gamma>1/2$
$$\int\|u\|^{-\gamma}d|\sigma|^{(2)}(u)\leq L^{-\eta},\ \hbox{for some }\eta=\eta(\gamma)>0.\label{ampli2}$$\label{ampli3}
\item $\varphi\star\sigma=\lambda \varphi$ with $\lambda \gg_{F,\eps} Q^{-\eps} $ for any $\eps>0$.\label{ampli4} 

\end{enumerate}	
\end{lem}
Since at least one of $\pi_2$ and $\pi_3$ is cuspidal, $\vPhi$ is rapidly decreasing. As $\vphi$ is $L^2$-normalized, the Cauchy--Schwarz inequality  implies that:
\begin{align}\label{Cauchy}
|\lambda|^2\left|\int_{[\rmG]} \vphi(h)\vPhi(h)dh\right|^2= &\left|\int_{[\rmG]} (\vphi\star\sigma)(g)\vPhi(g)dg\right|^2\\
\leq & \int_{[\rmG]} \left|\vPhi\star\check\sigma(g)\right|^2dg=
\peter{(\vphi_2\vphi_3)\star(\ov\sigma\star\check\sigma),\vphi_2\vphi_3}.	\notag
\end{align}

 The later inner product decomposes into a sum of terms of the shape
 $$\peter{u.\vphi_2\ \ov{\vphi_2},\vphi_3\ u.\ov{\vphi_3}}=\peter{\vPhi_{2,u}|_\rmG,\vPhi_{3,u}|_\rmG}$$
where $u\in\supp(\ov\sigma\star\check\sigma)$,  $\vPhi_{2,u},\vPhi_{3,u}$ are the functions on $\GA\times\GA$ given by
  \begin{equation}
  \label{Phi23udef}
  \vPhi_{2,u}(g_1,g_2)= u.\vphi_2(g_1)\times\ov\vphi_2(g_2),\vPhi_{3,u}(g_1,g_2)=\vphi_3(g_1)\times u.\ov\vphi_3(g_3),	
  \end{equation}
  and $|_\rmG$ denote the restriction to $\GA\hookrightarrow\GA\times\GA$ (diagonally embedded). The function $\vPhi_{2,u}$ and $\vPhi_{3,u}$ are factorable vectors in the $\GA^2$-automorphic representations
  $$\Pi_2=\pi_2\otimes\tilde\pi_2,\ \Pi_3=\pi_3\otimes\tilde\pi_3.$$
  To simplify notations we will omit the restriction $|_\rmG$ of $\vPhi_2$ and $\vPhi_3$ and write $\peter{\vPhi_{2,u},\vPhi_{3,u}}$ in place of $\peter{\vPhi_{2,u}|_\rmG,\vPhi_{3,u}|_\rmG}$. Thus our discussion so far gives 
\begin{align}\label{Cauchy2}
|\lambda|^2\left|\int_{[\rmG]} \vphi(h)\vPhi(h)dh\right|^2\leq  \int\limits_{u}\peter{\vPhi_{2,u},\vPhi_{3,u}} d(\ov\sigma\star\check\sigma)(u)
\end{align}
 which is essentially a finite sum in $u$.

 \subsection{Applying the Plancherel formula}\label{Plancherelsection}
 Suppose first that $\pi_3$ is a  cuspidal automorphic representation, then the restrictions to $\GA$ of $\vPhi_2$ and $\vPhi_3$ are rapidly decreasing modulo $\At\GF$ and by the Plancherel formula the above inner product decomposes as 
 \begin{align*}
\peter{\vPhi_{2,u},\vPhi_{3,u}}=&\int_{\pi'\in\mcA(\rmG,1)}\sum_{\psi\in\mcB(\pi')}
\peter{\vPhi_{2,u},\psi}\peter{\psi,\vPhi_{3,u}}\ d\mu_{\mathrm{Pl}}(\pi')\\
\label{decomp2}
=&\int_{\pi'\in\mcA(\rmG,1)}\sum_{\psi\in\mcB(\pi')}
\Pro^{\pi'}_{\Pi_2}(\vPhi_{2,u})(\psi)\ov{\Pro_{\Pi_3}^{\pi'}(\vPhi_{3,u})(\psi)}\ d\mu_{\mathrm{Pl}}(\pi')
\\=&\int_{\pi'\in\mcA(\rmG,1)}\peter{\Pro^{\pi'}_{\Pi_2}(\vPhi_{2,u}),{\Pro_{\Pi_3}^{\pi'}(\vPhi_{3,u})}} d\mu_{\mathrm{Pl}}(\pi')
 \end{align*}
 where
 $\pi'$ runs over the automorphic representations of $\Bt$ with trivial central character, $\mcB(\pi')$ is an orthonormal basis of the space of $\pi'$, and
$\Pro^{\pi'}_{\Pi_i},\ i=1,2$ is defined as in \eqref{Prpidef}.

Restricting the above integral to the finite and the generic spectrum we write
$$\peter{\vPhi_{2,u},\vPhi_{3,u}}=\peter{\vPhi_{2,u},\vPhi_{3,u}}_{\mathrm{finite}}+\peter{\vPhi_{2,u},\vPhi_{3,u}}_{\mathrm{generic}}.
$$

The contribution of the finite spectrum (characters of order at most $2$)
equals a finite sum of products of matrix coefficients:
$$
\peter{\vPhi_{2,u},\vPhi_{3,u}}_{\mathrm{finite}}=\sum_{\chi^2=1}\peter{u.\vphi_2,\chi\vphi_2}\peter{u.\vphi_3,\chi\vphi_3}
\ll Q^{\epsilon}\|u\|^{-\gamma},	
$$
where $\gamma=1-2\theta$ with $\theta<7/64$ (\cite{KiSa,BlBr}). Indeed the above sum comprises  at most two non-zero terms (the trivial character and the unique quadratic character (if any) such that $\pi_2\simeq\chi\pi_2$ and $\pi_3\simeq\chi\pi_3$); The bound then follows from bounds for spherical matrix coefficients, the support for $u$, and the bound towards Ramanujan-Petersson conjecture \cite{KiSa,BlBr}.

Using this bound together with item (\ref{ampli3}) of Lemma \ref{amplifier} we obtain that
\begin{equation}\label{finitecontribution}
|\ov\sigma|\star|\check\sigma|(u\mapsto|\peter{\vPhi_{2,u},\vPhi_{3,u}}_{\mathrm{finite}}|)\ll Q^{\epsilon}L^{-\delta}	
\end{equation}
for some absolute $\delta>0$.

\subsection{The generic contribution}\label{secgeneric} It remains to bound the generic spectrum contribution
\begin{equation}\label{genericint}
\peter{\vPhi_{2,u},\vPhi_{3,u}}_{\mathrm{generic}}=\int_{\pi'\ \mathrm{ generic}}\peter{\Pro^{\pi'}_{\Pi_2}(\vPhi_{2,u}),{\Pro_{\Pi_3}^{\pi'}(\vPhi_{3,u})}} d\mu_{\mathrm{Pl}}(\pi').
\end{equation}

In this paper we will not try to exploit cancellations between the terms of the $\pi'$-integral and will rather focus on controlling the length of spectral sum and individual terms.

%

\subsubsection{Triple product L-functions and Quantum unique ergodicity}\label{sectripleQUE}
We first look at the decay properties of the inner product
$$\peter{\Pro^{\pi'}_{\Pi_2}(\vPhi_{2,u}),{\Pro_{\Pi_3}^{\pi'}(\vPhi_{3,u})}}=\sum_{\psi\in\mcB(\pi')}
\Pro^{\pi'}_{\Pi_2}(\vPhi_{2,u})(\psi)\ov{\Pro_{\Pi_3}^{\pi'}(\vPhi_{3,u})(\psi)}.$$
Consider the case $u=1$, then
$$\Pro^{\pi'}_{\Pi_i}(\vPhi_{i,u})(\psi)=\int_{[\rmG]}\psi(g)|\vphi_i(g)|^2dg=\mu_{|\vphi_i|^2}(\psi )$$
where $\mu_{|\vphi_i|^2}$ is the probability measure with density $|\vphi_i(g)|^2dg$. The convergence of this measure to the standard measure for varying $\vphi_i$ is interpreted as an equidistribution property of the "mass" of $\vphi_i$ on (quotients) $[\rmG]$: in other terms a form of Quantum Unique Ergodicity (QUE) in the sense of Rudnick-Sarnak \cite{RS}. QUE has now been established in a number of contexts by different methods \cite{LuoSar,Sarnak,Lin,HolLin,ANP,YH18}.
Our present problem is a slight variant of the above. The comment justifies the choice of notations in Definition \ref{Defn:eqindex} below.


By \eqref{Lperiod} and bounding terms individually, we obtain that
\begin{align}\label{Eq:uppergeneric}
&\left|\peter{
\Pro^{\pi'}_{\Pi_2}(\vPhi_{2,u}),
{\Pro_{\Pi_3}^{\pi'}(\vPhi_{3,u})}}\right|\leq \sum\limits_{\psi\in \mcB(\pi')}\left|
\Pro^{\pi'}_{\Pi_2}(\vPhi_{2,u})(\psi)\right|\left|
\Pro^{\pi'}_{\Pi_3}(\vPhi_{3,u})(\psi)\right| \\
=&(C(\pi')Q)^{o(1)}\sum\limits_{\psi\in \mcB(\pi')}\frac{L(\pi'\ootimes\Pi_2,1/2)^{1/2}}{C(\pi'\ootimes\Pi_2)^{1/8}}\prod_v {\left|
\Pro^{\pi'}_{\Pi_2}(\vPhi_{2,u,v})(\psi_v)\right|}{C_v(\pi'\ootimes\Pi_2)^{1/8}} \notag\\
&\times\frac{L(\pi'\ootimes\Pi_3,1/2)^{1/2}}{C(\pi'\ootimes\Pi_3)^{1/8}}\prod_v {\left|
\Pro^{\pi'}_{\Pi_3}(\vPhi_{3,u,v})(\psi_v)\right|}{C_v(\pi'\ootimes\Pi_3)^{1/8}}	\notag\\
=&(C(\pi')Q)^{o(1)}\left(\frac{L(\pi'\ootimes\Pi_2,1/2)L(\pi'\ootimes\Pi_3,1/2)}{C(\pi'\ootimes\Pi_2)^{1/4}C(\pi'\ootimes\Pi_3)^{1/4}}\right)^{1/2} \notag\\
&\times\sum\limits_{\psi\in \mcB(\pi')}\prod_v \left|
\Pro^{\pi'}_{\Pi_2}(\vPhi_{2,u,v})(\psi_v)\right|\left|
\Pro^{\pi'}_{\Pi_3}(\vPhi_{3,u,v})(\psi_v)\right|\left(C_v(\pi'\ootimes\Pi_2)C_v(\pi'\ootimes\Pi_3)\right)^{1/8}. \notag
\end{align}

\begin{defn}\label{Defn:eqindex}
\begin{enumerate}
 \item[-]We call the quantity $$\Equid(\Pi|\pi')=\frac{L(\pi'\ootimes\Pi,s)}{C(\pi'\ootimes\Pi)^{1/4}}$$  the {\em global equidistribution index} of $\Pi$ with respect to $\pi'$.
\item[-]Given  some test vector $\vPhi\in \Pi$ we define the {\em local equidistribution index}
$$\equid_{v}(\vPhi|\psi)=
\left|\Pro^{\pi'_v}_{\Pi_v}(\vPhi_{v})(\psi_v)\right|^2C_v(\pi'\ootimes\Pi)^{1/4}.$$

\item[-]For any $S\subset \mcV$ we define
$$\equid_S(\vPhi|\psi):=\prod_{v\in S}\equid_{v}(\vPhi|\psi).$$
\item[-]For
$\Pi_{23}=\Pi_2\boxplus\Pi_3$, define $L(\pi'\ootimes\Pi_{23},s)=L(\pi'\ootimes\Pi_2,s)L(\pi'\ootimes\Pi_3,s)$ and similarly $\ C(\pi'\ootimes\Pi_{23})$, $\Equid(\Pi_{23}|\pi')$.
For test vector $\vPhi_{23}=\vPhi_2\boxplus\vPhi_3\in \Pi_{23},$ denote $e_S(\vPhi_{23}|\psi)=e_S(\vPhi_{2}|\psi)e_S(\vPhi_{3}|\psi)$.
\end{enumerate}

\end{defn} 

 \begin{rem}

At places $v$ where $\vPhi$ and $\psi$ are spherical, one has
$$\equid_{v}(\vPhi|\psi)=1.$$
Thus the product $\equid_S(\vPhi|\psi)$ is always finite. 
 \end{rem}
  
Now take the vector $$\vPhi_{23,u}:=\vPhi_{2,u}\boxplus\vPhi_{3,u}\in \Pi_{23}$$ (cf. \eqref{Phi23udef}) for $u\in\supp(\ov\sigma\star\check\sigma)$. Using the above notations, we can rewrite \eqref{Eq:uppergeneric} as
\begin{equation}\label{Eq:uppergeneric2}
\left|\peter{\Pro^{\pi'}_{\Pi_2}(\vPhi_{2,u}),{\Pro_{\Pi_3}^{\pi'}(\vPhi_{3,u})}}\right|\leq (C(\pi')Q)^{o(1)}\Equid(\Pi_{23}|\pi')^{1/2}\sum\limits_{\psi\in \mcB(\pi')}\equid_\mcV(\vPhi_{23,u}|\psi)^{1/2},
\end{equation}
and thus
\begin{align}\label{genericupper}
&\left|\peter{\vPhi_{2,u},\vPhi_{3,u}}_{\mathrm{generic}}\right| \\
\leq & \notag (C(\pi')Q)^{o(1)}\int_{\pi'\ \mathrm{ generic}}\Equid(\Pi_{23}|\pi')^{1/2}\sum\limits_{\psi\in \mcB(\pi')}\equid_\mcV(\vPhi_{23,u}|\psi)^{1/2} d\mu_{\mathrm{Pl}}(\pi').
\end{align}

\subsubsection{The global equidistribution index}

By the convexity bound for L-functions, we have
\begin{equation}\label{Eq:globalEqindex}
\Equid(\Pi_{23}|\pi')\ll_\epsilon C(\pi'\ootimes\Pi_{23})^{\epsilon}\ll_\epsilon  Q^\epsilon.
\end{equation} 
 The Generalized Lindel\"of hypothesis predicts that it is much smaller, controlled by $C(\pi'\ootimes\Pi_{23})^{-1/4+o(1)}$. The resolution of the subconvexity problem for $L(\pi'\ootimes\Pi_2,1/2)L(\pi'\ootimes\Pi_3,1/2)$ is equivalent to
$$\Equid(\Pi_{23}|\pi')\ll C(\pi'\ootimes\Pi_{23})^{-\delta}$$
for some absolute constant $\delta>0$. In some special cases
this bound is known (for example when $\pi_2$ or $\pi_3$ is either an Eisenstein series representation or a dihedral representation.)
Our result uses only \eqref{Eq:globalEqindex} and does not rely on this saving, but can be greatly improved if it is known.

\subsubsection{The generic spectral length}

For a factorable automorphic form $\varphi$, let $\ram(\varphi)$ denote the set of finite places $v$ where $\vphi_{v}$ is not  spherical.

For each place $v$, let
$$M_v=\min\{C(\pi_{2,v}), C(\pi_{3,v})\},$$
$$M=\prod\limits_{v}M_v,\  M_{\fin}=\prod\limits_{v\nmid \infty}M_v.$$
 
\begin{assumption}\label{Assumption:Main}
 For each $v|M_\fin\infty$, let $j_v=2$ or $3$ so that $C(\pi_{j_v,v})=M_v$. 
We assume the following 
\begin{enumerate}
\item For $v|M_{\fin}$,  $\varphi_{j_v,v}$   is $K(M'_v)-$invariant where $$M'_v=M_v^A $$ for some absolute constant $A$.
\item For $v|\infty$,  we assume that $\mcS^{\pi_{j_v,v}}_d\left(\varphi_{j_v,v}\right)\ll M_v^{A(d)}$ for some constants $A(d)$ depending only on $d$.
\end{enumerate}
\end{assumption}
Note that this assumption will be incorporated into Conjecture \ref{Conj:localresults} later on.

To simplify the notations, we shall omit subscript $v$ from $j$ from now on, though it can change for different $v$.

\begin{lem}\label{Lem:shortgenericlength} 
 Let $M'=\prod\limits_{v\text{\ finite}}(\|u\|_v M_v)^A $ for some
absolute constant $A$.
With Assumption \ref{Assumption:Main} above, we have
\begin{align}\label{genericupper2}
&\left|\peter{\vPhi_{2,u},\vPhi_{3,u}}_{\mathrm{generic}}\right|\\
\ll_\epsilon &\notag (LQ)^{\epsilon}\int_{\pi'\ \mathrm{ generic}}\Equid(\Pi_{23}|\pi')^{1/2}\sum\limits_{\psi\in \mcB((\pi')^{K(M')})}\equid_\mcV(\vPhi_{23,u}|\psi)^{1/2} d\mu_{\mathrm{Pl}}(\pi')\\
\ll_\epsilon & Q^\epsilon (LM)^A  \sup\limits_{\pi',\psi
}\Equid(\Pi_{23}|\pi')^{1/2}\equid_\mcV(\vPhi_{23,u}|\psi)^{1/2}, \notag
\end{align}
where the integral/supremum in   $\pi'$ is for those $\pi'$ with $C(\pi'_v)$ bounded by $(\|u\|_v  M_v)^A$ for all $v$;
The sum/supremum in $\psi$ is over orthonormal elements in $\pi'$ such that $\psi_v$ at finite places is  $K(M_v')-$invariant, and $\psi_v$
 at archimedean places satisfies $\mcS^{\pi_{v}'}_d\left(\psi_{v}\right)\ll M_v^{A(d)} $.
\end{lem}
\begin{proof}
%

At archimedean places, the local period integral $\Pro^{\pi'}_{\Pi_j}(\vPhi_{j,v})(\psi_v)$ is rapidly decaying by \cite[Lemma 6.9, Section 6.11]{Ne19} with respect to the Sobolev norm of $\psi_v$. Thus
up to a factor of size $Q^\epsilon M^A$, one can restrict the summation in $\psi$ to those with $\mcS^{\pi'_v}_d(\psi_v)\ll M_v^{A(d)}$, and the integral in  $\pi'$ to those with  analytic conductor at $v$ bounded by $M_v^A$ according to \cite[Lemma 6.6]{Ne19}. The dimension of such $\psi\in \pi'$ are controlled by  $\prod\limits_{v|\infty}M_v^A$ according to \cite[Lemma 2.6.3]{MVIHES}.

At non-archimedean places, the restriction on ramification comes from that
$ \Pro^{\pi'}_{\Pi_j}(\vPhi_{j,u,v})(\psi_v)$ would be vanishing when $\vPhi_{j,u,v}$ is $K(M'_v)-$invariant
while $\psi_v$ is in the orthogonal complement of  $(\pi_v')^{K(M'_v)}$, the subspace of $K(M'_v)-$invariant elements.



The  second line of \eqref{genericupper2} follows from the Weyl law with the bounded local conductors for $\pi'$ and $\psi$.
\end{proof}
\begin{example}
Consider the simplest case  where $\varphi_3$ is spherical. Then one only has to consider those $\pi'$/$\psi$ in Lemma \ref{Lem:shortgenericlength} 
whose ramifications at finite places happen only at $v$ where $u_v\neq 1$, and are controlled there by $\|u_v\|^{A}$.
\end{example}

\subsubsection{The local equidistribution index} 

Take $S=\ram(\vPhi_{23,u})$.  For relevant $\pi'$ and $\psi$ as in Lemma \ref{Lem:shortgenericlength}, we have
 $$\equid_{\mcV}(\vPhi_{23,u}|\psi)= \equid_{S}(\vPhi_{23,u}|\psi).$$

 One has the relation
$\ram(\vPhi_{23,u})=\ram(\vPhi_{23})\sqcup \{v,\ u_v\not=1\}$.
The study of $e_v(\vPhi_{23,u}|\psi)$ at $v\in \ram(\vPhi_{23})$ will be left to later sections. Here we control the local equidistribution index at $v$ where $u_v\neq 1$.

\begin{lem} Let $v$ be such that $u_v\not=1$, then for $\psi$ as in Lemma \ref{Lem:shortgenericlength},
$$\equid_{v}(\vPhi_{23,u}
|\psi)\ll C_v(\pi')^2\|u_v\|^{A'}$$
where the implied constant and $A'$ are absolute.
\end{lem}
\proof 
Since $u_v\not=1$, $\Pi_2$ and $\Pi_3$ are unramified and 
$$C_v(\pi'\ootimes\Pi_2)C_v(\pi'\ootimes\Pi_3)= C_v(\pi')^{8}.$$

 Recall that 
$$ \vPhi_{2,u}(g_1,g_2)= u.\vphi_2(g_1)\times\ov\vphi_2(g_2),\vPhi_{3,u}(g_1,g_2)=\vphi_3(g_1)\times u.\ov\vphi_3(g_2) .$$

It then follows from \cite[Lemma 3.5.2]{MVIHES}  and the distortion property for Sobolev norms \cite[S1b]{MVIHES} that for some absolute constants $A$ and $A'$
$$\|{\Pro^{\pi'}_{\Pi_2}(\vPhi_{2,u,v})\|\times\|\Pro^{\pi'}_{\Pi_3}(\vPhi_{3,u,v})\|}\ll\mcS^{\pi_{2,v}}_A(\vphi_2^\circ)\mcS^{\pi_{2,v}}_A(u_v.\vphi_2^\circ)\mcS^{\pi_{3,v}}_A(\vphi_3^\circ)\mcS^{\pi_{3,v}}_A(u_v.\vphi_3^\circ)\ll\|u_v\|^{A'}.$$
The lemma follows by dropping all but one term.
\qed

From this bound and Lemma \ref{Lem:shortgenericlength}  on the restriction for $\pi'$, we deduce that
\begin{equation}\label{equidlocfirstbound}
	\equid_\mcV(\vPhi_{23,u}|\psi)
\ll \|u\|^{A}\equid_{\ram(\vPhi_{23})}(\vPhi_{23}|\psi)
\end{equation}
for some absolute constant $A$.

\subsection{The main global bound}


By combining \eqref{Lperiod}, \eqref{Cauchy}, Lemma \ref{amplifier}(4), \eqref{finitecontribution}, Lemma \ref{Lem:shortgenericlength} and \eqref{equidlocfirstbound}
we obtain after integration over $u$ that there exist absolute constant $\delta,A>0$ 
 such that
\begin{align*}
&L(\pi\times\pi_2\times\pi_3,1/2)\\
\ll_\epsilon &\frac{C(\pi_1\times\pi_2\times\pi_2)^{1/4+\epsilon}}{\Feas(\vphi_2\otimes\vphi_3|\pi)}\left( L^{-\delta}+(LM)^{A}\sup\limits_{\pi',\psi
}\Equid(\Pi_{23}|\pi')^{1/2}\equid_{\ram(\vPhi_{23})}(\vPhi_{23}|\psi)^{1/2}\right).
\end{align*}
Here the supremum  is taken over $\pi'$ and $\psi$ as in Lemma \ref{Lem:shortgenericlength}. 

Choosing $L>1$ optimally if
\begin{equation}\label{Eq:EE}
\E(\vPhi_{23}):=\sup\limits_{\pi',\psi
}\Equid(\Pi_{23}|\pi')^{1/2}\equid_{\ram(\vPhi_{23})}(\vPhi_{23}|\psi)^{1/2}\leq  1,
\end{equation}
 and $L=1$ otherwise, we obtain (under the assumption that $\pi_1,\pi_2,\pi_3$ are all cuspidal) the following result:
\begin{theo}\label{metasubconvexbound} 
Under Assumption \ref{Assumption:Main}, there exist absolute constants $\delta, A>0$ such that 
$$L(\pi\times\pi_2\times\pi_3,1/2)\ll_\epsilon M^AC(\pi_1\times\pi_2\times\pi_2)^{1/4+\epsilon} \frac{1}{\Feas(\vphi_2\otimes\vphi_3|\pi)}\min\left(1,\E(\vPhi_{23})^\delta\right).$$
\end{theo}

\subsection{The Rankin--Selberg case}
In this paper we are also interested in the critical value of the Rankin--Selberg L-function $L(\pi_1\times\pi_2,1/2)$, in which case  $\varphi_3$ is an Eisenstein series in the above formulation.
In this section we show that Theorem \ref{metasubconvexbound} continue to hold in this case. 
We use the regularization process as developed in \cite{MVIHES}, and refer to \S 4.3 of that paper for the detailed definitions and notations.

We assume for simplicity that $\pi_3=1\boxplus\chi_3$ (with $\chi_3=(\chi_1\chi_2)^{-1}$, $\chi_i$ being central character of $\pi_i$). We have
$$L(\pi_1\times\pi_2\times\pi_3,1/2)=\left|L(\pi_1\times\pi_2,1/2)\right|^2$$
$$\eps(\pi_1\times\pi_2\times\pi_3)=\left|\eps(\pi_1\times\pi_2)\right|^2=1$$ 
so that $\rmG=\GLd$ and the Jacquet-Langlands correspondence is just the identity.

We take $\vphi_3$ of the shape
$$\vphi_3=\Eis(f_3)$$
for $f_3:\rmN(\Aa)B(F)\bash\GA\ra\Cc$ an element in $\pi(1,\chi_3)$. One can try to proceed exactly as in the beginning of this section up to \S \ref{Plancherelsection} and the evaluation of linear combination of inner products of the shape $\peter{\vPhi_{2,u},\vPhi_{3,u}}$. However since $\vphi_3$ is not rapidly decreasing we have to use the regularized version of Plancherel's formula proven in \cite[\S 4.3]{MVIHES}. In order to apply this formula, we have to deform slightly the vector $\vPhi_{3,u}$:
given $t\in\Cc$ let
$$\pi_3(t)
=|\cdot|^{it}\boxplus|\cdot|^{-it}\chi_3
,\ 
\vphi_3(t)=\Eis(f_3(t)),$$ 
where $f_3(t)$ is the section whose restriction to the maximal compact subgroup $\prod_v K_v$ coincide with $f_3$. 
Set $\Pi_3(t)=\pi_3(t)\otimes\tilde\pi_3$ and define the vector
\begin{align*}
\vPhi_{3,u}(t):(g_1,g_2)\mapsto \vphi_3(t)(g_1)\times u.\ov\vphi_3(g_2).
\end{align*}
As $\varphi_2$ is cuspidal, we observe that the function
$$t\mapsto \peter{\vPhi_{2,u},\vPhi_{3,u}(t)}$$
is anti-holomorphic in the $t$-variable near $t=0$, and to bound the value at $t=0$ it is sufficient to bound it uniformly on some circle $\mcC$ of fixed radius around $t=0$.

The set of exponents of $\vPhi_{3,u}(t)$ is $$T=\{|\cdot|^{1+it}_\Aa,|\cdot|^{1-it}_\Aa\chi_3,|\cdot|^{1+it}_\Aa\ov\chi_3,|\cdot|^{1-it}_\Aa\}$$
and the squares of the above exponents are not $|\cdot|^2_\Aa$ for $t\in \mcC$ for almost all radius of $\mcC$. We fix such $\mcC$ with small enough radius in the following.

Then as in the regularized version of Plancherel's formula \cite[Prop. 4.3.8]{MVIHES}  we have the following:
 \begin{align}
\peter{\vPhi_{2,u},\vPhi_{3,u}(t)}
&=\peter{\vPhi_{2,u},\mcE_3(t)}+\int\limits_{\pi' \ \mathrm{generic}
}
\sum_{\psi\in\mcB(\pi')}
\peter{\vPhi_{2,u},\psi}\peter{\psi,\vPhi_{3,u}(t)}_{reg}\ d\mu_{\mathrm{Pl}}(\pi'),
 \end{align}
 which we rewrite as
 $$\peter{\vPhi_{2,u},\vPhi_{3,u}(t)}:=\peter{\vPhi_{2,u},\vPhi_{3,u}(t)}_{\mathrm{deg}}+\peter{\vPhi_{2,u},\vPhi_{3,u}(t)}_{\mathrm{generic}}.$$
In these expressions,
 $$\mcE_3=\Eis(\vPhi_{3}(t)_N^*)$$
 is the Eisenstein series formed out of the exponents of $\vPhi_{3,u}(t)$, $\peter{\psi,\vPhi_{3,u}(t)}_{reg}$ denote the regularized inner product (which is the regular inner product if $\pi'$ is cuspidal).

 The term
 $$\peter{\vPhi_{2,u},\vPhi_{3,u}(t)}_{\mathrm{deg}}=\peter{\vPhi_{2,u},\mcE_3(t)}$$
 is called the  "degenerate term"; in a sense it replaces the
finite spectrum contribution which vanishes for $t\neq 0$ in this formula \cite[\S 5.2.7]{MVIHES}. 	
Indeed the degenerate term can be bounded by the method of \cite[Lemma 5.2.9]{MVIHES}, satisfying
\begin{equation}\label{degeneratebound}
\peter{\vPhi_{2,u},\vPhi_{3,u}(t)}_{\mathrm{deg}}\ll (C(\pi_2)C(\pi_3))^{o(1)}\|u\|^{-\gamma}	
\end{equation}
for some absolute constant $\gamma>1/2$, similar to the finite spectrum. One can then control its contribution similarly as in \eqref{finitecontribution}. 

To be self-contained we briefly describe the argument for  \eqref{degeneratebound} here. The degenerate term can be unfolded and related to  the Rankin--Selberg L-function for each Eisenstein series formed out of the exponents. For example the exponent $|\cdot|_{\A}^{1+it}\in T$ gives rise to a product of  $L(\pi_2\times \tilde{ \pi}_2, 1+t)$ with local period integrals at finite number of places, and some other unimportant factors. It suffices to control them separately. The $L-$function factorizes as $L(\pi_2,\Ad, 1+t)\zeta(1+t)$, and  $L( \pi_2,\Ad, 1+t)\ll C(\pi_2)^{o(1)}$ near $t=0$, while $\zeta(1+t)$ (despite having a pole at $t=0$) can be uniformly bounded on  fixed circle $\mcC$ round $t=0$, independent of $\pi_2$.

On the other hand the local period integrals can be bounded by $\|u\|^{-\gamma}$ as in  \cite[Section 5.2.10]{MVIHES}. Indeed for $v\in \ram(\vPhi_{23})$, the bound in, for example, \cite[(5.22)]{MVIHES} holds uniformly regardless of the ramifications.

%
%
%
%
%

Using \cite[Lemma 4.4.3]{MVIHES}, the generic contribution can be treated exactly as in \S \ref{secgeneric}, and  the analogues of Lemma \ref{Lem:shortgenericlength} and \eqref{equidlocfirstbound} hold.
Hence Theorem \ref{metasubconvexbound}  also hold in the Rankin--Selberg case.

\subsection{Conjecture on test vectors and subconvexity bound}
To further obtain subconvexity bound from Theorem \ref{metasubconvexbound}, we first formulate a local conjecture on the existence of proper test vectors.

Let $v$ be a place of $F$.
Recall that $\left|\Pro_{\Pi_v}^{\pi_v}\right|^2$ is either a local integral of product of matrix coefficients if all the global representations are cuspidal, or the absolute value squared of local Rankin--Selberg integral if one of the representations is an Eisenstein serie.
Recall that $$Q_v=C_v(\pi_1\times\pi_2\times\pi_3),\ M_v=\min\{C(\pi_{2,v}), C(\pi_{3,v})\}.$$
Recall the notation $\lleq_{v,a}$ and relevant results from Section \ref{Section:Inequalityconvention}.
\begin{defn}
Denote $$P=\prod_v P_v=
\prod_v \frac{C_v(\pi_1\times \pi_2\times \pi_3)^{1/2}}{\max\limits_{i=2,3}\{ C_v(\pi_i\times \pi_i) \}}. 
$$
\end{defn}
\begin{rem}\label{Rem:P1}
One can check case by case that $$\max\limits_{i=2,3}\{ C_v(\pi_i\times \pi_i)
\}
\leq C_v(\pi_1\times \pi_2\times \pi_3)^{1/2}
,$$
so $P_v\geq 1$ is always true.
\end{rem}

%


\yhn{
For $\varphi_{i,v}\in \pi_{i,v}$, where $\pi_{3,v}$ is a parabolically induced representation, denote by
\begin{equation}\label{Eq:LocalRS}
I_v^{\RS}\left(\varphi_{1,v},\varphi_{2,v},\varphi_{3,v}\right)=\int\limits_{{Z(\F_v) N}\backslash\GL_2{(\F_v)}}W_{\varphi_{1,v}}\left(g\right)\overline{W_{\varphi_{2,v}}\left(g\right)}\varphi_{3,v}\left(g\right)dg
\end{equation}
 the local  Rankin--Selberg integral. Here $W_{\varphi_v}$ is the Whittaker function associated to $\varphi_v$ with respect to a fixed additive character $\psi_v$. We also note that $\overline{W_{\varphi_v}}=W^-_{\varphi_v}$ where $W^-_{\varphi_v}$ is the Whittaker function for the additive character $\psi_v^-(x)=\psi_v(-x)$. 

In general for $\varphi_{i,v}\in \pi_{i,v}^B$, denote by
\begin{equation}\label{Eq:LocalT}
I_v^{\T}\left(\varphi_{1,v},\varphi_{2,v},\varphi_{3,v}\right)=\int\limits_{\F_v^\times\backslash \GL_2{(\F_v)}}\prod\limits_{i=1}^{3}\Phi_{\varphi_{i,v}}\left(g\right)dg
\end{equation}
the local integral for the triple product formula, where $\Phi_{\varphi_v}$ is the matrix coefficient associated to $\varphi_v$.

Let $$I_v=I_v^T\hbox{ or }\left|I_v^\RS\right|^2$$ depending on whether $\varphi_3$ or $\psi$ is an Eisenstein series. We formulate the following local conjecture, for which we skip subscript $v$'s.
\begin{conj}\label{Conj:localresults} Let $\nu$ be a place of $F$. We assume that if $\pi_3$ or $\pi'$ is an Eisenstein series (so that $I_v=\left|I_v^\RS\right|^2$) the corresponding local representation $\pi_{3,v}$ or $\pi_v'$ is tempered.
Then there exists normalized test vectors $\varphi_{i,v}\in\pi_{i,v}^B$ satisfying the following properties:
\begin{enumerate}
\item[(0)] If $C_v(\pi_j)=M_v$ for $j=2$ or $3$, then $\varphi_{j,v}$ is $K(M_v')-$invariant if $v$ is non-archimedean, and 
 $\mcS^{\pi_{j,v}}_d\left(\varphi_{j,v}\right)\ll M_v^{A(d)}$ if $v$ is archimedean as in Assumption \ref{Assumption:Main}.
\item  There exists some constant $A$ such that $$I(\varphi_{1,v},\varphi_{2,v},\varphi_{3,v}) \ggeq_{\epsilon}
Q_v^\epsilon M_v^A\frac{1}{Q_v^{1/4}}, 
 $$
\item 
There exists some absolute constant $A$ (independent of $v$) such that for  $\psi_v\in\pi_v'$ controlled by $M_v$ as in Lemma \ref{Lem:shortgenericlength}, and $j'$ the other index in $\{2,3\}$ different from $j$,
$$I_v(\varphi_{j',v},\varphi_{j',v},\psi_v)\lleq_{ \epsilon} Q_v^\epsilon M_v^A \frac{1}{\max\limits_{i=2,3}\{C_v(\pi_i\times\pi_i)\}^{1/2}}\frac{1}{P_v^{1/2-\theta}}.$$
\end{enumerate}
\end{conj}
Here $\theta<1/2$ is any bound towards the global Ramanujan conjecture. }

\begin{rem}\label{Rem:localeqindexj}
Note that for $j$ as above, we have
\yhn{$$\left|I_v(\varphi_{j,v},\varphi_{j,v},\psi_v) \right|\lleq_{v,\epsilon} M_v^AQ_v^\epsilon,$$
}
as $\pi_{j,v}$, $\varphi_{j,v}$, $\pi_v'$, $\psi_v$ are all controlled in terms of $M_v$.
\end{rem}

From  the definitions of feasibility index and equidistribution indices in Definition \ref{Defn:feasindex}, \ref{Defn:eqindex}, we immediately obtain the following:
\begin{theo}\label{Theo:mainSubbound}
Suppose that Conjecture \ref{Conj:localresults} is true. Then there exists  constants $\delta, A>0$  such that
\begin{equation}\label{Eq:Mainbound}
L(\pi_1\times\pi_2\times\pi_3,1/2)\ll_{\epsilon}M^{A} C(\pi_1\times\pi_2\times\pi_3)^{1/4+\epsilon} \frac{1}{P^\delta}.
\end{equation}

\end{theo}
\begin{rem}
In particular this is a subconvexity bound if $M\ll_\epsilon Q^{\epsilon} $ and there exists a constant $\gamma>0$ such that 
$$\prod_v\max\limits_{i=2,3}\{ C_v(\pi_i\times \pi_i)\}\leq M^{A}C(\pi_1\times\pi_2\times \pi_3)^{1/2-\gamma}.$$
\end{rem}

\begin{proof}[Proof of Theorem \ref{Theo:mainSubbound}]
\yhn{We first remark that $\left|\Pr_{\Pi_v}^{\pi_v}\right|^2$ differs from $I_v$ only by some absolute constant and normalizing L-factors, which can be controlled by $Q^\epsilon$ in the sense of Example \ref{Ex:control}. So we shall not distinguish them below.
}

The temperedness condition in  Conjecture \ref{Conj:localresults} is satisfied as when $\varphi_3$/$\psi$ is a unitary Eisenstein series in the Rankin--Selberg case, the associated local component $\pi_{3,v}$/ $\pi'_v$ is indeed tempered.

Let $j,j'$ be as in Conjecture \ref{Conj:localresults}.
For the test vectors $\vphi_i$ which are specified in Definition \ref{Defn:testvec1} and satisfy Conjecture \ref{Conj:localresults}, it follows from Theorem \ref{metasubconvexbound} (which requires item (0) of Conjecture \ref{Conj:localresults}), Definition \ref{Defn:feasindex} and item (1) of Conjecture \ref{Conj:localresults}  that
$$ 
L(\pi\times\pi_2\times\pi_3,1/2)\ll_\epsilon M^AC(\pi_1\times\pi_2\times\pi_2)^{1/4+\epsilon} \min\left(1,\E(\vPhi_{23})^\delta\right).
$$

For $\E(\vPhi_{23})$ as in \eqref{Eq:EE}, using Definition \ref{Defn:eqindex}, bound \eqref{Eq:globalEqindex}, item (2) of Conjecture \ref{Conj:localresults} and Remark \ref{Rem:localeqindexj},
we get
$$\E(\vPhi_{23})\ll_{\epsilon}M^A Q^\epsilon \frac{1}{P^{(1/2-\theta)\delta}}.$$
Then the theorem follows for a different $\delta>0$ as $P\geq 1$ by Remark \ref{Rem:P1}.
\end{proof}

\begin{rem}
Our formulation also implies that if Conjecture \ref{Conj:localresults} is true and the subconvexity bound for $L(\pi_{j'}\times\pi_{j'}\times\pi',1/2)$ holds for controlled $\pi'$, then the subconvexity bound for $L(\pi_1\times\pi_2\times \pi_3,1/2)$ always holds up to a factor $M^A$.
\end{rem}

The remaining  of this paper is devoted to  partially verifying Conjecture \ref{Conj:localresults}. In particular we prove the following result.
\begin{theo}\label{Theo:conjecturecase}
Suppose that 
 $\pi_i$ have trivial central characters with bounded archimedean components, and $M_{\fin}=1$. Then Conjecture \ref{Conj:localresults} 
 is true. 
\end{theo}

\begin{rem}
One can immediately obtain the same result if each $\pi_i$ is further  twisted by some character $\eta_i$ as long as $\prod \eta_i=1$, allowing more general central characters.

In general there are some sporadic cases (for example, in specific archimedean aspect, or for square-free $M_\fin$) where one can also verify the conjecture. For the conciseness of this paper we limit ourselves to the  case  in Theorem \ref{Theo:conjecturecase}.
\end{rem}

\section{Local preparations}\label{Sec:localPrep}
The remaining sections are purely local. For the sake of conciseness we shall omit subscript $v$ from the notations.
\subsection{Basics}
We collect some basic definitions and results here.

Let $F$ be a p-adic local field, with ring of integers $\OF$, uniformizer $\varpi$, valuation $v$,  residue field  $k_F$,  $q=|k_F|$ and $p$ be the characteristic of $k_F$. 
Let $E$ be an \'{e}tale quadratic algebra over $F$. When $E$ is a field, we denote $O_E$ to be the ring of integers, $\varpi_E$ to be a uniformizer, and $e_E$ to be the ramification index of $E$. We normalize the valuation on $E$ such that $v_E(\varpi_E)=1$. Define $U_E(i)=1+\varpi^iO_E$.
 To be  uniform, we also take the following conventions when working with the  case where $E$ splits.
\begin{defn}\label{Def:splitconvention}
Let $E$ be a split quadratic extension over $F$ identified with $F\times F$. Take $e_E=1$ in this case. Define $O_E=\OF\times \OF$, $O_E^\times=\OF^\times\times \OF^\times$, $U_E(i)=1+\varpi^iO_E$ for $i>0$. Also we write $v_E(x)=i$  when $x\in \varpi^iO_E^\times$. (In particular $v_E$ is not defined for all $E$ in this case.)
\end{defn}

For a non-trivial additive character $\psi$ over $F$, denote
$$ c(\psi)=i,$$ if  $i\geq 0$ is the smallest integer such that $\psi|_{\varpi_F^iO_F}=1.
$
This definition also works for additive characters over $E$. 
For a multiplicative character $\chi$ of $F^\times$, define
$$ c(\chi)=\begin{cases}
0, &\text{\ if }\psi(O_F^\times)=1;\\
i, &\text{\ if $i>0$ is smallest integer such that }\chi|_{U_F(i)}=1.
\end{cases}$$

Let $\psi$ now be an additive character over $F$ with $c(\psi)=0$. 
  Denote $\psi_E=\psi\circ \Tr_{E/F}$. Then $c(\psi_E)=-e_E+1$. For a multiplicative character $\chi$ on $F^\times$, we can associate a character $\chi_E=\chi\circ N_{E/F}$ on $E^\times$.

For multiplicative characters over $F^\times$, we introduce an equivalence relation here.
\begin{defn}
For any two characters $\chi_i$ of $F^\times$, $i=1,2$, 
$\chi_1\sim_j \chi_2$ if and only if $c(\chi_1\chi_2^{-1})\leq j$.
\end{defn}
We collect the following straightforward lemmas.
\begin{lem}\label{Lem:differentE}
Suppose that $2\nmid q$.
Let $E,E'$ be two non-isomorphic \'{e}tale quadratic algebra over $F$, $\alpha_1\in E$, $\alpha_2\in E'$ be trace 0 elements with $v(N_{E/F}(\alpha_1))=v(N_{E/F}(\alpha_2))$. Let $\alpha_0\in F$ with $v(N_{E/F}(\alpha_1))\leq v(N_{E/F}(\alpha_0))$. Then we have
$$\frac{N_{E/F}(\alpha_1+\alpha_0)}{N_{E'/F}(\alpha_2+\alpha_0)}\not\equiv 1 \mod \varpi.$$
\end{lem}
\begin{proof}
Since $\alpha_i$ are trace 0 for $i=1,2$ and $\alpha_0\in F$, we have
$$N(\alpha_i+\alpha_0)=N(\alpha_i)+N(\alpha_0).$$
By the condition $v(N_{E/F}(\alpha_1))\leq v(N_{E/F}(\alpha_0))$, we are reduced to prove $$\frac{N_{E/F}(\alpha_1)}{N_{E'/F}(\alpha_2)}\not\equiv1  \mod \varpi.$$ This follows directly from that $E,E'$ are not isomorphic and $2\nmid q$.
\end{proof}

\begin{lem}\label{lem:Gaussint}
Let $m\in F$ such that $v(m)=-j<0$, and $\mu$ be a character of $\OF^\times$ of level $k>0$. Then
\begin{equation}
  \left|\int\limits_{v(x)=0}\psi(mx)\mu^{-1}(x)d^\times x\right|=\begin{cases}
                                                   \sqrt{\frac{q}{(q-1)^2q^{k-1}}}=\zeta_F(1)q^{-k/2},&\text{\ if\ }j=k;\\
                                                   0,&\text{\ otherwise.}
                                                  \end{cases}
\end{equation}
\end{lem}
\begin{lem}\label{Lem:level1cancellation}
Let $\theta$ be a character over an \'{e}tale quadratic algebra  $E/F$, such that $e_E=1$, $c(\theta)=1$ $\theta|_{F^\times}=1$.
Then for any trace 0 element $\alpha$ with $v_E(\alpha)=0$,
$$\left|\sum\limits_{x\in k_F^\times, x+\alpha \in O_E^\times}\theta(x+\alpha)\right|\leq 2.$$
Here we have identified $O_F^\times/U_F(1)$ with $k_F^\times$ without confusion.
\end{lem}
\begin{proof}
Consider the case $E$ is a field first. 
As $c(\theta)=1$, we have
$$\sum\limits_{x,y\in k_F, (x,y)\neq (0,0)}\theta(x+y\alpha)=0.$$
 Using that $\theta|_{F^\times}=1$, we get that
$$\sum\limits_{x\in k_F^\times}\theta(x+\alpha)=\frac{1}{q-1}\sum\limits_{x,y\in k_F^\times }\theta(x+y\alpha)=-1-\theta(\alpha). $$
The claim follows.
On the other hand if $E\simeq F\times F$, $\theta=(\mu,\mu^{-1})$, $\alpha=(y,-y)$, then 
$$\sum\limits_{x\in k_F^\times, x+\alpha \in O_E^\times}\theta(x+\alpha)=\sum\limits_{x\in k_F^\times, x\not\equiv y,-y}\mu\lb \frac{x+y}{x-y}\rb=\sum\limits_{z\in k_F^\times, z\not\equiv \pm 1} \mu(z)=-\mu(1)-\mu(-1).$$
Again the claim follows.
\end{proof}
\begin{lem}\label{Lem:DualLiealgForChar}

For a character
  $\mu$ over $\F^\times$ with $c(\mu)\geq 2$, there exists $\alpha_\mu\in \F^\times $ with $v_\F(\alpha_\mu)=-c(\mu)+c(\psi_\F)$ such that:
\begin{equation}
 \mu(1+u)=\psi_\F(\alpha_\mu u) \text{\ \ for any $u\in \varpi_\F^{\lceil c(\mu)/2\rceil} O_\F$}.
\end{equation}

\end{lem}
Here we keep track of $c(\psi_F)$ so we can also apply the lemma directly to a character $\theta$ defined over an \'{e}tale quadratic algebra $E$. Also note that if $\theta $ is a character over an \'{e}tale quadratic algebra $E/F$ such that $ \theta|_{F^\times}=1$, then the  associated element $\alpha_\theta$ can be chosen as a trace 0 element in $E$.
\begin{cor}\label{Cor:DualLiealgForBasechangeChar}
Let $E/F$ be an \'{e}tale quadratic algebra. $\chi$ is defined over $F^\times$ with $c(\chi)\geq 2$,  and  is associated to $\alpha$ by $\chi(1+x)=\psi(\alpha x)$ for $v(x)\geq c(\chi)/2$ and $v(\alpha)=-c(\chi)$. 

Then $\alpha$ is also associated to $\chi_E$, in the sense that
$$\chi\circ N_{E/F}(1+y)= \psi_E(\alpha y), \forall y\in \varpi_E^{\lceil \frac{e_Ec(\chi)-e_E+1}{2} \rceil}O_E.$$
\end{cor}
Note that when $p\neq 2$, $$c(\chi_E)=e_Ec(\chi)-e_E+1.$$
\begin{lem}\label{Lem:GaussintTwist}
Let $\mu$, $\chi$ be multiplicative characters on $F^\times$ and $\psi$ be an additive character on $F$.
Suppose that $c(\mu)\geq 2 c(\chi)$, and $\alpha_\mu$ is associated to $\mu$ by Lemma \ref{Lem:DualLiealgForChar}. Then
\begin{equation}
\int\limits_{v(x)=-c(\mu)+c(\psi)}\mu(x)\chi(x)\psi(x)d^\times x=\chi(-\alpha_\mu)\int\limits_{v(x)=-c(\mu)+c(\psi)}\mu(x)\psi(x)d^\times x.
\end{equation}
\end{lem}
This result follows immediately from the p-adic stationary phase analysis.

Let $K_1(\varpi^c)$ denote the compact subgroup of $\GL_2$ whose elements are congruent to $\zxz{*}{*}{0}{1}\text{mod}{(\varpi^c)}$.
Now we record some basic facts about integrals on $\GL_2(F)$. 
\begin{lem} \label{Iwasawadecomp}
For every positive integer $c$,
$$\GL_2(F)=\coprod\limits_{0\leq i\leq c} B\zxz{1}{0}{\varpi^i}{1}K_1(\varpi^c).$$
Here $B$ is the Borel subgroup of $\GL_2$. 
\end{lem}

We normalize the Haar measure on $\GL_2(\F)$ such that its maximal compact open subgroup $K$ has volume 1. Then we have the following easy result (see, for example, \cite[Appendix A]{YH13}).
\begin{lem}\label{localintcoefficient}
Locally let $f$ be a $K_1(\varpi^c)-$invariant function, on which the center acts trivially. Then
 \begin{equation}
  \int\limits_{F ^\times\backslash\GL_2(\F )}f(g)dg=\sum\limits_{0\leq i\leq c}A_i\int\limits_{\F ^\times\backslash B(\F )}f\left(b\zxz{1}{0}{\varpi ^i}{1}\right)db.
 \end{equation}
 
Here $db$ is the left Haar measure on $\F ^\times\backslash B(\F )$, and
$$A_0=\frac{p}{p+1}\text{,\ \ \ }A_c=\frac{1}{(p+1)p^{c-1}}\text{,\ \ \ }A_i=\frac{p-1}{(p+1)p^i}\text{\ \ for\ }0<i<c.$$
\end{lem}
\subsection{More preparations for  Whittaker model and matrix coefficient}
Here we briefly review known results on the 
Whittaker function and matrix coefficient for newforms. 
Let $\pi$ be an irreducible smooth representation of $\GL_2(F)$ with trivial central character. In particular $\pi$ is unitary. Any element $\varphi\in \pi$  can be associated to a Whittaker function $W_\varphi$ in the Whittaker model of $\pi$. For asymptotic purposes, we assume either $2\nmid q$, or $c(\pi)$ is large enough when $2|q$.

\begin{lem}\label{Lem:Pairing}
The unitary pairings on $\pi$ can be given in the Whittaker model as follows:
$$\peter{W_1,W_2}=\int\limits_{x\in \F^\times} W_1(a(x))\overline{W_2(a(x))}d^\times x.$$
\end{lem}
In the following we give explicit formulae for Whittaker functions.
\begin{defn}
For $W\in \pi$ a Whittaker function, denote
$$W^{(j)}(\alpha)=W\left(\zxz{\alpha}{}{}{1}\zxz{1}{}{\varpi^j}{1}\right).$$
\end{defn}
\begin{rem}
If $\varphi_0$ is the newform, and $W_{\varphi_0}$ is the associated Whittaker function, then by Lemma \ref{Iwasawadecomp} and the basic properties of Whittaker newforms, $W_{\varphi_0}$ is determined by $W_{\varphi_0}^{(j)}$ for $0\leq j\leq c(\pi)$.
\end{rem}
\subsubsection{Unramified representations and Special unramified representations}

\begin{lem}\label{Lem:WiUnramified}
Suppose $\mu_i$ are unramified (that is, $c(\mu_i)=0$) and $\pi=\pi(\mu_1,\mu_2)$. Let $\varphi_0\in \pi$ be a newform and $W_{\varphi_0}$ be its associated Whittaker function normalized so that $W_{\varphi_0}(1)=1$. Then $W_{\varphi_0}$ is invariant under the maximal compact open subgroup and
\begin{equation}
W_{\varphi_0}^{(0)}(\alpha)=\begin{cases}
|\alpha|^{1/2}\frac{\mu_1(\varpi\alpha)-\mu_2(\varpi\alpha)}{\mu_1(\varpi)-\mu_2(\varpi)},&\text{if }v(\alpha)\geq 0;\\0,&\text{otherwise}.
\end{cases}
\end{equation} 
\end{lem}
\begin{rem}
Note that when $v(\alpha)\geq 0$, the numerator contains the denominator as a factor and can be canceled. In this sense the formula still holds when $\mu_1(\varpi)=\mu_2(\varpi)$. Also note that the above expression for $W_{\varphi_0}$ is not $L^2$-normalized, but differ only by a factor which can be controlled globally by $Q^\epsilon$.
\end{rem}
\begin{lem}\label{Lem:Wispecial}
Let $\pi= \sigma(\mu |\cdot|^{1/2}, \mu |\cdot|^{-1/2})$ be a special unramified representation, where $\mu$ is unramified with $\mu^2=1$.

The Whittaker function associated to the newform $\varphi_0\in \pi$ 
is given by
\begin{equation}\label{Eq:WSpecial}
W_{\varphi_0}^{(1)}\left(\alpha\right)=\begin{cases}
\mu(\alpha)|\alpha|, &\text{\ if $v(\alpha)\geq 0$;}\\
0, &\text{\ otherwise}.
\end{cases}
\end{equation}
\begin{equation}\label{Eq:WSpecial2}
W_{\varphi_0}^{(0)}\left(\alpha\right)=\begin{cases}
-q^{-1}\mu(\alpha)|\alpha|\psi(\alpha), &\text{\ if $v(\alpha)\geq -1$;}\\
0, &\text{\ otherwise}.
\end{cases}
\end{equation}

\end{lem}
Again the expressions are not $L^2$-normalized but differ only by a constant controlled by $Q^\epsilon$.
Lemma \ref{Lem:WiUnramified} and \ref{Lem:Wispecial} are well-known except \eqref{Eq:WSpecial2}, which can be obtained by reformulating \cite[Corollary 5.7]{YH13}. 
\subsubsection{Supercuspidal representation case}\label{Sec:SC}
Suppose that $p\neq 2$, or $c(\pi)$ large enough when $p=2$. Then a supercuspidal representation $\pi$ can be associated with a character $\theta$ over a quadratic field extension $E$ by  Local Langlands Correspondence, and are called dihedral.

The levels of dihedral supercuspidal representations can be associated to the levels of $\theta$ by the following relations.
\begin{enumerate}
	\item[Case 1.]$c \lb \pi \rb =2n+1$ corresponds to $e_E=2$ and $c \lb \theta \rb =2n$ .
	\item[Case 2.] $c \lb \pi \rb =4n$ corresponds to $e_E=1$ and $c \lb \theta \rb =2n$.
	\item[Case 3.] $c \lb \pi \rb =4n+2$ corresponds to $e_E=1$ and $c \lb \theta \rb =2n+1$ .
\end{enumerate}

Recall from \cite[Lemma 5.7]{HuSa:19} the following result, which is a reformulation of \cite[Lemma 3.1]{Assing} and holds actually for all dihedral supercuspidal representations.
\begin{lem}\label{Lem:WiSc}  
Let $\pi$ be a dihedral supercuspidal representation with trivial central character. Let $c=c(\pi)$.
 Denote
\begin{equation*}
C_0=\int\limits_{v_\E(u)=-c(\theta)-e_\E+1}\theta^{-1}(u)\psi_\E(u)d^\times u. 
\end{equation*}
As a function in $x$, $W^{(i)}(x)$ is supported on $v(x)=\min\{0,2i-c\}$, consisting only of level $c-i$ components (in the sense of Mellin transform), except when $i=c-1$ where it consists of level $\leq 1$ components.
In particular when $i\geq c/2$, $W^{(i)}(x)$ 
is supported on $v(x)=0$, and on the support,
\begin{equation*}
W^{(i)}(x)=C_0^{-1}\int\limits_{v_\E(u)=-c(\theta)-e_\E+1}\theta^{-1}(u) \psi\left(-\frac{1}{x}\varpi^{i}N_{\E/{F}}(u)\right)\psi_\E(u)d^\times u.
\end{equation*}

\end{lem}
The normalization by $C_0$ guarantees that $W^{(c)}(1)=1$. 
Note that it is also possible to give an expression in the case $i<c/2$ by using Atkin-Lehner symmetry. Though we do not need such explicit formula.
\begin{cor}\label{Cor:WiMellinSc}
Suppose $i\geq c/2$. Let $\chi$ be a character over $F^\times$ such that $\chi(\varpi)=1$,  $c(\chi)=c-i$ when $i\neq c-1$, or $c(\chi)\leq 1$ when $i=c-1$. Then
\begin{align*}
&\int\limits_{x\in \OF^\times}W^{(i)}(x)\chi(x)d^\times x\\
=& C_0^{-1}\int\limits_{x\in\OF^\times}\psi(-\varpi^{i-c}x)\chi^{-1}(x)d^\times x
\int\limits_{v_E(u)=-c(\theta)-e_\E+1}\theta^{-1}(u) \chi_E(u)\psi_\E(u)d^\times u.
\end{align*}
\end{cor}
Here we note that $v(N_{E/{F}}(u))=-c$ when $v_E(u)=-c(\theta)-e_\E+1$. 

Note that by \cite[Chapter 11]{BH06} there are (exceptional) supercuspidal representations over $p=2$ which are not directly related to some $\theta$ over a quadratic extension $E$, but they have bounded conductors by \cite[Corollary 45.6]{BH06}.

For non-dihedral supercuspidal representations, we only need the following less precise result coming from \cite[Corollary 2.18]{hu_triple_2017}:
\begin{lem}\label{Lem:nondihedralsupercuspidalW}
Let $\pi$ be a supercuspidal representation with trivial central character. Then 
$$W^{(c)}=\Char({O_\F^\times}),$$
and $W^{(i)}(x)$ consists of level $c-i$ components (and level $0$ component when $i=c-1$).
\end{lem}
\subsubsection{Principal series representation case}
Let $\pi=\pi(\mu^{-1},\mu)$ be a principal series representation with  $c(\mu)=c_0=c/2$. In this case denote
\begin{equation*}
C=\int\limits_{u\in \varpi^{-c_0}\OF^\times} \mu( u)\psi(- u)du.
\end{equation*}
Using Lemma \ref{lem:Gaussint} we have that  $|C|\asymp q^{c_0/2}.$

By \cite[Lemma 2.12]{hu_triple_2017},  we have
\begin{lem}\label{Lem:Wiprincipal}
As a function in $x$, $W^{(i)}(x)$ is supported on $v(x)=\min\{0,2i-c\}$, except when $i=c/2$, where $W^{(i)}(x)$ can be supported on $v(x)\geq 0$. $W^{(i)}(x)$ consists only of level $c-i$ components, except when $i=c-1$ where it consists of level $\leq 1$ components. More explicitly,
\begin{enumerate}
\item When $c_0<i\leq c(\pi)=2c_0$, $W^{(i)}(x)$ is supported on $x \in \OF^\times$, where
\begin{equation*}
W^{(i)}( x )=C^{-1}\int\limits_{u\in \varpi^{-c_0}\OF^\times} \mu(1+u\varpi^{i})\mu( x  u)\psi(- x  u)du.
\end{equation*}
\item If $i=c_0$, $W^{(c_0)}(x)$ is supported on $x\in \OF$, where
\begin{align*}
&W^{(c_0)}(x)\\
=&C^{-1}\int\limits_{v(u)\leq -c_0,u\notin \varpi^{-c_0}(-1+\varpi \OF)}\mu(1+u\varpi^{c_0})\mu(x u)\left|\frac{1}{ u\varpi^{c_0}(1+u\varpi^{c_0})}\right|^{1/2}\psi(-x u)q^{-v(x)/2}du.
\end{align*}

\end{enumerate}
As a function in $x$, $W^{(i)}$ consists only of level $c-i$ components, except when $i=c-1$ when it consists of level $\leq 1$ components.
\end{lem}

Recall the convention for split quadratic algebra in Definition \ref{Def:splitconvention}. It allows us to reformulate the above result similarly as in the field extension case. Denote $\theta=\mu^{-1}\otimes \mu$ in this case. Then principal series representations can be parameterized by $\theta$ over split quadratic algebra. 
\begin{cor}\label{Cor:PrincipalWiAlt}
Denote  $C_0=\int\limits_{u\in\varpi^{-c_0}O_E^\times }\theta^{-1}(u)\psi_E(u)d^\times u$. Then we have
\begin{enumerate}
\item When $c_0<i\leq c(\pi)=2c_0$, $W^{(i)}(x)$ is supported on $x \in \OF^\times$, where
\begin{equation*}
W^{(i)}( x )=C_0^{-1}\int\limits_{u\in \varpi^{-c_0}O_E^\times} \theta^{-1}(u)\psi\lb-\frac{\varpi^i}{x}N_{E/F}(u)\rb\psi_E(u)d^\times u.
\end{equation*}
\item Denote $\varpi^{i,j}=(\varpi^i,\varpi^j)$ as an element in $E$. For fixed $k=v(x)$, we have
\begin{align*}
&W^{(c_0)}( x )\\
=&\sum\limits_{0\leq j\leq k}\frac{q^{-k/2}\mu(\varpi^{k-2j})}{
C_0
}
\iint\limits_{(u,v)\in \varpi^{-c_0}O_E^\times} \theta^{-1}((u,v))\psi\lb -\frac{\varpi^{c_0+k}uv}{x}\rb\psi_E((\varpi^{k-j}   u, \varpi^j v))d^\times ud^\times v\\
=&\sum\limits_{0\leq j\leq k}\frac{q^{-k/2}\mu(\varpi^{k-2j})}{
C_0
}
\iint\limits_{u\in \varpi^{-c_0}O_E^\times} \theta^{-1}(u)\psi\lb -\frac{\varpi^{c_0+k}N_{E/F}(u)}{x}\rb\psi_E(\varpi^{k-j,j}   u)d^\times u.\notag
\end{align*}
\end{enumerate}
\end{cor}
\begin{proof}
Note that by Lemma \ref{Lem:Wiprincipal}, we have for $c_0<i\leq 2c_0$,
\begin{align*}
W^{(i)}( x )&=C^{-1}\int\limits_{u\in \varpi^{-c_0}\OF^\times} \mu(1-u\varpi^{i})\mu(- x  u)\psi( x  u)du\\
&=\frac{1}{C\int\limits_{v\in\varpi^{-c_0}\OF^\times }\mu^{-1}(v)\psi(v)dv}\iint\limits_{u,v\in \varpi^{-c_0}\OF^\times} \mu^{-1}(v)\psi((1-u\varpi^{i})v)\mu(- x  u)\psi( x  u)dudv\\
&=\frac{1}{\int\mu\psi du\int\mu^{-1}\psi dv}\iint\limits_{u,v\in \varpi^{-c_0}\OF^\times}\mu^{-1}(v)\mu(u)\psi\lb-\frac{\varpi^i uv}{x}\rb\psi(u+v)dudv.
\end{align*}
We can then change the Haar measures in the numerator and the denominator simultaneously.
Part (1) is indeed as claimed. Part (2) can be proven similarly, with extra care about valuations.
Note that for fixed $k=v(x)$, the domain for $u$ need to satisfy $v(u)\geq -c_0-k$ to have nonzero integrals. Denote $v(u)=-c_0-j$ for $0\leq j\leq k$. Then we have
\begin{align*}
&W^{(c_0)}( x )\\
=&\frac{1}{\int\mu\psi du}\int\limits_{v(u)\leq -c_0,u\notin \varpi^{-c_0}(1+\varpi \OF)} \mu(1-u\varpi^{c_0})\mu( x  u)\psi( x  u)q^{v(u)+c_0-k/2}du\\
=&\sum\limits_{0\leq j\leq k}\frac{q^{-j-k/2}\mu(\varpi^{-j})}{
\int\mu\psi du\int\mu^{-1}\psi dv
}\\
&\ \ 
\int\limits_{v\in \varpi^{-c_0}\OF^\times}\int\limits_{v(u)=-c_0-j,u\notin \varpi^{-c_0}(1+\varpi \OF)} \mu^{-1}(v)\psi((1-u\varpi^{c_0})\varpi^jv)\mu( x  u)\psi( x  u)dudv\\
=&\sum\limits_{0\leq j\leq k}\frac{q^{-j-k/2}\mu(\varpi^{k-j})}{
\int\mu\psi du\int\mu^{-1}\psi dv
}\\
&\ \ \int\limits_{v\in \varpi^{-c_0}\OF^\times}\int\limits_{v(u)=-c_0-j,u\notin x\varpi^{-c_0-k}(1+\varpi \OF)} \mu^{-1}(v)\psi\lb\lb 1-\frac{u\varpi^{c_0+k}}{x}\rb\varpi^jv\rb\mu(   u)\psi(\varpi^k   u)dudv\\
=&\sum\limits_{0\leq j\leq k}\frac{q^{-k/2}\mu(\varpi^{k-2j})}{
\int\mu\psi du\int\mu^{-1}\psi dv
}
\iint\limits_{(u,v)\in \varpi^{-c_0}O_E^\times} \theta^{-1}((u,v))\psi\lb -\frac{\varpi^{c_0+k}uv}{x}\rb\psi_E((\varpi^{k-j}   u,\varpi^j v))dudv.
\end{align*}
In the fourth equality we removed the restriction on $u$, as it only affects the case $j=0$, where if $u\in x\varpi^{-c_0}(1+\varpi \OF)$, the integral in $v$ will be zero. 
\end{proof}
Parallel to Corollary \ref{Cor:WiMellinSc}, we have
\begin{cor}\label{Cor:WiMellinPrin}
Suppose $i\geq c/2$. Let $\chi$ be a character over $F^\times$ such that $\chi(\varpi)=1$,  $c(\chi)=c-i$ when $i\neq c-1$ or $c(\chi)\leq 1$ when $i=c-1$. Then if $i>c/2$,
\begin{align*}
&\int\limits_{x\in \OF^\times}W^{(i)}(x)\chi(x)d^\times x\\
=&C_0^{-1}\int\limits_{x\in\OF^\times}\psi(-\varpi^{i-c}x)\chi^{-1}(x)d^\times x
\int\limits_{v(v)=-c(\theta)}\mu^{-1}(v) \chi(v)\psi(v)d^\times v\int\limits_{v(u)=-c(\theta)}\mu(u) \chi(u)\psi(u)d^\times u\\
=&C_0^{-1}\int\limits_{x\in\OF^\times}\psi(-\varpi^{i-c}x)\chi^{-1}(x)d^\times x\int\limits_{u\in \varpi^{-c_0}O_E^\times} \theta^{-1}(u)\chi_E(u)\psi_E(u)d^\times u.
\end{align*}
If $i=c/2$, $k\geq0$ then
\begin{align*}
&\int\limits_{x\in \OF^\times}W^{(c/2)}(\varpi^k x)\chi(x)d^\times x\\ =&\sum\limits_{0\leq j\leq k}\frac{q^{-k/2}\mu(\varpi^{k-2j})}{
C_0
}\int\limits_{x\in\OF^\times}\psi(-\varpi^{-c_0}x)\chi^{-1}(x)d^\times x
\int\limits_{v(v)=-c_0}\mu^{-1}(v) \chi(v)\psi(\varpi^j v)d^\times v \\
&\int\limits_{v(u)=-c_0}\mu(u) \chi(u)\psi(\varpi^{k-j}u)d^\times u.\end{align*}
Furthermore if $\mu$ is not an unramified twist of a quadratic character (or equivalently, $c(\mu^2)\neq 0$), we have
\begin{align}\label{Eq:Wmellintricky}
&\int\limits_{x\in \OF^\times}W^{(c/2)}(\varpi^k x)\chi(x)d^\times x\\
=&\sum\limits_{j=0,k}\frac{q^{-k/2}\mu(\varpi^{k-2j})}{
C_0
}\int\limits_{x\in\OF^\times}\psi(-\varpi^{-c_0}x)\chi^{-1}(x)d^\times x
\int\limits_{u\in \varpi_E^{-c_0}O_E^\times}\theta^{-1}(u)\chi_E(u)\psi_E(\varpi^{k-j,j}u)d^\times u.\notag
\end{align}

\end{cor}
\begin{proof}
The only tricky part is the last statement. When the condition for $\mu$ is satisfied and  $2\nmid q$, we have used that for any $\chi$ with $c(\chi)=c(\mu)$, at least one of $c(\mu^{-1} \chi)$, $c(\mu \chi)$ is $c(\mu)$. So $j=0$ or $k$ for the Gauss integrals to be nonzero.
\end{proof}

\begin{cor}\label{Cor:PrincipalL2}
Suppose $p\neq 2$, $\theta=\mu^{-1}\otimes \mu$ with $c(\mu^2)\geq 1$. Then
\begin{align*}
L^2(W,k):=\int\limits_{x\in \OF^\times}\left|W^{(c/2)}(\varpi^k x)\right|d^\times x=\begin{cases}
\frac{q-3}{q-1}, \text{\ if $k=0$};\\
\frac{2}{q^k}, \text{\ if $k\geq 1$}.
\end{cases}
\end{align*}
\end{cor}
\begin{proof}
 Using the spectral decomposition on $O_F^\times$, we have for $c\geq 4$,
\begin{align*}
L^2(W,k) =&\sum\limits_{c(\chi)=c/2}\left|\int\limits_{x\in \OF^\times}W^{(i)}(\varpi^k x)\chi(x)d^\times x\right|^2\\
=&\sum\limits_{c(\chi)=c/2}\sum\limits_{i,j= 0,k}\frac{\mu(\varpi^{k-2i})\overline{\mu(\varpi^{k-2j})}q^{2-k}}{|C_0|^2(q-1)^2q^{c/2}}\int\limits_{u\in \varpi_E^{-c_0}O_E^\times}\theta^{-1}(u)\chi_E(u)\psi_E(\varpi^{k-i,i}u)d^\times u \notag\\
&\overline{\int\limits_{w\in \varpi_E^{-c_0}O_E^\times}\theta^{-1}(w)\chi_E(w)\psi_E(\varpi^{k-j,j}w)d^\times w}.
\end{align*}


When $k=0$ and $c\geq 4$, the Gauss integral $\int\limits_{u\in \varpi_E^{-c_0}O_E^\times}\theta^{-1}(u)\chi_E(u)\psi_E(u)d^\times u$ is non-vanishing if and only if $c(\chi\mu_1)=c(\chi\mu_1^{-1})=c/2$. The number of such characters is $(q-3)(q-1)q^{c/2-2}$.
Then we have by Lemma \ref{lem:Gaussint}
$$L^2(W,0)=\frac{q-3}{q-1}.$$

When $0<k<c/2-1$, $\int\limits_{u\in \varpi_E^{-c_0}O_E^\times}\theta^{-1}(u)\chi_E(u)\psi_E(\varpi^{k-i,i}u)d^\times u$ is non-vanishing if and only if $c(\mu^{-1}\chi)=c/2-k$ and $i=k$, or $c(\mu\chi)=c/2-k$ and $i=0$. Thus for fixed $\chi$, only $i=j=0$ or $i=j=k$ contributes. The number of possible $\chi$ is $2(q-1)^2q^{c/2-k-2}$, and
$$L^2(W,k)=\frac{2}{q^{k}}.$$

When $k=c/2-1$, the discussion is similar, except that unramified $\chi$ also contributes. The number of level $1$ characters is $q-2$. Thus
$$L^2(W,c-1)= \frac{2}{q^{c-1}}\left[(q-2)\frac{q}{(q-1)^2}+\left(-\frac{1}{q-1}\right)^2\right]=\frac{2}{q^{c-1}}.$$

When $k\geq c/2$, $\int\limits_{u\in \varpi_E^{-c_0}O_E^\times}\theta^{-1}(u)\chi_E(u)\psi_E(u)d^\times u$ is non-vanishing if and only if $c(\mu^{-1}\chi)=0$ and $i=k$, or $c(\mu\chi)=0$ and $i=0$. Then we have
$$L^2(W,k)=\frac{2}{q^k}.$$
The total $L^2$ mass $\sum\limits_{k\geq 0}L^2(W,k)=1$, which is expected.

Consider now the case $c=2$, and $c(\mu^2)=1$. When $k\geq 1$, the contribution comes from $c(\mu^{-1}\chi)=0$ or $c(\mu\chi)=0$, and
$$L^2(W,k)=\frac{2}{q^k}.$$
When $k=0$, the contribution comes from all $\chi$ with $c(\chi)\leq 1$, and
$$L^2(W,0)=\left[(q-4)\frac{q}{(q-1)^2}+\frac{1}{(q-1)^2}+2\frac{1}{(q-1)^2}\right]=\frac{q-3}{q-1}.$$
Here the first term comes from $\chi\neq \mu,\mu^{-1},1$, the middle term comes from $\chi=1$, and the last term comes from $\chi=\mu,\mu^{-1}$.

%
%
%

\end{proof}

\subsection{List of levels}
Recall that when $\pi$ is a special  unramified representation, $c(\pi)=1$.
For the remaining cases, $\pi$ corresponds to a character $\theta$ over \'{e}tale quadratic algebra $E$, and we write $\pi=\pi_\theta$ in that case. One can uniformly describe its level by the formula
$$c(\pi_\theta)= \frac{2c(\theta)}{e_E}+e_E-1.$$
In the  Rankin--Selberg/triple product case, one can use the local Langlands correspondence to get the following:
\begin{lem}
Suppose that the central characters of $\pi_i$ are trivial.
\begin{enumerate}
\item If $c(\pi_i)=1$, then $c(\pi_1\times \pi_2)=2$;
\item If $c(\pi_1)\neq c(\pi_2)$, or $\pi_i$ are associated to different \'{e}tale quadratic algebras, then $$c(\pi_1\times\pi_2)=\max\{c(\pi_1)^2, c(\pi_2)^2\};$$
\item If $\pi_i$ are associated to $\theta_i$ over the same \'{e}tale quadratic algebra $E$, then
$$c(\pi_{\theta_1}\times\pi_{\theta_2})=c(\pi_{\theta_1\theta_2})+c(\pi_{\theta_1\overline{\theta_2}});$$
\item If $c(\pi_3)=0$, then $$c(\pi_1\times\pi_2\times\pi_3)=c(\pi_1\times \pi_2)^2.$$
\end{enumerate}
\end{lem}

\section{Partial orthogonality for Whittaker functions}\label{Sec:OrthgonalWhittaker}
Let $c(\Pi)$ be the exponent of finite conductor $C(\Pi)$. In this section we shall focus on the case $c(\pi_1)=c(\pi_2)$, $c(\pi_3)=0$.  
\begin{defn}\label{Defn:Pi}
For two Whittaker functions $W_i\in \pi_i$, normalized so that $W_i(1)=1$, we denote 
\begin{equation*}
P_i(W_1,W_2,\gamma)=\int\limits_{x\in \F^\times}W_1^{(i)}(x)\overline{W_2^{(i)}(x)}|x|^{-1/2+\gamma}d^\times x
\end{equation*}
where $\gamma$ is a parameter such that $|\gamma|<7/64$. In the following we skip the parameter $\gamma$ from $P_i$ without confusion. 
\end{defn}
As we shall see in the next section, this quantity is vital in evaluating the local period integrals. So our goal in this section is to evaluate or give proper bounds for it. The simplest case is where $c(\pi_i)=1$.

\begin{lem}\label{Lem:Wipairingspecial}
Suppose that $\pi_i=\sigma(\mu_i |\cdot|^{1/2}, \mu_i |\cdot|^{-1/2})$ are special unramified representations for $c(\mu_i)=0$ and $\mu_i^2=1$.
Then for $W_i\in \pi_i$ the Whittaker functions associated to newforms,
$$P_1(W_1,W_2)=\frac{1}{1-\mu_1\mu_2(\varpi)q^{-3/2-\gamma}}
. $$
\end{lem}
This follows directly from the definition of $P_i$ and Lemma \ref{Lem:Wispecial}.

We introduce the following technical assumption which we shall circumvent later on.
\begin{assumption}\label{Assump:notquadratictwist}
If $c(\pi_1)=c(\pi_2)>0$, then $\pi_i,i=1,2$ are not twists of unramified representations by quadratic characters. 
\end{assumption}
The main result of this section is the following.
\begin{prop}\label{Prop:Wipairing}
Suppose that $p\neq 2$, $c(\pi_1)=c(\pi_2)=c\geq 2$ satisfying Assumption \ref{Assump:notquadratictwist}, so that $c\leq c(\pi_1\times \pi_2)\leq 2c$.  
Let $W_i$ be the Whittaker functions associated to newforms  $\varphi_i^\circ\in\pi_i$. 
Then for $p\neq 2$, $i\geq c/2>1$, we have 
\begin{equation}\label{Eq:Wipairingsummary1}
\begin{cases}
P_i(W_1, W_2)=1 ,&\text{\ if $i\geq c(\pi_1\times \pi_2)/2$, except when $i=c/2$ and } \\
 & \text{\ $\pi_i$ are principal series;}\\
\left|P_i(W_1, W_2)- \frac{q-3}{q-1}\right|\leq \frac{2}{q^{1/2+\Re(\gamma)}-1}, &\text{\ if $i=c(\pi_1\times \pi_2)/2=c/2$, $\pi_i$ are principal series;}\\
P_i(W_1, W_2)=-\frac{1}{q-1}, &\text{\ if }
i=c(\pi_1\times \pi_2)/2-1>c/2;\\
|P_i(W_1, W_2)|\leq \frac{2}{q-1}, & \text{\ if } i=c(\pi_1\times \pi_2)/2-1=c/2; \\
P_i(W_1, W_2)=0, &\text{\ otherwise.}

\end{cases}\end{equation}
When $p$ is large enough and $c=2$, the first two lines of \eqref{Eq:Wipairingsummary1} still holds, while in the case $c(\pi_1\times\pi_2)=4$ and $i=1$, we have
\begin{equation*}
P_1(W_1,W_2)\ll\frac{1}{q}.
\end{equation*}
\end{prop}
\begin{rem}
 When $i=c(\pi_1\times \pi_2)/2=c/2$ and $\pi_i$ are principal series, the reason for difference is that some of the $L^2$-mass of $W_i$ goes to $v(x)> 0$.

Note that Assumption \ref{Assump:notquadratictwist} also allows us to apply \eqref{Eq:Wmellintricky} and Corollary \ref{Cor:PrincipalL2}. 
\end{rem}

The main ingredient for proving the above proposition is the following result on the relation between Gauss sums/integrals, which may have independent interest.
\begin{prop}\label{Prop:GaussintRelation}
Suppose that $p\neq 2$, $c(\pi_{\theta_1})=c(\pi_{\theta_2})=c\geq 2$, where $\pi_i=\pi_{\theta_i}$ satisfy Assumption \ref{Assump:notquadratictwist}.
Denote
\begin{align}\label{Eq:Pitheta}
&P_i(\theta_1,\theta_2)\\
=&\notag \sum\limits_{c(\chi)=c-i}\int\limits_{v_E(u)=-c(\theta)-e_\E+1}\theta_1^{-1}(u)\chi_E(u) \psi_E(u)d^\times u \overline{
\int\limits_{v_{E'}(w)=-c(\theta)-e_{E'}+1}\theta_2^{-1}(w) \chi_{E'}(w)\psi_{E'}(w)d^\times w},
\end{align}
and denote
\begin{align*}
&P_i^\circ(\theta_1,\theta_2)\\
=&\frac{P_i(\theta_1,\theta_2)}{\sharp\{c(\chi)=c-i\}\int\limits_{v_E(u)=-c(\theta)-e_\E+1}\theta_1^{-1}(u) \psi_E(u)d^\times u\overline{
\int\limits_{v_{E'}(w)=-c(\theta)-e_{E'}+1}\theta_2^{-1}(w) \psi_{E'}(w)d^\times w}}.
\end{align*}
Then for $p\neq 2$ and $i\geq c/2>1$,
\begin{equation}\label{Eq:thetaipairingsummary1}
\begin{cases}
P_i^\circ(\theta_1,\theta_2)= 1, & \text{\ if $i\geq c(\pi_1\times \pi_2)/2$, except when $i=c/2$ and }\\
& \text{\ \ $\pi_i$ are principal series;}\\
P_i^\circ(\theta_1, \theta_2)= \frac{q-3}{q-1}, &\text{\ if $i=c(\pi_1\times \pi_2)/2=c/2$ and $\pi_i$ are principal series;}\\
P_i^\circ(\theta_1,\theta_2)=-\frac{1}{q-1}, &\text{\ if $i=c(\pi_1\times \pi_2)/2-1>c/2$};\\
|P_i^\circ(\theta_1, \theta_2)|\leq \frac{2}{q-1}, & \text{\ if } i=c(\pi_1\times \pi_2)/2-1=c/2; \\
P_i^\circ(\theta_1,\theta_2)=0, &\text{\ otherwise.}
\end{cases}
\end{equation}
When $p$ is large enough and $c=2$, the first two lines of \eqref{Eq:thetaipairingsummary1} still holds, while in the case $c(\pi_1\times\pi_2)=4$ and $i=1$, we have
\begin{equation}
P_i^\circ(\theta_1,\theta_2)\ll\frac{1}{q}.
\end{equation}
\end{prop}

In the following, suppose that $\pi_1$ is associated to a character $\theta_1$ over quadratic algebra $E$, and $\pi_2$ is associated to $\theta_2$ over $E'$. 
We break the proof into two main parts. The first  part, consisting of Section \ref{Sec:EnotE'}, \ref{Sec:sameEfield}, \ref{Sec:SameEsplit}, assume that $c\geq 3$ and discuss according to whether $E\neq E'$, $E=E'$ are fields, or $E=E'$ split. One of the basic tools for these cases is the p-adic stationary phase analysis relying on Lemma \ref{Lem:DualLiealgForChar}.

The second part in Section \ref{Sec:c=2} covers the case $c=2$, where we can not apply Lemma \ref{Lem:DualLiealgForChar} and need different but elementary approaches. It may also be possible to approach this case using methods from algebraic geometry.

In Section \ref{Sec:p=2} we give partial results in the $p=2$ case when $c(\pi_i)$ are large enough, which is however sufficient for later purposes.

\subsection{Case $E\neq E'$, $c\geq 3$.}\label{Sec:EnotE'}
 Let $C_0$ $C_0'$ be the corresponding constant for $\pi_1$, $\pi_2$ as in Lemma \ref{Lem:WiSc} or Corollary \ref{Cor:PrincipalWiAlt}.

In this case, according to the beginning of Section \ref{Sec:SC},  either  $c(\pi_i)$ is even, where one of $E$, $E'$ is inert field extension while the other one splits; or $c(\pi_i)$ is odd and $E$ $E'$ are different ramified field extensions. Also by the condition $c\geq 3$, we have $c(\theta_i)\geq 2$.

In either cases, by Lemma \ref{Lem:WiSc}, the support of the integral for $P_i(W_1,W_2)$ is $v(x)=0$. Using the spectral decomposition on $\OF^\times$, we get that
\begin{equation}\label{Eq:PiSpectralD}
P_i(W_1,W_2)=\sum\limits_{\chi}\int\limits_{x\in \OF^\times} W_1^{(i)}(x)\chi(x)d^\times x\overline{\int\limits_{y\in \OF^\times}W_2^{(i)}(y)\chi(y)d^\times y}.
\end{equation}

By Lemma \ref{Lem:WiSc}, the sum in $\chi$ is over those with $c(\chi)=c-i$ when $i\neq c-1$, and $c(\chi)\leq 1$ when $i=c-1$.

\subsubsection{Case $i=c$}\label{Sec:i=c}
In this case, it is well known that $W_j^{(c)}$ are both characteristic functions on $\OF^\times$, thus $P_c(W_1,W_2)=1$.

\subsubsection{Case $i=c-1$}\label{Sec:EE'i=c-1}
Note that the contributions come from $c(\chi)\leq 1$.
Using Lemma \ref{lem:Gaussint}, Corollary \ref{Cor:WiMellinSc}, \ref{Cor:WiMellinPrin}, we have
\begin{align*}
P_{c-1}(W_1,W_2)=&\frac{1}{C_0\overline{C_0'}}\left[\frac{1}{(q-1)^2} 
P_c(\theta_1,\theta_2)+\frac{q}{(q-1)^2}P_{c-1}(\theta_1,\theta_2)\right].
\end{align*}
Let $\alpha_i$ be associated to $\theta_i$ using Lemma \ref{Lem:DualLiealgForChar}. Using the definition for $C_0, C_0'$, formula \eqref{Eq:Pitheta} and that $c(\theta_i)\geq 2 c(\chi_E)$, we have by Lemma \ref{Lem:GaussintTwist},
\begin{align*}
P_{c-1}(W_1,W_2)=\frac{1}{(q-1)^2} +\frac{q}{(q-1)^2} \sum\limits_{c(\chi)=1}\chi\lb\frac{N_{E/F}(\alpha_1)}{N_{E'/F}(\alpha_2)}\rb=-\frac{1}{q-1}.
\end{align*}
Here we have used that $\frac{N_{E/F}(\alpha_1)}{N_{E'/F}(\alpha_2)}\not\equiv 1$ by Lemma \ref{Lem:differentE}. 

\subsubsection{Case $c/2\leq i<c-1$, $c\geq 3$.}\label{Sec:EE'smalli}
Recall the notation that
$\chi_1\sim_j \chi_2$ if and only if $c(\chi_1\chi_2^{-1})\leq j$.

Consider first the case $c/2<i<c-1$.
Using Lemma \ref{lem:Gaussint}, Corollary \ref{Cor:WiMellinSc}, \ref{Cor:WiMellinPrin}, we have
\begin{align*}
P_i(W_1,W_2)
=\frac{1}{C_0\overline{C_0'}}\left|\int\psi\chi\right|^2P_i(\theta_1,\theta_2)=\frac{1}{C_0\overline{C_0'}(q-1)^2q^{c-i-2}}P_i(\theta_1,\theta_2).
\end{align*}
We break the sum over $\chi$ in $P_i(\theta_1,\theta_2)$ into a double sum of $\chi_0$ over the set of characters of level $c-i$ modulo $\sim_1$, and the sum of characters $\eta $ of level $\leq 1$. 
To each $\chi_0$, we associate $\alpha_0$ as in Lemma \ref{Lem:DualLiealgForChar}.
Then using Lemma \ref{Lem:GaussintTwist}, $P_i(\theta_1,\theta_2)$ can be rewritten as
\begin{align*}
\sum\limits_{\chi_0/\sim_1}\sum\limits_{c(\eta)\leq 1}\eta\lb \frac{N_{E/F}(\alpha_1+\alpha_0)}{N_{E'/F}(\alpha_2+\alpha_0)}\rb&\int\limits_{v_E(u)=-c(\theta)-e_\E+1}\theta_1^{-1}(u)\chi_{0,E}(u) \psi_{\E}(u)d^\times u \\
\times &\overline{
\int\limits_{v_{\E'}(u)=-c(\theta)-e_{\E'}+1}\theta_2^{-1}(u) \chi_{0,\E'}(u)\psi_{\E'}(u)d^\times u}.
\end{align*}
Using Lemma \ref{Lem:differentE}, the inner sum in $\eta$ is vanishing, thus
$P_i(W_1,W_2)=0$ in this case.

Now consider the case $i=c/2$.
In this case, we suppose without loss of generality that $E$ is split and $E'$ is inert. The expression for $\int\limits_{x\in \OF^\times} W_1^{(i)}(x)\chi(x)d^\times x$ is a little more complicated according to Corollary \ref{Cor:WiMellinPrin}, but since $W_2$ is supported only at $v(x)=0$, we only need to consider the case $k=j=0$ in Corollary \ref{Cor:WiMellinPrin}. The computation is then similar to the $c/2<i<c-1$ case, and we have again $P_i(W_1,W_2)=0$.


%
\subsection{Case $E=E'$ is a field extension, $c\geq 3$.}\label{Sec:sameEfield}
Let $\pi_j$ be associated to $\theta_j$ over the same 
quadratic field extension $E$, with $c(\theta_1)=c(\theta_2)$. Note that $\overline{\theta(x)}=\theta^{-1}(x)$ by the condition that $\theta|_{F^\times}=1$.
Without loss of generality, we assume that $$l:=c(\theta_2\theta_1^{-1})\leq c(\theta_1\theta_2).$$ 
\subsubsection{Case $i\geq c(\pi_{1}\times\pi_{2})/2$}\label{Sec:samefieldElargei}
If $l=0$, we have
\begin{align*}
&P_i(\theta_1,\theta_2)\\=&\sum\limits_{c(\chi)=c-i}\int\limits_{v_E(u)=-c(\theta)-e_\E+1}\theta_1^{-1}(u)\chi_E(u) \psi_E(u)d^\times u \overline{
\int\limits_{v_{E'}(w)=-c(\theta)-e_{E'}+1}\theta_2^{-1}(w) \chi_{E'}(w)\psi_{E'}(w)d^\times w}\\
=&\sum\limits_{c(\chi)=c-i}\theta_2\theta_1^{-1}(\varpi_E)^{-c(\theta)-e_E+1}\left|\ \int\limits_{v_E(u)=-c(\theta)-e_\E+1}\theta_1^{-1}(u)\chi_E(u) \psi_E(u)d^\times u\right|^2.\notag
\end{align*}
There is no cancellation and one can easily get that
$$P_i(W_1,W_2)=1.$$
Consider the case $l>0$ now.
We can rewrite $P_i(\theta_1,\theta_2)$ as
\begin{align}\label{Eq:PisameE}
&P_i(\theta_1,\theta_2)\\
=&\notag \sum\limits_{c(\chi)=c-i}\iint\limits_{v_E(u),v_E(w)=-c(\theta_1)-e_\E+1}\theta_1^{-1}(u)\theta_2(w)\chi_E\lb\frac{u}{w}\rb \psi_E(u-w)d^\times u d^\times w \\
=&\sum\limits_{c(\chi)=c-i}\iint\limits_{x\in O_E^\times,v_E(w)=-c(\theta_1)-e_\E+1}\theta_1^{-1}(x)\theta_2\theta_1^{-1}(w)\chi_E\lb x\rb \psi_E((x-1)w) d^\times wd^\times x.\notag
\end{align}

As $c(\theta_2\theta_1^{-1})=l$, the inner integral in $w$ is non-vanishing if and only if $x\in 1+\varpi_E^{c(\theta_1)-l}O_E^\times$. Writing $x=1+\varpi_E^{c(\theta_1)-l}y$ for $y\in O_E^\times$, $q_E=|\varpi_E|_E^{-1}$,
\begin{align}\label{Eq:relatedPi}
&P_i(\theta_1,\theta_2)\\
=&\notag\frac{\theta_1\theta_2^{-1}(\varpi_E^{c(\theta_1)-l})}{q_E^{c(\theta_1)-l}} \int\limits_{v_E(w)=-l-e_E+1}\theta_2\theta_1^{-1}(w)\psi_E(w)d^\times w \\
&\notag\sum\limits_{c(\chi)=c-i}\int\limits_{y\in O_E^\times}\theta_1^{-1}\chi_E(1+\varpi_E^{c(\theta_1)-l}y)\theta_1\theta_2^{-1}(y) d^\times y.
\end{align}

Note that $c(\chi_{E})=e_Ec(\chi)-e_E+1$, so when 
\begin{equation}\label{Eq:icondition}
e_E(c-i)-e_E+1\leq c(\theta_1)-l,
\end{equation}
we have 
\begin{align*}
P_i(\theta_1,\theta_2)=&\frac{\theta_1\theta_2^{-1}(\varpi_E^{c(\theta_1)-l})}{q_E^{c(\theta_1)-l}}\int\limits_{v_E(w)=-l-e_E+1}\theta_2\theta_1^{-1}(w)\psi_E(w)d^\times w\\
&  \sum\limits_{c(\chi)=c-i}\int\limits_{y\in O_E^\times}\theta_1^{-1}(1+\varpi_E^{c(\theta_1)-l}y)\theta_1\theta_2^{-1}(y) d^\times y\\
=&C_0\overline{C_0'}(q-1)^2q^{c-i-2}. \notag
\end{align*}
Here we arrived at the second line by working backward to \eqref{Eq:PisameE} without the twists by $\chi$.
The condition \eqref{Eq:icondition} translates into the following: when $e_E=1$, we have $i\geq c(\theta_1)+l$; when $e_E=2$, we have $i\geq \frac{c(\theta_1)+l+2}{2}$ as both $c(\theta_1)$ and $l$ must be even. In both cases, we can rewrite the condition as $$i\geq \frac{c(\pi_{1}\times\pi_{2})}{2}.$$

We can then proceed as in $E\neq E'$ case to get
$$P_i(W_1,W_2)=1$$
when $i\geq \frac{c(\pi_{\theta_1}\times\pi_{\theta_2})}{2}$.


\subsubsection{Case $c/2\leq i< c(\pi_{1}\times\pi_{2})/2$}\label{Sec:sameEfieldgeneric}
There are many possible situations in this case. We shall discuss case by case according to $l=c(\theta_2\theta_1^{-1})\leq c(\theta_1\theta_2)$. 

Note that $l=0$ does not happen  as there will not be an integer $i$ with $c/2\leq i< c(\pi_{1}\times\pi_{2})/2$ in this case. 
\begin{enumerate}
\item
Consider first the case  where $l=c(\theta_1)$. By the assumption $c(\theta_2\theta_1^{-1})\leq c(\theta_1\theta_2)$, this implies that $\alpha_1\not\equiv \pm\alpha_2$. As a result, we have
$\frac{N_{E/F}(\alpha_1)}{N_{E'/F}(\alpha_2)}\not\equiv 1$ and $\frac{N_{E/F}(\alpha_1+\alpha_0)}{N_{E'/F}(\alpha_2+\alpha_0)}\not\equiv 1$, and the result follows similarly as in Section \ref{Sec:EE'i=c-1}, \ref{Sec:EE'smalli}.
\item
Suppose now $l=1$, then automatically $e_E=1$, and $i=c/2$ in this case. 
From \eqref{Eq:relatedPi}, we break the sum over $\chi$  into a double sum of $\chi_0$ modulo $\sim_{c/2-1}$, and the sum of characters $\eta $ of level $\leq c/2-1$, with $\chi=\chi_0\eta$. For each such $\chi_0$, we can associated $\alpha_0$ by Lemma \ref{Lem:DualLiealgForChar}, then the sum in $\chi_0$ corresponds to a sum in $\alpha_0\in \varpi^{-c/2}\OF^\times \mod \varpi^{-c/2+1}\OF$.
The last line of \eqref{Eq:relatedPi} becomes
\begin{align*}
&\sum\limits_{c(\chi)=c-i}\int\limits_{y\in O_E^\times}\theta_1^{-1}\chi_E(1+\varpi_E^{c(\theta_1)-l}y)\theta_1\theta_2^{-1}(y) d^\times y\\&=(q-1)q^{c/2-2}\sum\limits_{\chi_0}\int\limits_{y\in O_E^\times}\psi_E((-\alpha_1+\alpha_0)\varpi_E^{c(\theta_1)-l}y)\theta_1\theta_2^{-1}(y) d^\times y \notag\\
&=(q-1)q^{c/2-2}\int\limits_{y\in \varpi_E^{-1} O_E^\times}\psi_E(y)\theta_1\theta_2^{-1}(y) d^\times y \sum\limits_{a_0}\theta_2\theta_1^{-1}\lb (-\alpha_1+\alpha_0)\varpi_E^{c(\theta_1)-1}\rb. \notag
\end{align*}
Here in the second equality we have used Corollary \ref{Cor:DualLiealgForBasechangeChar}, and that the integral is independent of $\eta$.
Using Lemma \ref{lem:Gaussint}, \ref{Lem:level1cancellation}, we have
$$\left|\sum\limits_{c(\chi)=c-i}\int\limits_{y\in O_E^\times}\theta_1^{-1}\chi_E(1+\varpi_E^{c(\theta_1)-l}y)\theta_1\theta_2^{-1}(y) d^\times y\right|\leq \frac{2q^{c/2-1}}{q+1},$$
$$\left|P_{c/2}(\theta_1,\theta_2)\right|\leq \frac{2q^2}{(q+1)(q^2-1)q^{c/2}},
$$
\begin{align*}
\left|P_{c/2}(W_1,W_2)\right|&=
\left|\frac{1}{C_0\overline{C_0'}(q-1)^2q^{c-i-2}}P_{c/2}(\theta_1,\theta_2)\right|\\
&\leq \frac{(1-q^{-2})^2q^c}{(q-1)^2q^{c/2-2}}\frac{2q^2}{(q+1)(q^2-1)q^{c/2}}=\frac{2}{q-1}.
\end{align*}
\item
If $1<l<c(\theta_1)$, we write $y=y_0(1+y_1)$ for $y_0\in O_E^\times/1+\varpi_E^{l-1}O_E$,  $y_1\in \varpi_E^{l-1}O_E$. Let $\alpha_3$ be the constant associated to $\theta_1\theta_2^{-1}$. Then
\begin{align*}
&\int\limits_{y\in O_E^\times}\theta_1^{-1}\chi_E(1+\varpi_E^{c(\theta_1)-l}y)\theta_1\theta_2^{-1}(y) d^\times y\\
=&\frac{1}{1-q_E^{-1}}\sum\limits_{y_0\in O_E^\times/1+\varpi_E^{l-1}O_E}\theta_1^{-1}\chi_E(1+\varpi_E^{c(\theta_1)-l}y_0)\theta_1\theta_2^{-1}(y_0)\\
&\int\limits_{y_1\in \varpi_E^{l-1}O_E}\psi_E\lb (-\alpha_1+\alpha_0)\frac{\varpi_E^{c(\theta_1)-l}y_0y_1}{1+\varpi_E^{c(\theta_1)-l}y_0}+\alpha_3y_1 \rb d y_1 .\notag
\end{align*}
The integral in $y_1$ is nonvanishing only if $y_0$ satisfies
$$(-\alpha_1+\alpha_0)\frac{\varpi_E^{c(\theta_1)-l}y_0}{1+\varpi_E^{c(\theta_1)-l}y_0}+\alpha_3\equiv 0 \mod\varpi_E^{-l-e_E+2}.$$
Note that this is only possible if $c(\chi_E)\leq c(\theta_1)$. Then $v_E(-\alpha_1+\alpha_0)=-c(\theta_1)-e_E+1,$ $v_E(\alpha_3)=-l-e_E+1>-c(\theta_1)-e_E+1$. As $c(\theta_1)>l$ we get 
$$y_0
\equiv \frac{\alpha_3}{(\alpha_1-\alpha_0)\varpi_E^{c(\theta_1)-l}}\mod \varpi_E.
$$
As $\alpha_1,\alpha_3\in E$ are trace 0 elements and $\alpha_0\in F$, we have
\begin{equation}\label{Eq:5221}
N_{E/F}(1+\varpi_E^{c(\theta_1)-l}y_0)\equiv 1+\frac{2\alpha_3}{\alpha_1} \not\equiv 1\mod \varpi^{(c(\theta_1)-l)/e_E+1}.
\end{equation}

We can then rewrite \eqref{Eq:relatedPi} as follows by imposing the condition on $y_0$
\begin{align*}
P_i(\theta_1,\theta_2)=&\frac{\theta_1\theta_2^{-1}(\varpi_E^{c(\theta_1)-l})}{q_E^{c(\theta_1)-l}} \int\limits_{w}\theta_2\theta_1^{-1}\psi_E \sum\limits_{c(\chi)=c-i}\\
&\int\limits_{y\equiv\frac{\alpha_3}{(\alpha_1-\alpha_0)\varpi_E^{c(\theta_1)-l}}\mod \varpi_E }\theta_1^{-1}\chi_E(1+\varpi_E^{c(\theta_1)-l}y)\theta_1\theta_2^{-1}(y) d^\times y.
\end{align*}
\begin{enumerate}
\item[(3i)]
If $c-i>(c(\theta_1)-l)/e_E+1$, which is equivalent to $i<c(\pi_{1}\times\pi_{2})/2-1$, we can break the sum of $\chi=\chi_0\eta$ into the sum of $\chi_0$ modulo $\sim_{(c(\theta_1)-l)/e_E+1}$, and the sum of $\eta$ with $c(\eta)\leq (c(\theta_1)-l)/e_E+1$. Note that the domain $y\equiv\frac{\alpha_3}{(\alpha_1-\alpha_0)\varpi_E^{c(\theta_1)-l}}\mod \varpi_E$ is independent of $\eta$. Then the sum in $\eta$ first is vanishing for any $y\equiv\frac{\alpha_3}{(\alpha_1-\alpha_0)\varpi_E^{c(\theta_1)-l}}\mod \varpi_E$ using \eqref{Eq:5221}, and
$$P_i(W_1,W_2)=P_i(\theta_1,\theta_2)=0.$$
\item[(3ii)]
If $c-i=(c(\theta_1)-l)/e_E+1$, which is equivalent to that $i=c(\pi_{1}\times\pi_{2})/2-1$, 
then the domain of integral is actually $y\equiv\frac{\alpha_3}{\alpha_1\varpi_E^{c(\theta_1)-l}}\mod \varpi_E$. We take the sum in $\chi$ first.
Using that
$$\sum\limits_{c(\chi)\leq c-i}\chi_E\left(1+\frac{\alpha_3}{\alpha_1}\right)=0$$
in this case,
and $\chi_E(1+\varpi_E^{c(\theta_1)-l}y)=1$ when $c(\chi)<c-i$,
 we obtain that
\begin{align*}
&P_i(\theta_1,\theta_2)\\
=&-\sum\limits_{c(\chi)<c-i}\frac{\theta_1\theta_2^{-1}(\varpi_E^{c(\theta_1)-l})}{q^{c(\theta_1)-l}} \int\limits_{w}\theta_2\theta_1^{-1}\psi_E \int\limits_{y\equiv\frac{\alpha_3}{\alpha_1\varpi_E^{c(\theta_1)-l}}\mod \varpi_E }\theta_1^{-1}(1+\varpi_E^{c(\theta_1)-l}y)\theta_1\theta_2^{-1}(y) d^\times y.
\end{align*}
By working backwards without $\chi$ twists, we get
$$
 P_i(W_1,W_2)=-\frac{1}{q-1}.
$$

\end{enumerate}
\end{enumerate}
\subsection{Case $E=E'$ splits}\label{Sec:SameEsplit}
Most of the discussions in Section \ref{Sec:sameEfield} hold for the split case. The only difference is the case $i=c/2$ due to the slightly more complicated expression in Corollary \ref{Cor:WiMellinPrin}.
In that case, we have

\begin{align*}
&P_{c/2}(W_1,W_2)\\
=&\sum\limits_{k\geq 0}q^{(1/2-\gamma)k}\sum\limits_{c(\chi)=c/2}\int\limits_{x\in \OF^\times} W_1^{(c/2)}(\varpi^k x)\chi(x)d^\times x\overline{\int\limits_{y\in \OF^\times}W_2^{(c/2)}(\varpi^k y)\chi(y)d^\times y}\\
=&\sum\limits_{k\geq 0}q^{(1/2-\gamma)k}\frac{
q^{-k}}{C_0\overline{C_0'}(q-1)^2q^{c/2-2}}\sum\limits_{c(\chi)=c/2}\sum\limits_{i,j= 0,k}\mu_1(\varpi^{k-2i})\overline{\mu_2(\varpi^{k-2j})}\notag\\
&\int\limits_{u\in \varpi_E^{-c_0}O_E^\times}\theta_1^{-1}(u)\chi_E(u)\psi_E(\varpi^{k-i,i}u)d^\times u\overline{\int\limits_{w\in \varpi_E^{-c_0}O_E^\times}\theta_2^{-1}(w)\chi_E(w)\psi_E(\varpi^{k-j,j}w)d^\times w}.
\notag
\end{align*}
Here $\theta_1=(\mu_1^{-1},\mu_1)$, $\theta_2=(\mu_2^{-1},\mu_2)$. Denote $\varpi_E=(\varpi,\varpi)$.
Denote  by $P_{c/2,k}(W_1,W_2)$ the corresponding summand in $P_{c/2}(W_1,W_2)$ for any fixed $k\geq 0$.
\subsubsection{Case $l=0$}\label{Sec:sameEsplitl0}
If $l=0$, then $C_0=C_0'$,
\begin{align*}
&\int\limits_{u\in \varpi_E^{-c_0}O_E^\times}\theta_2^{-1}(u)\chi_E(u)\psi_E(\varpi^{k-i,i}u)d^\times u\\
=&\theta_2\theta_1^{-1}(\varpi_E)^{c_0}\int\limits_{w\in \varpi_E^{-c_0}O_E^\times}\theta_1^{-1}(w)\chi_E(w)\psi_E(\varpi^{k-i,i}w)d^\times w.
\end{align*} 
There is no cancellation in the summation over $\chi$ for $P_{c/2,k}(\theta_1,\theta_2)$.
  Thus when $k=0$, we have by Corollary \ref{Cor:PrincipalL2},
$$P_{c/2,0}(W_1,W_2)=L^2(W_1,0)=\frac{q-3}{q-1},$$ 
When $k>0$, we use Corollary \ref{Cor:PrincipalL2} and Cauchy--Schwarz inequality to get
$$\left|P_{c/2,k}(W_1,W_2)\right|\leq \frac{2}{q^{1/2k}}q^{-\Re(\gamma) k}.$$
Adding up all $k\geq 0$, we get
$$\left| P_{c/2}(W_1,W_2)-\frac{q-3}{q-1} \right|\leq \frac{2}{q^{1/2+\Re(\gamma)}-1}.$$
In particular when $q$ is large enough, we have
$$P_{c/2}(W_1,W_2)\gg  1.$$

\subsubsection{Case $l>1$.} As $\pi(\mu,\mu^{-1})\simeq \pi(\mu^{-1},\mu)$, we may assume without loss of generality that 
$c(\mu_1\mu_2^{-1})=l$.
\begin{enumerate}
\item One can similarly relate $P_{c/2,0}(W_1,W_2)$ with $P_{c/2}(\theta_1,\theta_2)$ as in Section \ref{Sec:sameEfieldgeneric}, with the difference that the contribution to $P_{c/2,0}(W_1,W_2)$ comes from $\chi$ with $c(\chi\mu_1)=c(\chi\mu_1^{-1})=c(\chi\mu_2)=c(\chi\mu_2^{-1})=c/2$. Since $c/2>1$ by our condition, the method for the case (1) and (3i) in Section \ref{Sec:sameEfieldgeneric} applies here and
$$P_{c/2,0}(W_1,W_2)= 0.$$

\item When $0<k\leq c/2-l$,  we use similar strategy as in Section \ref{Sec:sameEfieldgeneric} with slight modifications. 
 Note that the case $l=c/2$ is excluded here.
By the proof of Corollary \ref{Cor:PrincipalL2}, we can separate the contribution of $\chi$ with $c(\mu_i^{-1}\chi)=c/2-k<c/2$, from that of $\chi$ with $c(\mu_i\chi)<c/2$.
Since these two cases are parallel,  we focus on the case $c(\mu_i^{-1}\chi)<c/2$, and consider
\begin{align}\label{Eq:Pc/2k}
&P_{c/2,k}(\theta_1,\theta_2)\\
:=&\sum\limits_{\chi} \int\limits_{u\in \varpi_E^{-c_0}O_E^\times}\theta_1^{-1}(u)\chi_E(u)\psi_E(\varpi^{0,k}u)d^\times u\overline{\int\limits_{w\in \varpi_E^{-c_0}O_E^\times}\theta_2^{-1}(w)\chi_E(w)\psi_E(\varpi^{0,k}w)d^\times w}. \notag\\
=&\sum\limits_{\chi}\iint\limits_{x\in O_E^\times,w\in \varpi_E^{-c_0}O_E^\times}\theta_1^{-1}(x)\theta_2\theta_1^{-1}(w)\chi_E\lb x\rb \psi_E(\varpi^{0,k}(x-1)w) d^\times wd^\times x.\notag
\end{align}

Here the sum in $\chi$ is for those with $\chi(\varpi)=1$, $c(\chi)=c/2$ and $c(\mu_i^{-1}\chi)=c/2-k$.


For the integral in $w$ to be nonvanishing, we get that $x\in 1+\varpi^{c/2-l, c/2-k-l}O_E^\times$. Then as in Section \ref{Sec:sameEfieldgeneric}, we write $x=1+\varpi^{c/2-l, c/2-k-l}y$ with $y\in O_E^\times$ when $k<c/2-l$, or $y=(y_1,y_2)$ for $y_i\in \OF^\times$, $y_2\not\equiv -1\mod\varpi$ when $k=c/2-l$. We can uniformly describe the domain for $y$ as $y=(y_1,y_2)\in O_E^\times, y_2\not\equiv -\varpi^{c/2-k-l}$.  Then 
\begin{align}\label{Eq:Pc/2ktheta}
P_{c/2,k}(\theta_1,\theta_2)
=&
\frac{\theta_1\theta_2^{-1}(\varpi_E^{c/2-l})}{q^{c-2l-k}} \int\limits_{w\in \varpi_E^{-l}O_E^\times}\theta_2\theta_1^{-1}(w)\psi_E(w)d^\times w \sum\limits_{c(\mu_i^{-1}\chi)=c/2-k}\\
&\int\limits_{y=(y_1,y_2)\in O_E^\times, y_2\not\equiv -\varpi^{c/2-k-l}}\theta_1^{-1}\chi_E(1+\varpi^{c/2-l,c/2-k-l}y)\theta_1\theta_2^{-1}(y) d^\times y. \notag
\end{align}

%
%


\yhn{Unlike case (3) in Section \ref{Sec:sameEfieldgeneric}, we do not need stationary phase analysis in this case to see cancellations. Indeed one can take the sum in $\chi$ first}
and break the sum of $\chi=\chi_0\chi_1$ into a sum over $\chi_0$ up to $\sim_{c/2-k-l+1}$, and a sum over $\chi_1$ with $c(\chi_1)\leq c/2-k-l+1$. The sum in $\chi_1$ first gives 
$$\sum\limits_{\chi_1}\chi_{1,E}(1+\varpi^{c/2-l,c/2-k-l}y)=0 , \forall y\in O_E^\times. $$
 Note that the first component is always trivial as $c/2-l\geq c/2-k-l+1\geq c(\chi_1)$. The cancellation comes from the second component.
Thus when $0<k\leq c/2-l$ we get 
$$P_{c/2,k}(W_1,W_2)= 0.$$
\item If $k>c/2-l $ and $c(\mu_1\mu_2^{-1})=l$, we can not simultaneously have $c(\mu_1\chi)$ and $c(\mu_2\chi)\leq c/2-k$, or $c(\mu_1^{-1}\chi)$ and  $c(\mu_2^{-1}\chi)\leq c/2-k$.  Thus
$$P_{c/2,k}(W_1,W_2)= 0.$$
\end{enumerate}
Adding up all pieces, we get that when $l>1$,
$$P_{c/2}(W_1,W_2)=P_{c/2,0}(W_1,W_2)+2\sum\limits_{k\geq 0}P_{c/2,k}(W_1,W_2)=0.$$
Here the multiple $2$ comes from the dichotomy that either $c(\mu_i^{-1}\chi)=c/2-k$ or $c(\mu_i\chi)=c/2-k$.
\subsubsection{The case $l=1$.} Here we can apply the method in Section \ref{Sec:sameEfieldgeneric} case (2).
\begin{enumerate}
\item If $k<c/2-l$, we keep working with the case $c(\mu_i^{-1}\chi)=c/2-k$ and can still start with the last integral in \eqref{Eq:Pc/2ktheta}. Using Lemma \ref{Lem:DualLiealgForChar}, we have
\begin{align}\label{Eq:5331}
&\int\limits_{y\in O_E^\times}\theta_1^{-1}\chi_E(1+\varpi^{c/2-l,c/2-k-l}y)\theta_1\theta_2^{-1}(y) d^\times y\\
=&
\int\limits_{y_1\in \OF^\times} \mu_1\chi(1+\varpi^{c/2-1}y_1)\mu_1^{-1}\mu_2(y_1)d^\times y_1\int\limits_{y_2\in \OF^\times} \mu_1^{-1}\chi(1+\varpi^{c/2-k-1}y_2)\mu_1\mu_2^{-1}(y_2)d^\times y_2\notag\\
=&
\int\limits_{y_1\in \OF^\times} \psi(b\varpi^{c/2-1}y_1)\mu_1^{-1}\mu_2(y_1)d^\times y_1\int\limits_{y_2\in \OF^\times} \psi(a\varpi^{c/2-k-1}y_2)\mu_1\mu_2^{-1}(y_2)d^\times y_2\notag\\
=&\mu_1^{-1}\mu_2\lb\frac{a}{b}\varpi^{-k}\rb\int\limits_{y_1\in \OF^\times} \psi(\varpi^{-1}y_1)\mu_1^{-1}\mu_2(y_1)d^\times y_1
\int\limits_{y_1\in \OF^\times} \psi(\varpi^{-1}y_2)\mu_1\mu_2^{-1}(y_2)d^\times y_2.\notag
\end{align}
Here $a$ is associated to $\mu_1^{-1}\chi$ with $v(a)=-c/2+k$, and $b$ is associated to $\mu_1\chi$ with $v(b)=-c/2$. The expression depends only on $a\mod \varpi^{-c/2+k+1}$ and $b\mod \varpi^{-c/2+1}$.

If $k>0$, when we take the sum over $\chi$ with $c(\mu_1^{-1}\chi)=c/2-k$, the congruence class for $b$ is not affected, while the congruence class for $a$ runs over $\varpi^{-c/2+k}\OF^\times\mod  \varpi^{-c/2+k+1}$. As $c(\mu_1^{-1}\mu_2)=l=1$, the sum in $\chi$  with $c(\mu_1^{-1}\chi)=c/2-k$ for $\mu_1^{-1}\mu_2\lb\frac{a}{b}\varpi^{-k}\rb$ is thus vanishing, and
$$P_{c/2,k}(W_1,W_2)=0$$
in this case.

If $k=0$, we take the sum in $\chi$ with $c(\chi)=c(\mu_i\chi)=c(\mu_i^{-1}\chi)=c/2$. We break the sum in $\chi=\chi_0\eta$ into a sum of $\chi_0\sim_{c/2-1}$ and a sum of $\eta$ with $c(\eta)\leq c/2-1$. The sum in $\eta$ gives a counting constant as $\eta_E(1+\varpi^{c/2-1,c/2-1}y)=1$ in \eqref{Eq:5331}.
Let $\alpha_0$ be associated to $\chi_0$ while $\alpha_i$ be associated to $\mu_i$ for $i=1,2$, then we have $a\equiv \alpha_0-\alpha_1$, $b\equiv \alpha_0+\alpha_1$. The sum of $\chi_0$ is then equivalent to a sum of $$t=\frac{a}{b}\equiv \frac{\alpha_0-\alpha_1}{\alpha_0+\alpha_1}$$ running through all $\OF^\times\mod\varpi \OF$ except $\pm 1$. As
$$\sum\limits_{t\in \OF^\times}\mu_1^{-1}\mu_2(t)=0,$$
we have
\begin{equation}\label{Eq:Random}
\sum\limits_{\chi}\int\limits_{y\in O_E^\times}\theta_1^{-1}\chi_E(1+\varpi_E^{c/2-l}y)\theta_1\theta_2^{-1}(y) d^\times y=(q-1)q^{c/2-2}(-1-\mu_1^{-1}\mu_2(-1)) \frac{q}{(q-1)^2}.
\end{equation}
Recall that 
\begin{align*}
&P_{c/2,0}(\theta_1,\theta_2)\\
=&
\frac{\theta_1\theta_2^{-1}(\varpi_E^{c/2-l})}{q^{c-2l}} \int\limits_{v_E(w)=-l}\theta_2\theta_1^{-1}(w)\psi_E(w)d^\times w\sum\limits_{\chi}\int\limits_{y\in O_E^\times}\theta_1^{-1}\chi_E(1+\varpi_E^{c/2-l}y)\theta_1\theta_2^{-1}(y) d^\times y .
\end{align*}
Thus by \eqref{Eq:Random} and Lemma \ref{lem:Gaussint},
$$ 
\left|P_{c/2,0}(\theta_1,\theta_2)\right|\leq \frac{2}{(q-1)^3q^{c/2-2}},
$$
$$\left|P_{c/2,0}(W_1,W_2)\right|=\left|\frac{\left|\int\psi\chi\right|^2}{C_0\overline{C_0'}}P_{c/2,0}(\theta_1,\theta_2) \right|\leq \frac{2}{q-1}.$$
\item If $k=c/2-1$, we again focus on the sum in $\chi$ with $c(\mu_i^{-1}\chi)\leq 1 $. Recall that $\theta_i=(\mu_i^{-1},\mu_i)$. Then the integral $$\sum\limits_{c(\mu_i^{-1}\chi)\leq 1}\int\limits_{y=(y_1,y_2)\in O_E^\times, y_2\not\equiv -\varpi^{c/2-k-l}}\theta_1^{-1}\chi_E(1+\varpi^{c/2-l,c/2-k-l}y)\theta_1\theta_2^{-1}(y) d^\times y$$
as in \eqref{Eq:Pc/2ktheta} can be explicitly written as

\begin{align*}
&\sum\limits_{c(\mu_i^{-1}\chi)\leq 1}
\int\limits_{y_1\in \OF^\times} \mu_1\chi(1+\varpi^{c/2-1}y_1)\mu_1^{-1}\mu_2(y_1)d^\times y_1
\int\limits_{y_2\in \OF^\times, y_2\not\equiv -1} \mu_1^{-1}\chi(1+y_2)\mu_1\mu_2^{-1}(y_2)d^\times y_2
\\
=&\mu_1\mu_2^{-1}(b\varpi^{c/2})\int\limits_{y_1\in \OF^\times} \psi(\varpi^{-1}y_1)\mu_1^{-1}\mu_2(y_1)d^\times y_1
\int\limits_{y_2\in \OF^\times, y_2\not\equiv -1} \sum\limits_{c(\mu_i^{-1}\chi)\leq 1}\mu_1^{-1}\chi(1+y_2)\mu_1\mu_2^{-1}(y_2)d^\times y_2 \notag\\
=&0.\notag 
\end{align*}
Here $b\in \varpi^{-c/2}\OF^\times$ is the constant associated to $\mu_1\chi$ and is independent from $\chi$ with $c(\mu_i^{-1}\chi)\leq 1 $.
Thus again 
$$P_{c/2,k}(W_1,W_2)= 0.$$
\item If $k>c/2-l$, similar to $l>1$ case we also have 
$$P_{c/2,k}(W_1,W_2)= 0.$$

Adding up all pieces, we get that when $l=1$,
$$\left|P_{c/2}(W_1,W_2)\right|\leq \frac{2}{q-1}.$$
\end{enumerate}
\subsection{$c(\pi_i)=2$ case}\label{Sec:c=2}
Note that in this case $\pi_i$ is  either  a supercuspidal representation constructed from an inert field extension, or a principal series representation from a split extension. In this case the associated characters satisfy $c(\theta_i)=1$, so we can no longer apply Lemma \ref{Lem:DualLiealgForChar} and p-adic stationary phase analysis for the local integrals.

Consider first the computation of $P_2(W_1,W_2)$, or that of $P_1(W_1,W_2)$ when $E=E'$ and $c(\theta_1\theta_2^{-1})=0$. The discussions in Section \ref{Sec:i=c}, \ref{Sec:samefieldElargei} and  \ref{Sec:sameEsplitl0} still hold, as we didn't use Lemma \ref{Lem:DualLiealgForChar} there.
Thus we always have
$$P_2(W_1,W_2)=1.$$
When $E=E'$ are fields, we have
$$P_1(W_1,W_2)=1.$$
 When $E=E'$ split, we have
$$ \left|P_1(W_1,W_2)-\frac{q-3}{q-1}\right|\leq \frac{2}{q^{1/2+\Re(\gamma)}-1}.
$$
For the remaining computations, we first prove the following lemma:
\begin{lem}\label{Lem:residuecancellation}
Let $\theta$ be a nontrivial charcter on $k_E^\times$ with $\theta|_{k_F^\times}=1$, where $k_E$ is a quadratic field extension over $k_F$. Without loss of generality we assume that $k_E=k_F[\sqrt{D}]$.
Then 
$$\left|\sum\limits_{a\in k_F}\theta(1+a\sqrt{D})\right|=1.$$
\end{lem}
\begin{proof}
Indeed as $\theta$ is nontrivial, we have
$$\sum\limits_{(a,b)\in k_F^2, (a,b)\neq 0}\theta(a+b\sqrt{D})=0.$$
Using that $\theta|_{k_F^\times}=1$, we get that
$$\sum\limits_{(a,b)\in k_E^2, (a,b)\neq 0}\theta(a+b\sqrt{D})=(q-1)\sum\limits_{a\in k_F}\theta(1+a\sqrt{D})+(q-1)\theta(\sqrt{D}).$$
The lemma follows.
\end{proof}
We summarize the remaining computations in the following.
\begin{lem}
Suppose that $p$ is large enough.
Suppose that $\theta_i$ are defined over \'{e}tale quadratic algebras $E$ and $E'$ with $c(\theta_i)=1$, such that either $E\neq E'$, or $E=E'$ and $c(\theta_1\theta_2)=c(\theta_1\theta_2^{-1})=1$. Then we have
\begin{align}\label{Eq:level1cancellation}
&P_1(\theta_1,\theta_2)\\
=&\notag\sum\limits_{c(\chi)\leq 1}\int\limits_{v_E(u)=-1}\theta_1^{-1}(u)\chi_E(u) \psi_E(u)d^\times u \overline{
\int\limits_{v_{E'}(w)=-1}\theta_2^{-1}(w) \chi_{E'}(w)\psi_{E'}(w)d^\times w}\ll \frac{1}{q^2} ,\\
&P_1(W_1,W_2)\ll \frac{1}{q}.\notag
\end{align}

\end{lem}
Note that if we directly apply the bound for each Gauss integrals, we will get
$$P_1(\theta_1,\theta_2)\ll \frac{1}{q},$$
which is also sharp if $E=E'$ and $\theta_1=\theta_2^{\pm 1}$. The lemma claims additional power saving when $\theta_i$ are not related.

Note that as $p$ is large enough, we can freely add or subtract the term $\chi=1$ in the definition of $P_1(\theta_1,\theta_2)$ without affecting the final asymptotic estimate, as
 $$\left|\int\limits_{v_E(u)=-1}\theta_1^{-1}(u) \psi_E(u)d^\times u \right|\ll \frac{1}{q}.$$
\begin{proof}

Here we take a case by case approach. There might be more uniform way to obtain the power saving.

Suppose first that $E=E'$. By a change of variable $u=wx$, we have
\begin{align*}
P_1(\theta_1,\theta_2)&=\sum\limits_{c(\chi)\leq 1}\int\limits_{v_E(w)=-1}\int\limits_{v_E(x)=0}\theta_1^{-1}\theta_2(w) \theta_1^{-1}(x)\chi_E(x)\psi_E((x-1)w) d^\times x d^\times w\\
&=\int\limits_{v_E(w)=-1}\int\limits_{v_E(x)=0}\theta_1^{-1}\theta_2(w) \theta_1^{-1}(x)\sum\limits_{c(\chi)\leq 1}\chi_E(x)\psi_E((x-1)w) d^\times x d^\times w\notag\\
&=(q-1)\int\limits_{v_E(w)=-1}\int\limits_{Nm(x)\equiv 1\mod p}\theta_1^{-1}\theta_2(w) \theta_1^{-1}(x)\psi_E((x-1)w) d^\times x d^\times w.\notag
\end{align*}
Here $q-1$ comes from the number of $\chi$. 
The integral in $w$ is nonvanishing if and only if $x-1\in O_E^\times$, in which case we can apply a change of variable $v=(x-1)w$ and get

\begin{align*}
&P_1(\theta_1,\theta_2)\\=&(q-1)\int\limits_{v_E(v)=-1}\theta_1^{-1}\theta_2(v)\psi_E(v)d^\times v\int\limits_{Nm(x)\equiv 1\mod p, x-1\in O_E^\times} \theta_1\theta_2^{-1}(x-1)\theta_1^{-1}(x) d^\times x.  \notag
\end{align*}
Note that 
$$\int\limits_{v_E(v)=-1}\theta_1^{-1}\theta_2(v)\psi_E(v)d^\times v\ll \frac{1}{q}.
$$
Thus to verify the lemma in this case, it suffices to prove that
$$\int\limits_{Nm(x)\equiv 1\mod p,x-1\in O_E^\times}\theta_1\theta_2^{-1}(x-1)\theta_1^{-1}(x) d^\times x\ll\frac{1}{q^2}.$$
Using Hilbert 90, we write $x=y/\overline{y}$ for $y\in k_F^\times\backslash k_E^\times$. Using that $c(\theta_i)=1$, we have
\begin{align*}
&\int\limits_{Nm(x)\equiv 1\mod p,x-1\in O_E^\times}\theta_1\theta_2^{-1}(x-1)\theta_1^{-1}(x) d^\times x\\=&
\Vol(1+\varpi_E O_E, d^\times x)
\sum\limits_{y\in k_F^\times\backslash k_E^\times, y-\overline{y}\in k_E^\times}\theta_1\theta_2^{-1}(y-\overline{y})\theta_1^{-1}(y)\theta_2(\overline{y}),
\end{align*}
where $$\Vol(1+\varpi_E O_E, d^\times x)\asymp \frac{1}{q^2}.$$
Now if $k_E=k_F[\sqrt{D}]$ is a field,
for the sum in $y$ with $ y\not\equiv \overline{y}$,
we can choose a special set of representatives $y=1+a\sqrt{D}$ or $\sqrt{D}$. $y=\sqrt{D}$ contributes a term of absolute value $1$, while the sum 
\begin{align*}
\sum\limits_{a\in k_F} \theta_1\theta_2^{-1}(y-\overline{y})\theta_1^{-1}(y)\theta_2(\overline{y})&=\sum\limits_{a\in k_F} \theta_1\theta_2^{-1}(2a\sqrt{D})\theta_1^{-1}\theta_2^{-1}(1+a\sqrt{D})\\
&=\theta_1\theta_2^{-1}(\sqrt{D})\sum\limits_{a\in k_F} \theta_1^{-1}\theta_2^{-1}(1+a\sqrt{D})\notag\\
&\ll 1.\notag
\end{align*}
Here in the first two equalities we have used that $\theta_i|_{F^\times}=1$. In the last line we applied Lemma \ref{Lem:residuecancellation}. The result follows in this case.

On the other hand if $k_E=k_F\times k_F$, 
for the sum in $y$ with $y-\overline{y}\in k_E^\times$, we can choose a set of representatives $y=(a,1)$ with $a\in k_F^\times$, $a\neq 1$. Then
\begin{align*}
\sum\limits_{y\in k_F^\times\backslash k_E^\times, y-\overline{y}\in k_E^\times}\theta_1\theta_2^{-1}(y-\overline{y})\theta_1^{-1}(y)\theta_2(\overline{y})&=\theta_1\theta_2^{-1}(1,-1)\sum\limits_{a\in k_F^\times, a\neq 1}\mu_1\mu_2(a)\\
&=-\theta_1\theta_2^{-1}(1,-1).\notag
\end{align*}
Here in the second equality we have used that $c(\mu_1\mu_2)=1$ and thus 
$$\sum\limits_{a\in k_F^\times}\mu_1\mu_2(a)=0.$$
Again the result follows.

Suppose without loss of generality now that $E$ is a field extension while $E'=F\times F$. By writing $w=(y_1,y_2)\in \varpi^{-1}O_E^\times$ in \eqref{Eq:level1cancellation} and sum in $\chi$ first, we get that the sum in $\chi$ is vanishing unless $y_2\equiv \frac{N_{E/F}(u)}{y_1}\mod \varpi \OF$. Thus
\begin{align*}
&P_1(\theta_1,\theta_2)\\=&\int\limits_{v_E(u)=-1}\theta_1^{-1}(u) \psi_E(u) \overline{
\int\limits_{v_{F}(y_1)=-1}\mu_2(y_1)\mu_2^{-1}\lb\frac{N_{E/F}(u)}{y_1}\rb \psi\lb y_1+\frac{N_{E/F}(u)}{y_1}\rb d^\times y_1  }d^\times u.
\end{align*}
Here the number of $\chi$ cancels with the volume of $y_2$ with the above mentioned congruence requirement. Make a change of variable $u=y_1 v$, we get
\begin{align*}
P_1(\theta_1,\theta_2)&=\int\limits_{v_E(v)=0}
\int\limits_{v_{F}(y_1)=-1}\theta_1^{-1}(v) \overline{\mu_2^{-1}\lb N_{E/F}(v)\rb} \psi\lb(Tr_{E/F}(v)-1-N_{E/F}(v))y_1\rb d^\times y_1 d^\times v .\\
&=\int\limits_{v_E(v)=0}
\int\limits_{v_{F}(y_1)=-1}\theta_1^{-1}(v) \overline{\mu_2^{-1}\lb N_{E/F}(v)\rb} \psi\lb-N_{E/F}(v-1)y_1\rb d^\times y_1 d^\times v. \notag
\end{align*}

Here in the first equality we have used that $\theta_1|_{\F^\times}=1$. 
Note that as $v\in O_E^\times$, $v(N_{E/F}(v-1))=2v_E(v-1)$.
We can then integrate in $y_1$ first and break the integral in $v$ according to whether 
$v\equiv 1\mod \varpi O_E$, 
\begin{align*}
P_1(\theta_1,\theta_2)&=\int\limits_{v_E(v)=0, v\equiv 1}\theta_1^{-1}(v) \overline{\mu_{2,E}^{-1}(v)} d^\times v -\frac{1}{q-1}\int\limits_{v_E(v)=0, v\not\equiv 1}
\theta_1^{-1}(v) \overline{\mu_{2,E}^{-1}(v)} d^\times v\\
&=\frac{q}{q-1}\int\limits_{v_E(v)=0, v\equiv 1}\theta_1^{-1}(v) \overline{\mu_{2,E}^{-1}(v)} d^\times v -\frac{1}{q-1}\int\limits_{v_E(v)=0}
\theta_1^{-1}(v) \overline{\mu_{2,E}^{-1}(v)} d^\times v\notag.
\end{align*}
Note that $\Vol(1+\varpi_E O_E,d^\times v)=\frac{1}{q^2-1}$, so the first term can be directly controlled by $\frac{1}{q^2}$. On the other hand $\theta_1^{-1} \overline{\mu_{2,E}^{-1}}$ is a nontrivial character on $O_E^\times$, so the second term is vanishing. Thus we get in the last case
$$P_1(\theta_1,\theta_2)\ll \frac{1}{q^2}.$$
\end{proof}
\subsection{$p=2$ case}\label{Sec:p=2}
In addition to possible non-dihedral supercuspidal representations, there are several other potential issues in this case. For example, \eqref{Eq:5221} may not be true. $c(\theta_1^{-1}\theta_2)$ and $c(\theta_1\theta_2)$ may both be smaller than $c(\theta_i)$ when $E=E'$, though we have following slightly weaker result.
\begin{lem}\label{Lem:levelp=2}
Suppose that $c(\theta_i)$ are large enough, and $l= c(\theta_1^{-1}\theta_2)\leq c(\theta_1\theta_2)$. Then
$$c(\theta_1\theta_2)=c(\theta_i)+O(1).$$
\end{lem}
\begin{proof}
By the assumption $l\leq c(\theta_1\theta_2)$, we only have to consider the case where $l$ is sufficiently smaller than $c(\theta_i)$. Let $\alpha_i$ be associated to $\theta_i$ by Corollary \ref{Cor:DualLiealgForBasechangeChar}. Then by $c(\theta_1^{-1}\theta_2)=l$,
$$v_E(\alpha_2-\alpha_1)\geq -l+c(\psi_E). $$
$\theta_1\theta_2$ is then associated to $\alpha_1+\alpha_2$ with
$$v_E(\alpha_1+\alpha_2)=v_E(\alpha_1-\alpha_2+2\alpha_2)= v_E(\alpha_2)+v_E(2), $$
from which the lemma follows.
\end{proof}
%

For the partial pairing $P_i(W_1,W_2)$, the following weaker result  suffices for our purposes while avoiding many technique issues.
\begin{lem}\label{Lem:Wipairingp=2}
Suppose that $p=2$ and $c(\pi_1)=c(\pi_2)$ is large enough.
Let $W_i$ be the Whittaker functions  associated to newforms in $\pi_i$.
 Then there exists an absolutely bounded positive integer $a$ such that
$$P_i(W_1, W_2)=1$$
for any $i\geq  c(\pi_1\times\pi_2)/2+a$.
\end{lem}
Here we allow $i\geq c(\pi_{1})$, in which case one directly have $P_i(W_1,W_2)=1$ from the newform theory.

The assumption that $c(\pi_1)=c(\pi_2)$ large enough excludes the case where $\pi_i$ could be non-dihedral supercuspidal representations. 
This result can be then proven similarly as in Section \ref{Sec:i=c}, \ref{Sec:samefieldElargei}. Picking sufficiently large but bounded $a$ also avoids the case $i=c/2$ for principal series representations according to Lemma \ref{Lem:levelp=2}.

\section{Bounds for local period integrals}\label{Sec:localbounds}

For $\eta$  a character of $\F^\times$, Recall that an element $\varphi_3
\in \pi_3=\pi(\eta,\eta^{-1})$ satisfies \begin{equation}\label{Equation:LocalInducedElement}
\varphi_3 \left(\zxz{a_1}{n}{0}{a_2} g\right)=\eta(a_1)\eta^{-1}(a_2)\left|\frac{a_1}{a_2}\right|^{1/2}\varphi(g).
\end{equation}

Denote by
\begin{equation}\label{Equation:LocalRS}
I^{\RS}\left(\varphi_1,\varphi_2,\varphi_3\right)=\int\limits_{{Z(\F) N}\backslash\GL_2{(\F)}}W_{\varphi_1}\left(g\right)\overline{W_{\varphi_2}\left(g\right)}\varphi_3\left(g\right)dg
\end{equation}
 the local integral for the Rankin--Selberg integral. Here $W_{\varphi}$ is the Whittaker function associated to $\varphi$ with respect to the fixed additive character $\psi$. We also note that $\overline{W_{\varphi}}=W^-_{\varphi}$ where $W^-_{\varphi}$ is for $\psi^-(x)=\psi(-x)$. 

In general for $\varphi_i\in \pi_i$, denote by
\begin{equation}\label{Equation:LocalT}
I^{\T}\left(\varphi_1,\varphi_2,\varphi_3\right)=\int\limits_{\F^\times\backslash \GL_2{(\F)}}\prod\limits_{i=1}^{3}\Phi_{\varphi_i}\left(g\right)dg
\end{equation}
the local integral for the triple product formula, where $\Phi_{\varphi_i}$ is the matrix coefficient associated to $\varphi_i$.

This section focuses on giving lower and upper bound for $I^\RS$ and $I^T$ with proper choice of $L^2$-normalized test vectors. For conciseness we omit the terms $Q^{\epsilon}$ from the computations.

\subsection{Set up and preparations}
By the setting in Theorem \ref{Theo:conjecturecase}, we assume without loss of generality that $\eta$  and $\pi_3$ are unramified. We shall 
search for test vectors of the following shape: 
$\varphi_i=\varphi_i^\circ$ are normalized newforms for $i=1,2$, and $\varphi_3=\pi_3\lb\zxz{\varpi^{-n}}{}{}{1}\rb \varphi_3^\circ$ for $c/2\leq n\leq c$ where $c=\max\{c(\pi_1), c(\pi_2)\}$.

Suppose that $\eta= |\cdot|^\gamma$ with $|\Re(\gamma)|<7/64$.

The main strategy for the computations follows that of \cite{YH20}, which we briefly recall here. 
As all integrand are invariant by $K_0(\varpi^{c})$ by our setting, we have
\begin{equation*}
I^{\RS}\left(\varphi_1,\varphi_2,\varphi_3\right)=\sum\limits_{i=0}^{c}A_i\int\limits_{a\in \F^\times}W_{\varphi_1}^{(i)}\overline{W_{\varphi_2}^{(i)}}(a) \varphi_3^\circ\left(\zxz{a\varpi^{-n}}{0}{0}{1}\zxz{1}{0}{\varpi^{i-n}}{1}\right)|a|^{-1} d^\times a
\end{equation*}
for the constants $A_i$ as in Lemma \ref{localintcoefficient}.

The Iwasawa decomposition (in the sense of Lemma \ref{Iwasawadecomp}) for $\zxz{a\varpi^{-n}}{0}{0}{1}\zxz{1}{0}{\varpi^{i-n}}{1}$ can be done as follows: When $i\geq n$, {it is} already in the standard form. When $i<n$, we have
\begin{equation}\label{Eq:Iwa1}
 \zxz{a\varpi^{-n}}{0}{0}{1}\zxz{1}{0}{\varpi^{i-n}}{1}=\zxz{a\varpi^{-i}}{a\varpi^{-n}}{0}{\varpi^{i-n}}\zxz{0}{-1}{1}{\varpi^{n-i}}. 
\end{equation}
Thus by the property of $\varphi_3^\circ$ with normalization $\varphi_3^\circ(1)=1$, we get
\begin{align}\label{Eq:RSintdecomp}
&I^{\RS}\left(\varphi_1,\varphi_2,\varphi_3\right)\\
=&\notag\sum\limits_{i=n}^{c}A_iq^{n(1/2+\gamma)}\int\limits_{a\in \F^\times}W_{\varphi_1}^{(i)}\overline{W_{\varphi_2}^{(i)}}(a) |a|^{\gamma-1/2} d^\times a+\sum\limits_{i=0}^{n-1}A_iq^{(2i-n)(1/2+\gamma)}\int\limits_{a\in \F^\times}W_{\varphi_1}^{(i)}\overline{W_{\varphi_2}^{(i)}}(a) |a|^{\gamma-1/2} d^\times a\\
=&\sum\limits_{i=n}^{c}A_iq^{n(1/2+\gamma)}P_i(W_{\varphi_1},W_{\varphi_2})+\sum\limits_{i=0}^{n-1}A_iq^{(2i-n)(1/2+\gamma)}P_i(W_{\varphi_1},W_{\varphi_2}) \notag.
\end{align}

The following lemma relates $P_i(W_{\varphi_1},W_{\varphi_2})$ with $i<c/2$ to those with $i>c/2$  using the Atkin-Lehner operator when $c(\pi_1)=c(\pi_2)$.
\begin{lem}\label{Lem:AtkinLehnerSym}
Suppose that $c=c(\pi_1)=c(\pi_2)$, and $i<c/2$. Then
$$\left|P_i(W_{\varphi_1},W_{\varphi_2})\right|=q^{(c-2i)(\Re(\gamma)-1/2)}\left|P_{c-i}(W_{\varphi_1},W_{\varphi_2})\right|.$$
\end{lem}
\begin{proof}
 Recall that for $\omega_c=\zxz{0}{1}{-\varpi^c}{0}$ which stabilises the congruence subgroup $K_0\left(\varpi^c\right)$, we have by the uniqueness of the newform,
\begin{equation*}
\pi_i\left(\omega_c\right)\varphi_i^\circ=a_i\varphi_i^\circ
\end{equation*}
for $i=1,2$ and some constants $a_i=\pm 1$. 
It is also straightforward to verify that
$$\zxz{1}{}{\varpi^i}{1}\zxz{0}{1}{-\varpi^c}{0}= \zxz{-\varpi^{c-i}}{1}{}{\varpi^i}\zxz{1}{}{\varpi^{c-i}}{1}\zxz{-1}{}{}{1}$$
As a result for $i<c/2$ and newforms $\varphi\in \pi_1$ or $\pi_2$,
$$
W_{\varphi}^{(i)}(a)=a_i W_{\varphi}^{(c-i)}(-a\varpi^{c-2i})\psi(\varpi^{-i}a).
$$
By Lemma \ref{Lem:WiSc}, \ref{Lem:Wiprincipal},
$W_{\varphi}^{(c-i)}$ is supported only at $v(a)=0$, thus
 $W_{\varphi}^{(i)}(a)$ is supported only at $v(a)=2i-c$. Thus
\begin{align*}
P_i(W_{\varphi_1},W_{\varphi_2})&=\int\limits_{a\in \F^\times}W_{\varphi_1}^{(i)}\overline{W_{\varphi_2}^{(i)}}(a) |a|^{\gamma-1/2} d^\times a\\
&=a_1\overline{a_2}q^{(c-2i)(\gamma-1/2)}\int\limits_{a\in \F^\times}W_{\varphi_1}^{(c-i)}\overline{W_{\varphi_2}^{(c-i)}}(a) |a|^{\gamma-1/2} d^\times a
\end{align*}
The lemma now follows by taking absolute values on both sides.
\end{proof}
The following result relates the local triple product integral with the local Rankin--Selberg integral, which is originally due to \cite{MVIHES} and is later extended in other works in for example \cite[Proposition 5.1]{HS19}.
\begin{lem}\label{Lem:T-RS}
Suppose that $\pi_3$
is a parabolically-induced representation, and $\pi_i$ satisfies the bound $\theta{<1/6}$ towards the Ramanujan conjecture. 
Let $\tilde{\pi_i}$ be the contragredient representation of $\pi_i$, $\varphi_i\in \pi_i$ and $\tilde{\varphi_i}\in \tilde{\pi_i}$. Let $(\cdot, \cdot)$ be the natural $\GL_2-$invariant pairing between $\pi_i$ and $\tilde{\pi_i}$. 
Suppose furthermore that $\varphi_3$ belongs to the model $\pi(\eta,\eta^{-1})$ and $\tilde{\varphi_3}$ belongs to the model $\pi(\eta^{-1},\eta)$.
Then
\begin{align}\label{Eq:Triple-RS}
&\int\limits_{Z\backslash \GL_2(\F)}\prod_i (\pi_i(g)\varphi_i,\tilde{\varphi_i})dg\\
=&\zeta_\F(1)\int\limits_{{Z(\F) N}\backslash\GL_2{(\F)}}W_{\varphi_1}\left(g\right)W_{\varphi_2}\left(Jg\right)\varphi_3\left(g\right)dg
\int\limits_{{Z(\F) N}\backslash\GL_2{(\F)}}W_{\tilde{\varphi_1}}\left(Jg\right)W_{\tilde{\varphi_2}}\left(g\right)\tilde{\varphi_3}\left(g\right)dg.\notag
\end{align}
Here $J=\zxz{-1}{0}{0}{1}$ and $W_\varphi(Jg)$ is the Whittaker function associated to $\varphi$ with respect to $\psi^-$.
\end{lem}
In our case, we have $\tilde{\pi_i}\simeq \pi_i$ as the central characters are trivial, and $(\cdot,\cdot)$ is the usual unitary pairing. We can take $\tilde{\varphi_i}=\varphi_i$, but note that $\varphi_3$ and $\tilde{\varphi_3}$ are in different models.

\subsection{Upper bounds}
Recall the notation $\lleq_{v,\epsilon}$ from Section \ref{Section:Inequalityconvention}. 
We need a special case of the upper bounds obtained in \cite{YH20} here.
Note however that \cite[Corollary 3.17]{YH20} contains an error  in the non-tempered case, leading to a weaker upper bound (while the main conclusions there should still hold). This is  corrected  in the following result:

\begin{prop}\label{prop:upboundforRS&T}
Suppose that $\pi_3$ is unramified, and $c=c(\pi_1)=c(\pi_2)$. Denote $c'=c(\pi_2\times\pi_2)$.
Suppose that $\pi_i$ satisfies the bound $\theta<7/64$ towards the Ramanujan conjecture, for $i=1,2,3$. Let $\varphi_i^\circ\in \pi_i$ be normalized newforms. Suppose that $n\geq c/2$, and $ \pi_3= \pi(\eta,\eta^{-1})$ with $\eta=|\cdot|^\gamma$ in the case of local Rankin--Selberg integral.

When $\Re(\gamma)\geq 0$,
\begin{equation}\label{Eq:upperbd1}
\left|I^{\RS}\left(\varphi_1^\circ,\varphi_2^\circ,\pi_3\left(\zxz{\varpi^{-n}}{0}{0}{1}\right)\varphi_3^\circ\right)\right|
\lleq_{v,\epsilon} \frac{1}{q^{\left(1/2-|\Re(\gamma)|-\epsilon\right)n}}.
\end{equation}
When $\Re(\gamma)\leq 0$,
\begin{equation}\label{Eq:upperbd2}
\left|I^{\RS}\left(\varphi_1^\circ,\varphi_2^\circ,\pi_3\left(\zxz{\varpi^{-n}}{0}{0}{1}\right)\varphi_3^\circ\right)\right|
\lleq_{v,\epsilon} \frac{1}{q^{c'/2}}\frac{1}{q^{\left(1/2-|\Re(\gamma)|-\epsilon\right)(n-c')}}.
\end{equation}
For general $\gamma$,
\begin{equation}\label{Eq:upperbd3}
\left|I^{T}\left(\varphi_1^\circ,\varphi_2^\circ,\pi_3\left(\zxz{\varpi^{-n}}{0}{0}{1}\right)\varphi_3^\circ\right)\right|
\lleq_{v,\epsilon} \frac{1}{q^{c'/2}}\frac{1}{q^{\left(1/2-|\Re(\gamma)|-\epsilon\right)(2n-c')}}   .
\end{equation}
\end{prop}
\begin{proof}
Note that \eqref{Eq:upperbd3} follows from \eqref{Eq:upperbd1}, \eqref{Eq:upperbd2} and Lemma \ref{Lem:T-RS}, as $\eta^{-1}=|\cdot|^{-\gamma}$.

On the other hand, \eqref{Eq:upperbd1} and \eqref{Eq:upperbd2} follows from   \eqref{Eq:RSintdecomp}, and that $P_i(W_1,W_2)\ll 1$  for $i\geq c/2$ according to Lemma \ref{Lem:Wipairingspecial} Proposition \ref{Prop:Wipairing}, and Lemma \ref{Lem:AtkinLehnerSym} for $i<c/2$ with a direct computation. When $|\Re(\gamma)|<7/64$, the main contributions come from $i=n$ when $\Re(\gamma)>0$,  $ c/2\leq i\leq n$ when $\Re(\gamma)=0$, and $i=\lceil c/2\rceil=c'/2$ when $\Re(\gamma)<0$.
Note that the conclusions hold even when $n>c$.
\end{proof}

\subsection{Lower bounds}\label{Sec:lowerbound} 
Our main ingredient for the lower bound of period integrals is Proposition \ref{Prop:Wipairing}, thus we assume for now Assumption \ref{Assump:notquadratictwist}. In that case we choose the following  test vectors: 
\begin{choiceoftest}
In all cases, we take $\varphi_i=\varphi_i^\circ$ to be newforms for $i=1,2$. Take $\varphi_3=\pi_3\lb  \zxz{\varpi^{-n}}{}{}{1} \rb \varphi_3^\circ$ with $n=c(\pi_1\times \pi_2)/2$  when $p$ is large enough, and $n= c(\pi_1\times \pi_2)/2+a$ for some non-negative absolutely bounded integer $a$ when $p$ is bounded.
\end{choiceoftest}

\begin{lem}\label{Lem:WpairingCaseunbalanced}
Suppose that $c(\pi_1)<c(\pi_2)=c$ and $\varphi_i\in \pi_i$ are $L^2$-normalized newforms for $i=1,2$. Suppose that either $p$ or $c$ is large enough. Then
$$P_c(W_{\varphi_1},W_{\varphi_2})\asymp 1, \left|P_{c-1}(W_{\varphi_1},W_{\varphi_2})\right|\leq \frac{1}{q-1}\left|P_c(W_{\varphi_1},W_{\varphi_2})\right|$$
with implied constants controlled by $Q^\epsilon$.
\end{lem}
\begin{proof}
It suffices to do  case-by-case computations. For $\varphi_1$, we only need its information $W^{(c(\pi_1))}_{\varphi_1}$ to compute $P_c$ and $P_{c-1}$ as $c(\pi_1)<c$. This  is provided by Lemma \ref{Lem:WiUnramified}, \ref{Lem:Wispecial}, \ref{Lem:nondihedralsupercuspidalW}, and $W^{(c(\pi_1))}_{\varphi_1}=\Char (O_F^\times)$ whenever $c(\pi_1)\geq 2$ (this is true even when $\pi_1$ is a non-dihedral supercuspidal representation). 
For $\varphi_2$, we apply Lemma \ref{Lem:Wispecial} when $c=1$, and Lemma \ref{Lem:WiSc}, Corollary \ref{Cor:WiMellinSc}/Corollary \ref{Cor:PrincipalWiAlt},  \ref{Cor:WiMellinPrin} when $c\geq 2$, as we avoid the non-dihedral supercuspidal representations by the assumption $p$ or $c$ large enough.

Here we only verify the  case where $c(\pi_1)=1$, $c\geq 2$, and left the remaining cases to readers.
In this case, as $W^{(c)}_{\varphi_2}$ is again $\Char (O_F^\times)$,
\begin{align*}
P_2(W_{\varphi_1},W_{\varphi_2})&=\int\limits_{x\in \F^\times}W_{\varphi_1}^{(2)}(x)\overline{W_{\varphi_2}^{(2)}(x)}|x|^{-1/2+\gamma}d^\times x=\int\limits_{x\in O_F^\times}W_{\varphi_1}^{(c(\pi_1))}(x)d^\times x=1
\end{align*}
by Lemma \ref{Lem:Wispecial}. As remarked before the formula in Lemma \ref{Lem:Wispecial} is not $L^2$-normalized but the resulting difference is controlled by $Q^\epsilon$.

On the other hand by Lemma \ref{Lem:Wispecial}, we have
\begin{align*}
P_1(W_{\varphi_1},W_{\varphi_2})&=\sum\limits_{j\geq 0}q^{(1/2-\gamma)j}\int\limits_{v(x)=j}W_{\varphi_1}^{(1)}(x)\overline{W_{\varphi_2}^{(2)}(x)}d^\times x\\
&=\sum\limits_{j\geq 0}\mu(\varpi)^jq^{(-1/2-\gamma)j}\overline{\int\limits_{v(x)=j}W_{\varphi_2}^{(2)}(x)d^\times x}.
\end{align*}
By taking $\chi=1$ in Corollary \ref{Cor:WiMellinSc}/  \ref{Cor:WiMellinPrin}, we have
$$
\int\limits_{v(x)=j}W_{\varphi_2}^{(2)}(x)d^\times x=\begin{cases}
-\frac{1}{q-1}, \text{\ if $j=0$},\\
0, \text{\ otherwise}.
\end{cases}
$$
Thus $P_1(W_{\varphi_1},W_{\varphi_2})=-\frac{1}{q-1}$ in this case.
\end{proof}

\begin{prop}\label{Prop:lowerbdperiod}
Suppose that 
$\pi_1$ $\pi_2$ satisfy Assumption \ref{Assump:notquadratictwist}. Suppose that either $p$ is large enough 
 or $c(\pi_1\times \pi_2)$ is large enough. 
Then there exists a choice of test vectors as specified above, such that $I^{\RS}\left(\varphi_1,\varphi_2,\varphi_3\right)$ and $I^{T}\left(\varphi_1,\varphi_2,\varphi_3\right)$ are non-zero and satisfy the lower bounds
\begin{equation}\label{Eq:RSlowerbound}
\left|I^{\RS}\left(\varphi_1,\varphi_2,\varphi_3\right)\right|\ggeq_{v,\epsilon}  \frac{1}{q^{(1/2-\Re(\gamma))c(\pi_1\times \pi_2)/2}},
\end{equation}
\begin{equation}\label{Eq:triplelowerbound}
\left|I^{T}\left(\varphi_1,\varphi_2,\varphi_3\right)\right|\ggeq_{v,\epsilon} \frac{1}{q^{c(\pi_1\times \pi_2)/2}}.
\end{equation}
\end{prop}
Note that when $p$ and $c(\pi_1\times\pi_2)$ are both bounded, the size of local integral does not affect the asymptotic analysis.

We shall see below that the ambiguity for the choice of $n$ when $p$ is small is due to the fact that the supposed main term in the computation could be smaller than the supposed error terms. But at least one of the choices for $n$ will give the right size of the local integrals.
\begin{proof}
\eqref{Eq:triplelowerbound} follows from Lemma \ref{Lem:T-RS} and applying \eqref{Eq:RSlowerbound} for $\gamma$ and $-\gamma$. In the following we focus on proving \eqref{Eq:RSlowerbound}.

If $c(\pi_1)\neq c(\pi_2)$, we can assume without loss of generality that $c(\pi_1)<c(\pi_2)=c$. In this case $c(\pi_1\times \pi_2)=2c$, and we take $$\varphi_3=\pi_3\lb \zxz{\varpi^{-c}}{}{}{1}\rb\varphi_3^\circ.$$
We  start with \eqref{Eq:RSintdecomp}.
By Lemma \ref{Lem:WiUnramified}, \ref{Lem:Wispecial}, \ref{Lem:WiSc}, \ref{Lem:nondihedralsupercuspidalW} and \ref{Lem:Wiprincipal}, $P_i(W_{\varphi_1},W_{\varphi_2})=0$ for any $i<c-1$ due to the difference of levels when $c\geq 2$, or empty set when $c=1$. Thus when $p$ is large enough or $c$ is large enough, we have by Lemma \ref{Lem:WpairingCaseunbalanced}
\begin{align*}
I^\RS(\varphi_1,\varphi_2,\varphi_3)&\gg_{v,\epsilon} A_cq^{c(1/2+\gamma)}-A_{c-1}q^{(c-2)(1/2+\gamma)}\frac{1}{q-1}\\
=&\frac{1}{(q+1)q^{c-1}}q^{c(1/2+\gamma)}\lb 1- \frac{1}{q^{1+2\gamma}}\rb\gg \frac{1}{q^{(1/2-\Re(\gamma))c}}.
\end{align*}

If $c(\pi_1)=c(\pi_2)$, we again start with \eqref{Eq:RSintdecomp}. There are more cases to consider here.
\begin{enumerate}
\item  Case $c(\pi_i)=1$ with $p$ large enough. Suppose that $\pi_i=\sigma(\mu_i |\cdot|^{1/2}, \mu_i |\cdot|^{-1/2})$ for $c(\mu_i)=0$ and $\mu_i^2=1$.
Using Lemma \ref{Lem:Wispecial}, \ref{Lem:Wipairingspecial} and \ref{Lem:AtkinLehnerSym} we have for $n=1$,
$$I^\RS= A_1 q^{1/2+\gamma} P_1(W_1,W_2) +A_0q^{-1/2-\gamma}P_0(W_1,W_2)\gg q^{-1/2+\Re(\gamma)}.$$
Note that $c(\pi_1\times\pi_2)=2$ in this case. Thus \eqref{Eq:RSlowerbound} is true.
\item Case $c(\pi_1\times\pi_2)/2\geq c/2+1$, $p$ large enough. Using Proposition \ref{Prop:Wipairing} and Lemma \ref{localintcoefficient}, we have for $n= c(\pi_1\times\pi_2)/2$,
\begin{equation*}
I_1:=\left|\sum\limits_{i=n}^{c}A_i q^{n(1/2+\gamma)}P_i(W_{\varphi_1},W_{\varphi_2}) \right|=\frac{q^{-n/2+1}q^{n\Re(\gamma)}}{q+1},
\end{equation*}
as $P_i(W_{\varphi_1},W_{\varphi_2})=1$ for each $i$ in the range. This part is expected to be the main term. On the other hand using Lemma \ref{Lem:AtkinLehnerSym}, 
\begin{align*}
&I_2:=\left|\sum\limits_{i=0}^{c-n}A_i q^{(2i-n)(1/2+\gamma)}P_i(W_{\varphi_1},W_{\varphi_2}) \right|\\
\leq & \frac{q^{c-n+1}}{q+1}q^{-n(1/2+\Re(\gamma))}q^{-c(1/2-\Re(\gamma))}=\frac{q^{c/2-3n/2+1} q^{(c-n)\Re(\gamma)}}{q+1}.
\end{align*}

There will also be contributions coming from $i=n-1$ and $i=c-n+1$ in this case, which are
\begin{align*}
I_3:&=\left|A_{n-1} 	q^{(n-2)(1/2+\gamma)}P_{n-1}(W_{\varphi_1},W_{\varphi_2})+A_{c-n+1} 	q^{(2c-3n+2)(1/2+\gamma)}P_{c-n+1}(W_{\varphi_1},W_{\varphi_2})\right|\\
&\leq \frac{q^{-n/2}q^{(n-2)\Re(\gamma)}}{q+1}+  \frac{q^{c/2-3n/2+1}q^{(c-n)\Re(\gamma)}}{q+1}\notag
\end{align*}
Note that this upper bound is still correct when $c(\pi_1\times\pi_2)/2=c/2+1$.
The remaining $i$'s (if exist) have zero contribution by Proposition \ref{Prop:Wipairing}.
Then we have
\begin{align*}
\left|I^\RS\right|\geq I_1-I_2-I_3=\frac{q^{-n/2}q^{n\Re(\gamma)}}{q+1}\lb q- q^{-2\Re(\gamma)}-2q^{c/2-n+1}q^{(c-2n)\Re(\gamma)} \rb.
\end{align*}

When $n\geq c/2+1$, and $p\geq 5$, one can easily verify that 
$$\left|I^\RS\right|\gg q^{-n/2}q^{n\Re(\gamma)}.$$

On the other hand, if $p=3$ and 
$n\geq \frac{c+3}{2}$, then one can also directly verify that the same bound holds.

\item Case $c$ is odd, $c(\pi_1\times\pi_2)/2=\frac{c+1}{2}>1$ and $p$ large enough. This case happens when both $\pi_i$ are supercuspidal representations associated to characters over ramified field extensions.
 We directly have
for $n=c(\pi_1\times\pi_2)/2$ that 
$$\left|I^\RS\right|\geq I_1-I_2
\geq \frac{q^{-n/2}q^{n\Re(\gamma)}}{q+1}\lb q- q^{1/2-\Re(\gamma)}\rb\gg q^{-n/2+n\Re(\gamma)}$$
when $2\nmid p$. Here $I_1$, $I_2$ are defined similarly as in the previous case.


\item Case $c$ is even, $c(\pi_1\times\pi_2)/2=c/2$ and $p$ large enough. If $\pi_i$ are supercuspidal representations, we have by Proposition \ref{Prop:Wipairing} that $P_i(W_{\varphi_1},W_{\varphi_2})=1$ for $i\geq c/2$. So for $n=c/2$, we have
\begin{equation*}
I_1:=\sum\limits_{i=c/2+1}^{c}A_i q^{n(1/2+\Re(\gamma))}P_i(W_{\varphi_1},W_{\varphi_2}) =\frac{q^{-c/4}q^{c\Re(\gamma)/2}}{q+1},
\end{equation*}
\begin{equation*}
I_2:=\left|\sum\limits_{i=0}^{c/2-1}A_i q^{(2i-n)(1/2+\gamma)}P_i(W_{\varphi_1},W_{\varphi_2}) \right|\leq \frac{q^{-c/4} q^{c\Re(\gamma)/2}}{q+1},
\end{equation*}
\begin{equation*}
I_3:= A_{c/2}q^{c/2(1/2+\Re(\gamma))}P_{c/2}(W_{\varphi_1},W_{\varphi_2})=\frac{(q-1)q^{-c/4}q^{c\Re(\gamma)/2}}{q+1}.
\end{equation*}
Thus for any $2\nmid p$, $I_3$ is the main term and we have
\begin{equation*}
\left|I^\RS\right|\geq I_1+I_3-I_2\gg q^{-c/4}q^{c\Re(\gamma)/2}.
\end{equation*}
When $\pi_i$ are principal series, $I_1,$ $I_2$ can be controlled similarly, while $I_3$ is slightly more complicated.
By Proposition \ref{Prop:Wipairing}, we have
\begin{equation*}
\left|I_3\right|\geq \frac{(q-1)q^{-c/4}q^{c\Re(\gamma)/2}}{q+1}\lb \frac{q-3}{q-1}-\frac{2}{q^{1/2+\Re(\gamma)}-1}\rb.
\end{equation*}
One can check that when, for example, $p\geq 23$, we have
$$|I_3|\gg q^{-c/4}q^{c\Re(\gamma)/2}, \ \left|I^\RS\right|\gg I_1+I_3-I_2\gg q^{-c/4}q^{c\Re(\gamma)/2}.$$

\item 
Consider now the case $p$ is bounded ($p\leq 23$ for example, including the case $p=2$), and $c$ is large enough.
Denote $n_0=c(\pi_1\times\pi_2)/2+a$, where $a$ is as in Lemma \ref{Lem:Wipairingp=2} when $p=2$, or $a=1$ when $p>2$. Thus $P_i(W_1,W_2)=1$ whenever $i\geq n_0$ by Lemma  \ref{Lem:Wipairingp=2} or Proposition \ref{Prop:Wipairing}.

Let $c'=\max\{c, n_0+1\}$. Using Lemma \ref{localintcoefficient} for $c'$, we compare the calculation of \eqref{Eq:RSintdecomp} for $n=n_0$ and $n=n_0+1$. 
\begin{align*}
I_{n_0}:&=I^\RS\lb \varphi_1,\varphi_2,\pi_3\lb\zxz{\varpi^{-n_0}}{}{}{1}\rb \varphi_3^\circ\rb
\\
&=\sum\limits_{i=n_0}^{c'}A_iq^{n_0(1/2+\gamma)}P_i(W_{\varphi_1},W_{\varphi_2})+\sum\limits_{i=0}^{n_0-1}A_iq^{(2i-n_0)(1/2+\gamma)}P_i(W_{\varphi_1},W_{\varphi_2})\\
&=\sum\limits_{i=n_0+1}^{c'}A_iq^{n_0(1/2+\gamma)}P_i(W_{\varphi_1},W_{\varphi_2})+\sum\limits_{i=0}^{n_0}A_iq^{(2i-n_0)(1/2+\gamma)}P_i(W_{\varphi_1},W_{\varphi_2}),\notag
\end{align*}
\begin{align*}
I_{n_0+1}:&=I^\RS\lb \varphi_1,\varphi_2,\pi_3\lb\zxz{\varpi^{-n_0-1}}{}{}{1}\rb \varphi_3^\circ\rb\\
&=\sum\limits_{i=n_0+1}^{c'}A_iq^{(n_0+1)(1/2+\gamma)}P_i(W_{\varphi_1},W_{\varphi_2})+\sum\limits_{i=0}^{n_0}A_iq^{(2i-n_0-1)(1/2+\gamma)}P_i(W_{\varphi_1},W_{\varphi_2}).
\end{align*}
We can do a linear combination of two equations to cancel the terms with $i\leq n_0$. In particular we have
\begin{equation*}
I_{n_0}-q^{1/2+\gamma} I_{n_0+1}=\sum\limits_{i=n_0+1}^{c'}A_iq^{n_0(1/2+\gamma)}(1-q^{1+\gamma})P_i(W_{\varphi_1},W_{\varphi_2}).
\end{equation*}

Using that $P_i(W_{\varphi_1},W_{\varphi_2})=1$ for $i\geq n_0$, and that $q$ is bounded in this case, we get that 
\begin{equation*}
\left|I_{n_0}-q^{1/2+\gamma} I_{n_0+1}\right|\gg q^{n_0(-1/2+\Re(\gamma))}.
\end{equation*}
Then we either have 
\begin{equation*}
\left|I_{n_0}\right|\gg q^{n_0(-1/2+\Re(\gamma))},
\end{equation*}
or
\begin{equation*}
\left|I_{n_0+1}\right|\gg q^{n_0(-1/2+\Re(\gamma))}.
\end{equation*}
\end{enumerate}

\end{proof}

\subsection{Proof of Theorem \ref{Theo:conjecturecase}}
We can now prove Theorem \ref{Theo:conjecturecase}, that is, verify Conjecture \ref{Conj:localresults} under the conditions that $\pi_i$ have trivial central characters and bounded Archimedean components, and $M_{\fin}=1$ at finite places. 
\yhn{Recall that $I=I^T$ or $\left|I ^\RS\right|^2$ depending on whether  $\varphi_3$ or $\psi$ comes from an Eisenstein series.}

Note that the formulations in Conjecture \ref{Conj:localresults} are symmetric in index $i=2,3$ and are  local, we may assume without loss of generality that $C(\pi_{3})=1$. 

We first address the issue about potential conflicts with Assumption \ref{Assump:notquadratictwist}. Suppose now without loss of generality that $c(\pi_1)=c(\pi_2)=2$, $\pi_1=\pi_1'\otimes \chi$ for a quadratic character $\chi$ and an unramified representation $\pi_1'$, not satisfying Assumption \ref{Assump:notquadratictwist}. Let $\pi_2'=\pi_2\otimes \chi$. Then $\pi_i'$ still have trivial central characters, and 
$$C(\pi_i\times \pi_j)=C(\pi_i'\times \pi_j')$$
for $i,j=1,2$. One can now reduce the test vector problem to that of the triple $\pi_1', \pi_2',\pi_3$, which satisfies Assumption \ref{Assump:notquadratictwist} now.

With Assumption \ref{Assump:notquadratictwist} we take now $\varphi_{i}=\pi_i\left(\zxz{\varpi^n}{}{}{1}\right)\varphi_{i }^\circ$ for $i=1,2$, and $\varphi_{3 }=\varphi_{3 }^\circ$, where $n$ is as in 
Proposition \ref{Prop:lowerbdperiod}.
Thus $\varphi_3$ satisfies item (0) of Conjecture \ref{Conj:localresults}.


Up to a uniform translates of test vectors which does not change the period integrals, the required lower bound in item (1) of Conjecture \ref{Conj:localresults} is provided by Proposition \ref{Prop:lowerbdperiod}. 
 Note that in the case of Rankin--Selberg L-function, we can take $\Re(\gamma)=0$ for \eqref{Eq:RSlowerbound} when  we assume $\pi_3$ to be tempered.

Now in the spectral expansion in Lemma \ref{Lem:shortgenericlength}, $\pi' $ must be unramified and $\psi $ must be spherical for $v\in 
\ram(\vPhi_{23}) $ as $M_\fin =1$. The required upper bound in item (2) of Conjecture \ref{Conj:localresults} follows from Proposition \ref{prop:upboundforRS&T}.
 Indeed, in the triple product case and $\psi =\psi ^\circ$, we have


\yhn{
\begin{align*}
I (\varphi_{2 },\varphi_{2 },\psi)&= I^T \left(\pi_{2 }\left(\zxz{\varpi^n}{}{}{1}\right)\varphi_{2 }^\circ,\pi_{2 }\left(\zxz{\varpi^n}{}{}{1}\right)\varphi_{2 }^\circ , \psi ^\circ \right)\\
&= I^T \left(\varphi_{2 }^\circ,\varphi_{2 }^\circ , \pi' \left(\zxz{\varpi^{-n}}{}{}{1}\right)\psi ^\circ \right)\notag\\
&\leq_\epsilon \frac{Q ^\epsilon}{q^{\theta c(\pi_{2 }\times\pi_{2 })}} \frac{1}{Q ^{1/2(1/2-\theta)}}\leq \frac{Q ^\epsilon}{\max\limits_{i=2,3}\{C (\pi_i\times\pi_i)\}^{1/2}}\frac{1}{P ^{1/2-\theta}}. \notag
\end{align*}

}

Here in the third line
 we have applied \eqref{Eq:upperbd3} and that $q^{2n}\asymp_{\epsilon} Q ^{1/2+\epsilon}$ for the choice of $n$ above.

The Rankin--Selberg case can be verified similarly using \eqref{Eq:upperbd1} from Proposition \ref{prop:upboundforRS&T} and taking $\Re(\gamma)=0$.

\end{document}